\documentclass[a4paper,oneside,10pt]{article}%
\usepackage{amsmath}
\usepackage{amsfonts}
\usepackage{amssymb}
\usepackage{graphicx}
\usepackage{color}
\usepackage[square,numbers,sort&compress]{natbib}%
\providecommand{\U}[1]{\protect\rule{.1in}{.1in}}

\newtheorem{theorem}{Theorem}[section]

\newtheorem{definition}[theorem]{Definition}
\newtheorem{assumption}[theorem]{Assumption}

\newtheorem{lemma}[theorem]{Lemma}

\newtheorem{proposition}[theorem]{Proposition}
\newtheorem{remark}[theorem]{Remark}

\newenvironment{proof}[1][Proof]{\noindent \textbf{#1.} }{\  \rule{0.5em}{0.5em}}
\numberwithin{equation}{section}

\begin{document}

\title{The existence and uniqueness of viscosity solution to a kind of
Hamilton-Jacobi-Bellman equations}
\author{Mingshang Hu\thanks{Zhongtai Securities Institute for Financial Studies,
Shandong University, Jinan, Shandong 250100, PR China. humingshang@sdu.edu.cn.
Research supported by NSF (No. 11671231) and Young Scholars Program of
Shandong University (No. 2016WLJH10). }
\and Shaolin Ji\thanks{Zhongtai Securities Institute for Financial Studies,
Shandong University, Jinan, Shandong 250100, PR China. jsl@sdu.edu.cn
(Corresponding author). Research supported by NSF (No. 11571203).}
\and Xiaole Xue\thanks{Zhongtai Securities Institute for Financial Studies,
Shandong University, Jinan 250100, PR China. xiaolexue1989@gmail.com,
xuexiaole.good@163.com. Research supported by NSF (Nos. 11701214 and 11801315)
and Natural Science Foundation of Shandong Province(ZR2018QA001). } }
\maketitle

\textbf{Abstract}. In this paper, we study the existence and uniqueness of
viscosity solutions to a kind of Hamilton-Jacobi-Bellman (HJB) equations
combined with algebra equations. This HJB equation is related to a stochastic
optimal control problem for which the state equation is described by a fully
coupled forward-backward stochastic differential equation. By extending Peng's
backward semigroup approach to this problem, we obtain the dynamic programming
principle and show that the value function is a viscosity solution to this HJB
equation. As for the proof of the uniqueness of viscosity solution, the
analysis method in Barles, Buckdahn and Pardoux \cite{Baeles-BP} usually does
not work for this fully coupled case. With the help of the uniqueness of the
solution to FBSDEs, we propose a novel probabilistic approach to study the
uniqueness of the solution to this HJB equation. We obtain that the value
function is the minimum viscosity solution to this HJB equation. Especially,
when the coefficients are independent of the control variable or the solution
is smooth, the value function is the unique viscosity solution.



\addcontentsline{toc}{section}{\hspace*{1.8em}Abstract}

\section{Introduction}

Dynamic programming principle (DPP), originated by Bellman in 1950s, is a
powerful tool to solve optimal control problems. Since then DPP and related
Hamilton-Jacobi-Bellman (HJB) equations have been intensively studied by a lot
of researchers for various kinds of stochastic optimal control problems
(see\cite{Bay1,Bay2,Buc1,Buc2,Buckdahn-Li,Li-W,
Ma-Y-2,Peng-1992,Tang1,Yong-Zhou,Zhou-1990-2,Zhou-1991} and the references therein).


In this paper, we study the existence and uniqueness of the viscosity solution
to the following HJB equation,
\begin{equation}
\left\{
\begin{array}
[c]{l}%
\partial_{t}W(t,x)+\inf\limits_{u\in U}H(t,x,W(t,x),DW(t,x),D^{2}W\left(
t,x\right)  ,u)=0,\\
W(T,x)=\phi(x),
\end{array}
\right.  \label{intro-hjb}%
\end{equation}
where
\begin{equation}%
\begin{array}
[c]{l}%
H(t,x,v,p,A,u)\\
=\frac{1}{2}\mathrm{tr}[\sigma\sigma^{\intercal}%
(t,x,v,V(t,x,v,p,u),u)A]+p^{\intercal}b(t,x,v,V(t,x,v,p,u),u)\\
\ \ +g(t,x,v,V(t,x,v,p,u),u),\\
V(t,x,v,p,u)=p^{\intercal}\sigma(t,x,v,V(t,x,v,p,u),u),\\
(t,x,v,p,A,u)\in\lbrack0,T]\times\mathbb{R}^{n}\times\mathbb{R}\times
\mathbb{R}^{n}\times\mathbb{S}^{n}\times U.
\end{array}
\label{intro-h}%
\end{equation}
To solve (\ref{intro-hjb}) we need to find the solution $V$ of the algebra
equation in (\ref{intro-h}) first. Once $V$ is computed, it is directly
plugged to $H(\cdot)$ so that we obtain the true generator of (\ref{intro-hjb}).

This kind of problem has the following stochastic optimal control
interpretation. The controlled system is described by the following fully
coupled forward-backward stochastic differential equation (FBSDE):
\begin{equation}
\left\{
\begin{array}
[c]{rl}%
dX_{s}^{t,x;u}= & b(s,X_{s}^{t,x;u},Y_{s}^{t,x;u},Z_{s}^{t,x;u},u_{s}%
)ds+\sigma(s,X_{s}^{t,x;u},Y_{s}^{t,x;u},Z_{s}^{t,x;u},u_{s})dB_{s},\\
dY_{s}^{t,x;u}= & -g(s,X_{s}^{t,x;u},Y_{s}^{t,x;u},Z_{s}^{t,x;u}%
,u_{s})ds+Z_{s}^{t,x;u}dB_{s},\\
X_{t}^{t,x;u}= & x,\ Y_{T}^{t,x;u}=\phi(X_{T}^{t,x;u}),\;s\in\lbrack t,T],
\end{array}
\right.  \label{intro-fbsde}%
\end{equation}
where $B=(B_{s})_{s\in\lbrack t,T]}$ is a standard $d$-dimensional Brownian
motion and $u\in\mathcal{U}[t,T]$ is an admissible control. The cost
functional is defined by the solution to the backward stochastic differential
equation (BSDE) at time $t$ in (\ref{intro-fbsde}) and the value function of
our control problem is
\begin{equation}
W(t,x)=\underset{u\in\mathcal{U}[t,T]}{ess\inf}Y_{t}^{t,x;u}.
\label{intro-value}%
\end{equation}
The fully coupled forward-backward stochastic control problem
$(\ref{intro-fbsde})-(\ref{intro-value})$ has its own independent interests in
mathematical finance such as the stochastic differential utility,
leader-follower stochastic differential games, principal-agent problem, and so forth.

It is worth pointing out that the extra algebraic equation in (1.2) stems from
that the solution procedure of the FBSDE depends on a relationship between the
solution $Z$ of the backward SDE and the diffusion coefficients of the forward
SDE. This fact is first revealed in the famous ``Four Step Scheme" for solving
an FBSDE (see \cite{MPY,Ma-Y}). Essentially, this algebraic equation is
exactly the First Step of the ``Four Step Scheme" (see \cite{MPY,Ma-Y}).

It is well-known that the above value function may not necessarily smooth
enough as we want, that is the reason why the researchers introduce the notion
of viscosity solution (see \cite{Crandall-lecture,Lions2} and references
therein). When the coefficients $b$ and $\sigma$ of (\ref{intro-fbsde}) are
independent of the variables $y$ and $z$, Peng \cite{Peng-1992, Peng-lecture}
first obtained that the above defined $W$ is a viscosity solution to
(\ref{intro-hjb}). For this case, the uniqueness of the viscosity solution to
(\ref{intro-hjb}) can be obtained by applying the method in Barles, Buckdahn
and Pardoux \cite{Baeles-BP} (see Theorem 5.3 in \cite{Buckdahn-Li} for details).

When $b$ and $\sigma$ depend on $y$ and $z$ in (\ref{intro-fbsde}), the
control system (\ref{intro-fbsde}) becomes a fully coupled FBSDE and the
corresponding HJB equation (\ref{intro-hjb}) becomes a fully nonlinear
parabolic partial differential equation (PDE) combined with an algebra
equation which leads to the solvability and uniqueness of (\ref{intro-hjb})
being extraordinary difficult. Note that when (\ref{intro-fbsde}) is
independent of the control variable $u$, the HJB equation (\ref{intro-hjb})
degenerates to a semilinear parabolic PDE. In fact, even for this extreme
case, the well-posedness of (\ref{intro-hjb}) is still an open problem which
is proposed by Peng \cite{Peng99}. Recently, Li and Wei \cite{Li-W}, Li
\cite{Li jun} proved that $W$ is a viscosity solution to (\ref{intro-hjb})
under the monotonicity conditions for $b$, $\sigma$, $g$ and $\phi$. As for
the uniqueness of the viscosity solution, there are only a few results for
some special cases. For instance, when $b$, $\sigma$, $g$ are independent of
control variable $u$, and $\sigma$ is independent of $y$ and $z$, by applying
the method in Barles, Buckdahn and Pardoux \cite{Baeles-BP}, Wu and Yu
\cite{Wu-Y} proved that $W$ is the unique viscosity solution to
(\ref{intro-hjb}) under the monotonicity conditions for $b$, $\sigma$, $g$ and
$\phi$ in the space of all continuous functions which are Lipschitz continuous
in $x$.

So the existence and uniqueness of the viscosity solution to (\ref{intro-hjb})
is an interesting and challenging problem. In this paper, we apply the
probabilistic approach to deal with this problem. In more details, with the
help of the value function of the above stochastic optimal control problem and
the existence and uniqueness of the solution to the controlled system
(\ref{intro-fbsde}), we attacks this difficult problem, especially the
uniqueness part. Before we focus on the uniqueness results, we need to deal
with the following problems.

The first problem is the well-posedness of the fully coupled forward-backward
controlled system (\ref{intro-fbsde}). There are many literatures on the
well-posedness of fully coupled FBSDEs. When the coefficients of a fully
coupled FBSDE are deterministic and the diffusion coefficient of the forward
equation is nondegenerate, Ma, Protter and Yong \cite{MPY} proposed the
four-step scheme approach. Under some monotonicity conditions, Hu and Peng
\cite{Hu-Peng95} first obtained an existence and uniqueness result which was
generalized by Peng and Wu \cite{PW}. Yong \cite{Yong1997, Yong-2010B}
developed this approach and called it the method of continuation. The fixed
point approach is due to Antonelli \cite{Antonelli}, Pardoux and Tang
\cite{Pardoux-Tang}. The readers may refer to Ma and Yong \cite{Ma-Y},
Cvitani\'{c} and Zhang \cite{Cvi-Zhang}, Ma, Wu, Zhang and Zhang
\cite{Ma-WZZ}, Yong and Zhou \cite{Yong-Zhou} for the FBSDE theory. In this
paper, we adopt the fixed point approach which is given in Hu, Ji and Xue
\cite{Hu-JX}.

The second problem is the existence of the viscosity solution to
(\ref{intro-hjb}). We obtain that value function $W$ defined in
(\ref{intro-value}) is a viscosity solution to the HJB equation
(\ref{intro-hjb}) under the assumption in Hu, Ji and Xue \cite{Hu-JX}. As
pointed out in Remark \ref{re-mon}, our proofs still hold under monotonicity
conditions. Comparing with Li and Wei \cite{Li-W}, both papers
develop Peng's stochastic backward semigroup approach in \cite{Peng-1992,
Peng-lecture}. But our main proofs are different from the ones in Li and Wei
\cite{Li-W}. To establish the DPP, the stochastic backward semigroup is
defined by a fully coupled FBSDE.  Unlike the decoupled case,
changing $(Y,Z)$ does affect $X$ in a fully coupled FBSDE. Thus, the approach
of establishing the DPP for the decoupled forward-backward controlled system
does not work any more. By constructing two new auxiliary FBSDEs (see
(\ref{new-new-3422}) and (\ref{eq-xy-til-t})), we introduce a new approach to
prove the DPP for the fully coupled forward-backward controlled system. We
also simplify some proofs as follows. To prove the continuity property of the
value function $W(t,x)$ in $t$, we build a FBSDE (see (\ref{FBSDE-continuous}%
)) which makes the proof easier. For the combined algebra equation, we
construct a simple contraction mapping to prove the existence and uniqueness,
and obtain some properties of this algebra equation.

Then, we study the uniqueness of viscosity solution to the HJB equation
(\ref{intro-hjb}) in four cases. The first case is that $\sigma$ is
independent of $y$ and $z$. By using the method in \cite{Baeles-BP}, we prove
the uniqueness of viscosity solution to (\ref{intro-hjb}) in the space of all
continuous functions which are Lipschitz continuous in $x$. But, when $\sigma$
depends on $y$ and $z$, the method in \cite{Baeles-BP} does not work as
pointed out in Remark \ref{re-appen}. The second case is that $\sigma$ is
independent of $z$. Different from the analysis method in \cite{Baeles-BP},
for this case, we propose a novel probabilistic approach to prove the
uniqueness. In more details, we construct a new fully coupled forward-backward
stochastic control system (\ref{new-lx-18}) in which $\sigma$ only depends on
the variables $t$, $x$ and $u$. By the uniqueness result in the first case,
the value function for this new control system is the unique viscosity
solution to the HJB equation (\ref{new-lx-17}). Thanks to Proposition
\ref{pr-um}, we prove that $W$ defined in (\ref{intro-value}) is the minimum
viscosity solution to the HJB equation (\ref{intro-hjb}) in the space of all
continuous functions which are Lipschitz continuous in $x$. It is worthing to
point out that when $b$, $\sigma$, $g$ are independent of the control variable
$u$, $W$ is just the unique viscosity solution. The third case is that
$\sigma$ depends on $y$ and $z$. We construct a new decoupled forward-backward
stochastic control system (\ref{new-lx-22}). Following the similar approach in
the second case, we prove that $W$ defined in (\ref{intro-value}) is the
minimum viscosity solution to the HJB equation (\ref{intro-hjb}) in a smaller
space (see Theorem \ref{th-um2}). Especially, when $b$, $\sigma$, $g$ are
independent of the control variable $u$, $W$ is also the unique viscosity
solution. The fourth case is that the solution to HJB equation
(\ref{intro-hjb}) is smooth. We construct a new BSDE (\ref{new-asd-11}). With
the help of the comparison theorem for BSDEs, we prove that the solution is
just the value function defined in (\ref{intro-value}).

{ In order to study the well-posedness of FBSDEs, Ma, Wu, Zhang and
Zhang\cite{Ma-WZZ} proposed an important concept \textquotedblleft decoupling
field". When $b$, $\sigma$, $g$ are independent of the control variable $u$,
in Theorem \ref{th-uni-pde} or \ref{uni-pde}, the value function $W$ is a
decoupling field. For another viscosity solution $\tilde{W}$, we have verified
that $\tilde{W}=W$ by a probabilistic approach. From the perspective of
decoupling field, we have proved that $\tilde{W}$ is just a regular decoupling
field which leads to the uniqueness naturally. This reflects that the
decoupling field is a very important concept in the theory of FBSDEs.}

{Our paper is organized as follows. In section 2, we formulate our problem and
a related stochastic optimal control problem. In section 3, we prove that the
value function of the related stochastic control problem is a viscosity
solution to the HJB equation by establishing the DPP and the properties of the
value function. The uniqueness results are obtained in section 4.}

\section{The problem formulation}

Denote by $\mathbb{R}^{n}$ the $n$-dimensional real Euclidean space,
$\mathbb{R}^{k\times n}$ the set of $k\times n$ real matrices and
$\mathbb{S}^{n}$ the set of $n\times n$ symmetric matrix. Let $U$ be a
nonempty and compact subset in $\mathbb{R}^{k}$. Let $\langle\cdot
,\cdot\rangle$ (resp. $\Vert\cdot\Vert$) denote the usual scalar product
(resp. usual norm) of $\mathbb{R}^{n}$ and $\mathbb{R}^{k\times n}$. The
scalar product (resp. norm) of $M=(m_{ij})$, $N=(n_{ij})\in\mathbb{R}^{k\times
n}$ is denoted by $\langle M,N\rangle=\mathrm{tr}\{MN^{\intercal}\}$ (resp.
$\Vert M\Vert=\sqrt{MM^{\intercal}}$), where the superscript $^{\intercal}$
denotes the transpose of vectors or matrices.

We will study the existence and uniqueness of viscosity solution to the
following { HJB equation combined with an algebra equation
\begin{equation}
\label{pre-hjb}%
\begin{array}
[c]{l}%
\partial_{t}W(t,x)+\inf\limits_{u\in U}H(t,x,W(t,x),DW(t,x),D^{2}W\left(
t,x\right)  ,u)=0,\ W(T,x)=\phi(x).
\end{array}
\end{equation}
Here }$H\left(  \cdot\right)  $ is defined as follows{
\begin{equation}%
\begin{array}
[c]{l}%
H(t,x,v,p,A,u)\\
=\frac{1}{2}\mathrm{tr}[\sigma\sigma^{\intercal}%
(t,x,v,V(t,x,v,p,u),u)A]+p^{\intercal}b(t,x,v,V(t,x,v,p,u),u)\\
\ \ +g(t,x,v,V(t,x,v,p,u),u),\\
(t,x,v,p,A,u)\in\lbrack0,T]\times\mathbb{R}^{n}\times\mathbb{R}\times
\mathbb{R}^{n}\times\mathbb{S}^{n}\times U,
\end{array}
\label{pre-h}%
\end{equation}
}$V(t,x,v,p,u)$ is the solution to the following algebra equation
\[
V(t,x,v,p,u)=p^{\intercal}\sigma(t,x,v,V(t,x,v,p,u),u),
\]
where {%
\[
b:[t,T]\times\mathbb{R}^{n}\times\mathbb{R}\times\mathbb{R}^{1\times d}\times
U\rightarrow\mathbb{R}^{n}, \ \sigma:[t,T]\times\mathbb{R}^{n}\times
\mathbb{R}\times\mathbb{R}^{1\times d}\times U\rightarrow\mathbb{R}^{n\times
d},
\]%
\[
g:[t,T]\times\mathbb{R}^{n}\times\mathbb{R}\times\mathbb{R}^{1\times d}\times
U\rightarrow\mathbb{R}, \ \phi:\mathbb{R}^{n}\rightarrow\mathbb{R}.
\]
We impose the following assumption on these functions.}

\begin{assumption}
\label{assum-1} (i) $b,\sigma,g,\phi$ are continuous with respect to
$s,x,y,z,u$, and there exist constants $L_{i}>0$, $i=1,2,3$ such that%
\[
|b(s,x_{1},y_{1},z_{1},u)-b(s,x_{2},y_{2},z_{2},u)|\leq L_{1}|x_{1}%
-x_{2}|+L_{2}(|y_{1}-y_{2}|+|z_{1}-z_{2}|),
\]%
\[
|\sigma(s,x_{1},y_{1},z_{1},u)-\sigma(s,x_{2},y_{2},z_{2},u)|\leq L_{1}%
|x_{1}-x_{2}|+L_{2}|y_{1}-y_{2}|+L_{3}|z_{1}-z_{2}|,
\]%
\[
|g(s,x_{1},y_{1},z_{1},u)-g(s,x_{2},y_{2},z_{2},u)|\leq L_{1}(|x_{1}%
-x_{2}|+|y_{1}-y_{2}|+|z_{1}-z_{2}|),
\]%
\[
|\phi(x_{1})-\phi(x_{2})|\leq L_{1}|x_{1}-x_{2}|,
\]
for all $s\in\lbrack0,T],x_{i},y_{i},z_{i}\in\mathbb{R}^{n}\times
\mathbb{R}\times\mathbb{R}^{1\times d},$ $i=1,2$, $u\in U$.

 (ii)\ $\bar{\Lambda}:=8C_{2}(L)(1+T^{2})c_{1}^{2}<1$, where
$C_{2}(\cdot)$ is defined in Lemma \ref{sde-bsde} and Remark \ref{re-app-2};
$L=\max\{L_{1}, \newline\sqrt{C_{2}(L_{1})} (1-\sqrt{\Lambda})^{-1} \}$,
$\Lambda=8C_{2}(L_{1})(1+T^{2})c_{1}^{2}$, $c_{1}=\max\{L_{2,}L_{3}\}$ .
\end{assumption}

\begin{remark}
\label{linear growth}Since $U$ is compact, from the above assumption (i) we
obtain that\ $|\psi(s,x,y,z,u)|\leq L(1+|x|+|y|+|z|), $ where $L>0$ is a
constant and $\psi=b,$ $\sigma,$ $g$ and $\phi$.
\end{remark}

\begin{remark}
Since $C_{2}(\cdot)$ is increasing and $\bar{\Lambda}<1$, we have $\Lambda<1$.
If $c_{1}\downarrow0$, then it is easy to verify that $\Lambda\downarrow0$
which leads to $\bar{\Lambda}\downarrow0$. Thus, the above assumption (ii)
holds when $c_{1}$ is sufficient small.
\end{remark}

As pointed out in the introduction, this kind of problem has a stochastic
optimal control interpretation. Now we formulate this related stochastic
optimal control problem.

Let $B=(B_{t}^{1},B_{t}^{2},...,B_{t}^{d})_{0\leq t\leq T}^{\intercal}$ be a
standard $d$-dimensional Brownian motion defined on a complete probability
space $(\Omega,\mathcal{F},P)$\ over $[0,T]$. Denote by $\mathbb{F}%
=\{\mathcal{F}_{t},0\leq t\leq T\}$ the natural filtration of $B$, where
$\mathcal{F}_{0}$ contains all $P$-null sets of $\mathcal{F}$. Given
$t\in\lbrack0,T)$, denote by $\mathcal{U}[t,T]$ the set of all $\mathbb{F}%
$-adapted $U$-valued processes on $[t,T]$. For each given $p\geq1$, we
introduce the following spaces.

$L^{p}(\mathcal{F}_{t};\mathbb{R}^{n})$ : the space of $\mathcal{F}_{t}%
$-measurable $\mathbb{R}^{n}$-valued random vectors $\zeta$ such that
$\mathbb{E}[|\zeta|^{p}]<\infty$; $L^{\infty}(\mathcal{F}_{t};\mathbb{R}^{n}%
)$: the space of $\mathcal{F}_{t}$-measurable $\mathbb{R}^{n}$-valued random
vectors $\xi$ such that \ $||\xi||_{\infty}=ess\sup_{\omega\in\Omega}%
|\xi(\omega)|<\infty; $ $L_{\mathcal{F}}^{p}(t,T;\mathbb{R}^{n})$: the space
of $\mathbb{F}$-adapted $\mathbb{R}^{n}$-valued stochastic processes on
$[t,T]$ such that \ $\mathbb{E}[\int_{t}^{T}|f(r)|^{p}dr]<\infty; $
$L_{\mathcal{F}}^{\infty}(t,T;\mathbb{R}^{n})$: the space of $\mathbb{F}%
$-adapted $\mathbb{R}^{n}$-valued stochastic processes on $[t,T]$ such that
\ $||f(\cdot)||_{\infty}=ess\sup_{(r,\omega)\in\lbrack t,T]\times\Omega
}|f(r,\omega)|<\infty; $ $L_{\mathcal{F}}^{p,q}(t,T;\mathbb{R}^{n})$: the
space of $\mathbb{F}$-adapted $\mathbb{R}^{n}$-valued stochastic processes on
$[t,T]$ such that \ $||f(\cdot)||_{p,q}=\{\mathbb{E}[(\int_{t}^{T}%
|f(r)|^{p}dr)^{\frac{q}{p}}]\}^{\frac{1}{q}}<\infty; $ $L_{\mathcal{F}}%
^{p}(\Omega;C([t,T];\mathbb{R}^{n}))$: the space of $\mathbb{F}$-adapted
$\mathbb{R}^{n}$-valued continuous stochastic processes on $[t,T]$ such that
\ $\mathbb{E}[\sup_{t\leq r\leq T}|f(r)|^{p}]<\infty. $

Let $t\in\lbrack0,T]$, $\xi\in L^{2}(\mathcal{F}_{t};\mathbb{R}^{n})$ and an
admissible control\ $u(\cdot)\in\mathcal{U}[t,T]$. Consider the following
controlled fully coupled FBSDE:
\begin{equation}
\left\{
\begin{array}
[c]{rl}%
dX_{s}^{t,\xi;u}= & b(s,X_{s}^{t,\xi;u},Y_{s}^{t,\xi;u},Z_{s}^{t,\xi;u}%
,u_{s})ds+\sigma(s,X_{s}^{t,\xi;u},Y_{s}^{t,\xi;u},Z_{s}^{t,\xi;u}%
,u_{s})dB_{s},\\
dY_{s}^{t,\xi;u}= & -g(s,X_{s}^{t,\xi;u},Y_{s}^{t,\xi;u},Z_{s}^{t,\xi;u}%
,u_{s})ds+Z_{s}^{t,\xi;u}dB_{s},\;s\in\lbrack t,T],\\
X_{t}^{t,\xi;u}= & \xi,\ Y_{T}^{t,\xi;u}=\phi(X_{T}^{t,\xi;u}).
\end{array}
\right.  \label{state-eq}%
\end{equation}

Under the Assumption \ref{assum-1}, by Theorem 2.2 in \cite{Hu-JX}, the
equation \eqref{state-eq} has a unique solution $(X^{t,\xi;u},Y^{t,\xi;u},
\newline Z^{t,\xi;u})\in L_{\mathcal{F}}^{2}(\Omega;C([t,T];\mathbb{R}%
^{n}))\times L_{\mathcal{F}}^{2}(\Omega;C([t,T];\mathbb{R}))\times
L_{\mathcal{F}}^{2,2}(t,T;\mathbb{R}^{1\times d})$. For each given
$(t,x)\in\lbrack0,$ $T]\times\mathbb{R}^{n}$, define the value function
\begin{equation}
W(t,x)=\underset{u\in\mathcal{U}[t,T]}{ess\inf}Y_{t}^{t,x;u}. \label{obje-eq}%
\end{equation}

\section{The existence of viscosity solutions}

{In order to prove the existence of the viscosity solution, we need to study
the} above fully coupled stochastic optimal control problem. It is well-known
that DPP is an important approach to solving stochastic optimal control
problems (see \cite{Yong-Zhou,Zhou-1990-2,Zhou-1991}). So in the first
subsection, {the DPP for the stochastic control problem (\ref{state-eq}%
)-(\ref{obje-eq}) is established. Then we prove the value function is a
viscosity solution to the HJB equation (\ref{pre-hjb}) in the second
subsection.}


\subsection{The dynamic programming principle}

Define $\mathcal{U}^{t}[t,T]$ the space of all $U$-valued $\{\mathcal{F}%
_{s}^{t}\}_{t\leq s\leq T}$-adapted processes on $[t,T]$, where $\{\mathcal{F}%
_{s}^{t}\}_{t\leq s\leq T}$ is the $P$-augmenta-tion of the natural filtration
of $(B_{s}-B_{t})_{t\leq s\leq T}$. For each $v\in\mathcal{U}^{t}[t,T]$, it is
easy to verify that the solution $(X_{s}^{t,x;v},Y_{s}^{t,x;v},Z_{s}%
^{t,x;v})_{s\in\lbrack t,T]}$ to equation \eqref{state-eq} is $\{\mathcal{F}%
_{s}^{t}\}_{t\leq s\leq T}$-adapted, which implies that $Y_{t}^{t,x;v}%
\in\mathbb{R}$.

It should note that in this paper, the constant $C$ will change from line to
line in the following proofs.

\begin{proposition}
\label{pro-new-1}Suppose Assumption \ref{assum-1} holds. Then
\begin{equation}
W(t,x)=\underset{v\in\mathcal{U}^{t}[t,T]}{\inf}Y_{t}^{t,x;v}.
\label{eq-deter}%
\end{equation}

\end{proposition}

\begin{proof}
Since $\mathcal{U}^{t}[t,T]\subset\mathcal{U}[t,T]$, we obtain $W(t,x)\leq
\inf_{v\in\mathcal{U}^{t}[t,T]}Y_{t}^{t,x;v}$ by the definition of $W(t,x)$.
On the other hand, for each given $u\in\mathcal{U}[t,T]$, by Lemma 13 in
\cite{Hu-J}, there exists a sequence $(u^{m})$ in $\mathcal{U}[t,T]$ such that
$\mathbb{E}\left[  \int_{t}^{T}\left\vert u_{s}^{m}-u_{s}\right\vert
^{2}ds\right]  \rightarrow0$ as $m\rightarrow\infty$. Moreover, we can take
$u_{s}^{_{m}}=\sum_{i=1}^{m}v_{s}^{i,m}I_{A_{i}}$, $s\in\lbrack t,T]$, where
$\left\{  A_{i}\right\}  _{i=1}^{m}$ is a partition of $\left(  \Omega
,\mathcal{F}_{t}\right)  $ and $v^{i,m}\in\mathcal{U}^{t}[t,T]$. By Theorem
2.2 in \cite{Hu-JX}, we get%
\begin{equation}%
\begin{array}
[c]{l}%
\mathbb{E}\left[  \left\vert Y_{t}^{t,x;u^{m}}-Y_{t}^{t,x;u}\right\vert
^{2}\right] \\
\leq C\mathbb{E}\left[  \int_{t}^{T}|b(s,X_{s}^{t,x;u},Y_{s}^{t,x;u}%
,Z_{s}^{t,x;u},u_{s}^{m})-b(s,X_{s}^{t,x;u},Y_{s}^{t,x;u},Z_{s}^{t,x;u}%
,u_{s})|^{2}ds\right] \\
\text{ \ \ }+C\mathbb{E}\left[  \int_{t}^{T}\left\vert g(s,X_{s}^{t,x;u}%
,Y_{s}^{t,x;u},Z_{s}^{t,x;u},u_{s}^{m})-g(s,X_{s}^{t,x;u},Y_{s}^{t,x;u}%
,Z_{s}^{t,x;u},u_{s})\right\vert ^{2}ds\right] \\
\text{ \ \ }+C\mathbb{E}\left[  \int_{t}^{T}|\sigma\left(  s,X_{s}%
^{t,x;u},Y_{s}^{t,x;u},Z_{s}^{t,x;u},u_{s}^{m}\right)  -\sigma(s,X_{s}%
^{t,x;u},Y_{s}^{t,x;u},Z_{s}^{t,x;u},u_{s})|^{2}ds\right]  .
\end{array}
\label{eq-diff-contr}%
\end{equation}
By Remark \ref{linear growth},%
\[
|\psi(s,X_{s}^{t,x;u},Y_{s}^{t,x;u},Z_{s}^{t,x;u},u)|\leq L(1+|X_{s}%
^{t,x;u}|+|Y_{s}^{t,x;u}|+|Z_{s}^{t,x;u}|),
\]
where $\psi=b,$ $\sigma,$ $g$ and $s\in\lbrack t,T]$, $u\in U$. Applying the
dominated convergence theorem to (\ref{eq-diff-contr}), we obtain%
\begin{equation}
\mathbb{E}\left[  \left\vert Y_{t}^{t,x;u^{m}}-Y_{t}^{t,x;u}\right\vert
^{2}\right]  \rightarrow0\text{ as }m\rightarrow\infty. \label{eq-deff-1}%
\end{equation}
Note that%
\[%
\begin{array}
[c]{l}%
\left(  X_{s}^{t,x;u^{m}},Y_{s}^{t,x;u^{m}},Z_{s}^{t,x;u^{m}}\right)
_{s\in\lbrack t,T]}\\
=\left(  \sum_{i=1}^{m}X_{s}^{t,x;v^{i,m}}I_{A_{i}},\sum_{i=1}^{m}%
Y_{s}^{t,x;v^{i,m}}I_{A_{i}},\sum_{i=1}^{m}Z_{s}^{t,x;v^{i,m}}I_{A_{i}%
}\right)  _{s\in\lbrack t,T]}.
\end{array}
\]
Then,%
\begin{equation}
Y_{t}^{t,x;u^{m}}=\sum_{i=1}^{m}Y_{t}^{t,x;v^{i,m}}I_{A_{i}}\geq
\underset{v\in\mathcal{U}^{t}[t,T]}{\inf}Y_{t}^{t,x;v}\text{, }P\text{-}a.s..
\label{eq-deff-2}%
\end{equation}
Combining (\ref{eq-deff-1}) and (\ref{eq-deff-2}), we obtain%
\[
Y_{t}^{t,x;u}\geq\underset{v\in\mathcal{U}^{t}[t,T]}{\inf}Y_{t}^{t,x;v}\text{,
}P\text{-}a.s.,
\]
which yields that $W(t,x)\geq\inf_{v\in\mathcal{U}^{t}[t,T]}Y_{t}^{t,x;v}$.
Thus $W(t,x)=\inf_{v\in\mathcal{U}^{t}[t,T]}Y_{t}^{t,x;v}$.
\end{proof}

The above Proposition shows that $W(t,x)$ is a deterministic function. To
prove $W(t,\xi)=ess\inf_{u\in\mathcal{U}[t,T]}Y_{t}^{t,\xi;u}$, we need the
following two lemmas.

\begin{lemma}
\label{est-initial}Suppose Assumption \ref{assum-1} holds. Then there exists a
constant $C$ depending on $L_{1}$, $L_{2}$, $L_{3}$ and $T$ such that, for
each $u\in\mathcal{U}[t,T]$ and $\xi,\xi^{\prime}\in L^{2}(\mathcal{F}%
_{t};\mathbb{R}^{n})$,%
\begin{equation}%
\begin{array}
[c]{rl}
& \mathbb{E}\left[  \left.  \sup\limits_{t\leq s\leq T}\left(  |X_{s}%
^{t,\xi;u}-X_{s}^{t,\xi^{\prime};u}|^{2}+|Y_{s}^{t,\xi;u}-Y_{s}^{t,\xi
^{\prime};u}|^{2}\right)  +\int_{t}^{T}|Z_{s}^{t,\xi;u}-Z_{s}^{t,\xi^{\prime
};u}|^{2}ds\right\vert \mathcal{F}_{t}\right] \\
& \ \ \leq C\left\vert \xi-\xi^{\prime}\right\vert ^{2};
\end{array}
\label{est-initial-1}%
\end{equation}%
\begin{equation}
\mathbb{E}\left[  \left.  \sup\limits_{t\leq s\leq T}\left(  |X_{s}^{t,\xi
;u}|^{2}+|Y_{s}^{t,\xi;u}|^{2}\right)  +\int_{t}^{T}|Z_{s}^{t,\xi;u}%
|^{2}ds\right\vert \mathcal{F}_{t}\right]  \leq C\left(  1+\left\vert
\xi\right\vert ^{2}\right)  . \label{est-initial-2}%
\end{equation}

\end{lemma}

\begin{proof}
Without loss of generality, we only prove the case $d=1$. Set $\hat
{X}=X^{t,\xi;u}-X^{t,\xi^{\prime};u}$, $\hat{Y}=Y^{t,\xi;u}-Y^{t,\xi^{\prime
};u}$, $\hat{Z}=Z^{t,\xi;u}-Z^{t,\xi^{\prime};u}$. Then $\left(  \hat{X}%
,\hat{Y},\hat{Z}\right)  $ satisfies the following FBSDE,%
\[
\left\{
\begin{array}
[c]{rl}%
d\hat{X}_{s}= & \left[  b^{1}(s)\hat{X}_{s}+b^{2}(s)\hat{Y}_{s}+b^{3}%
(s)\hat{Z}_{s}\right]  ds\\
& +\left[  \sigma^{1}(s)\hat{X}_{s}+\sigma^{2}(s)\hat{Y}_{s}+\sigma^{3}%
(s)\hat{Z}_{s}\right]  dB_{s},\\
d\hat{Y}_{s}= & -\left[  g^{1}(s)\hat{X}_{s}+g^{2}(s)\hat{Y}_{s}+g^{3}%
(s)\hat{Z}_{s}\right]  ds+\hat{Z}_{s}dB_{s},\text{ }s\in\lbrack t,T],\\
\hat{X}_{t}= & \xi-\xi^{\prime},\ \hat{Y}_{T}=\phi^{1}(T)\hat{X}_{T},
\end{array}
\right.
\]
where
\[
b^{1}\left(  s\right)  =\left\{
\begin{array}
[c]{cc}%
\frac{b(s,X_{s}^{t,\xi;u},Y_{s}^{t,\xi;u},Z_{s}^{t,\xi;u},u_{s})-b(s,X_{s}%
^{t,\xi^{\prime};u},,Y_{s}^{t,\xi;u},Z_{s}^{t,\xi;u},u_{s})}{X_{s}^{t,\xi
;u}-X_{s}^{t,\xi^{\prime};u}}\text{,} & \hat{X}_{s}\neq0,\\
0\text{,} & \hat{X}_{s}=0,
\end{array}
\right.
\]
and $b^{i}$, $\sigma^{i}$, $g^{i}$, $\phi^{1}$ are defined similarly,
$i=1,2,3$. By Assumption \ref{assum-1}, $b^{i}$, $\sigma^{i}$, $g^{i}$ and
$\phi^{1}$ are bounded for $i=1,2,3$. Then (\ref{est-initial-1}) holds by
Theorem 2.2 in \cite{Hu-JX}.

Since $|\psi(s,0,0,0,u)|\leq L$ for $s\in\lbrack t,T]$ and $u\in U$ where
$\psi=b,$ $\sigma,$ $g$, by Theorem 2.2 in \cite{Hu-JX}, we obtain
\[%
\begin{array}
[c]{l}%
\mathbb{E}\left[  \left.  \sup\limits_{t\leq s\leq T}\left(  |X_{s}^{t,\xi
;u}|^{2}+|Y_{s}^{t,\xi;u}|^{2}\right)  +\int_{t}^{T}|Z_{s}^{t,\xi;u}%
|^{2}ds\right\vert \mathcal{F}_{t}\right] \\
\text{ \ }\leq C\mathbb{E}\left[  \left.  \left\vert \xi\right\vert
^{2}+\left(  \int_{t}^{T}\left(  |b|+|g|\right)  \left(  s,0,0,0,u_{s}\right)
ds\right)  ^{2}+\int_{t}^{T}|\sigma\left(  s,0,0,0,u_{s}\right)
|^{2}ds\right\vert \mathcal{F}_{t}\right] \\
\text{ \ }\leq C\left(  1+\left\vert \xi\right\vert ^{2}\right)  .
\end{array}
\]
This completes the proof.
\end{proof}

\begin{lemma}
\label{le-w-new}Suppose Assumption \ref{assum-1} holds. Then there exist two
constants $C$ and $C^{\prime}$ depending on $L_{1}$, $L_{2}$, $L_{3}$ and $T$
such that, for each $t\in\lbrack0,T]$ and $x,x^{\prime}\in\mathbb{R}^{n}$,%
\[
|W(t,x)-W(t,x^{\prime})|\leq C|x-x^{\prime}|\text{ and }|W(t,x)|\leq
C^{\prime}(1+|x|),
\]
where $C=\sqrt{C_{2}}\left(  1-\sqrt{\Lambda}\right)  ^{-1}$.
\end{lemma}

\begin{proof}
By Proposition \ref{pro-new-1} and Lemma \ref{est-initial}, we obtain
\[%
\begin{array}
[c]{rl}%
|W(t,x)-W(t,x^{\prime})| & =\left\vert \underset{v\in\mathcal{U}%
^{t}[t,T]}{\inf}Y_{t}^{t,x;v}-\underset{v\in\mathcal{U}^{t}[t,T]}{\inf}%
Y_{t}^{t,x^{\prime};v}\right\vert \\
& \leq\underset{v\in\mathcal{U}^{t}[t,T]}{\sup}\left\vert Y_{t}^{t,x;v}%
-Y_{t}^{t,x^{\prime};v}\right\vert \\
& \leq\underset{v\in\mathcal{U}^{t}[t,T]}{\sup}\left\{  \mathbb{E}\left[
\left.  \sup\limits_{t\leq s\leq T}\left\vert Y_{s}^{t,x;v}-Y_{s}%
^{t,x^{\prime};v}\right\vert ^{2}\right\vert \mathcal{F}_{t}\right]  \right\}
^{\frac{1}{2}}\\
& \leq C|x-x^{\prime}|.
\end{array}
\]
By the proof of Theorem 2.2 in \cite{Hu-JX}, we can obtain $C=\sqrt{C_{2}%
}\left(  1-\sqrt{\Lambda}\right)  ^{-1}$. The second inequality can be proved similarly.
\end{proof}

\begin{proposition}
\label{pro-w-random}Suppose Assumption \ref{assum-1} holds. Then, for each
$\xi\in$$L^{2}(\mathcal{F}_{t};\mathbb{R}^{n})$, we have \ $W(t,\xi
)=ess\inf_{u\in\mathcal{U}[t,T]}Y_{t}^{t,\xi;u}. $
\end{proposition}

\begin{proof}
It is clear that there exists a sequence of random vectors $\xi^{m}$%
$=\sum_{i=1}^{m}x_{i}^{m}I_{A_{i}^{m}}$ such that $\mathbb{E}\left[
\left\vert \xi^{m}-\xi\right\vert ^{2}\right]  \rightarrow0$ as $m\rightarrow
\infty$, where $\left\{  A_{i}^{m}\right\}  _{i=1}^{m}$ is a partition of
$\left(  \Omega,\mathcal{F}_{t}\right)  $ and $x_{i}^{m}\in\mathbb{R}^{n}$.
Similar to the proof of Proposition \ref{pro-new-1}, we have%
\[%
\begin{array}
[c]{l}%
\left(  X_{s}^{t,\xi_{m};u},Y_{s}^{t,\xi_{m};u},Z_{s}^{t,\xi_{m};u}\right)
_{s\in\lbrack t,T]}\\
=\left(  \sum_{i=1}^{m}X_{s}^{t,x_{i}^{m};u}I_{A_{i}^{m}},\sum_{i=1}^{m}%
Y_{s}^{t,x_{i}^{m};u}I_{A_{i}^{m}},\sum_{i=1}^{m}Z_{s}^{t,x_{i}^{m};u}%
I_{A_{i}^{m}}\right)  _{s\in\lbrack t,T]}.
\end{array}
\]
Thus
\begin{equation}%
\begin{array}
[c]{rl}%
\underset{u\in\mathcal{U}[t,T]}{ess\inf}Y_{t}^{t,\xi_{m};u} & =\underset{u\in
\mathcal{U}[t,T]}{ess\inf}\sum\limits_{i=1}^{m}Y_{t}^{t,x_{i}^{m};u}%
I_{A_{i}^{m}}\\
& =\sum\limits_{i=1}^{m}\left(  \underset{u\in\mathcal{U}[t,T]}{ess\inf}%
Y_{t}^{t,x_{i}^{m};u}\right)  I_{A_{i}^{m}}\\
& =\sum\limits_{i=1}^{m}W\left(  t,x_{i}^{m}\right)  I_{A_{i}^{m}}\\
& =W\left(  t,\xi^{m}\right)  .
\end{array}
\label{eq-new-111}%
\end{equation}
On the other hand, by Lemmas \ref{est-initial} and \ref{le-w-new}, we have
\begin{equation}%
\begin{array}
[c]{rl}%
\left\vert \underset{u\in\mathcal{U}[t,T]}{ess\inf}Y_{t}^{t,\xi_{m}%
;u}-\underset{u\in\mathcal{U}[t,T]}{ess\inf}Y_{t}^{t,\xi;u}\right\vert  &
\leq\underset{u\in\mathcal{U}[t,T]}{ess\sup}\left\vert Y_{t}^{t,\xi_{m}%
;u}-Y_{t}^{t,\xi;u}\right\vert \\
& \leq C\left\vert \xi^{m}-\xi\right\vert ,
\end{array}
\label{eq-new-112}%
\end{equation}%
\begin{equation}
\left\vert W(t,\xi^{m})-W\left(  t,\xi\right)  \right\vert \leq C\left\vert
\xi^{m}-\xi\right\vert . \label{eq-new-113}%
\end{equation}
Combining (\ref{eq-new-111}), (\ref{eq-new-112}) and (\ref{eq-new-113}), we
get%
\[
\left\vert \underset{u\in\mathcal{U}[t,T]}{ess\inf}Y_{t}^{t,\xi;u}-W\left(
t,\xi\right)  \right\vert \leq2C\left\vert \xi^{m}-\xi\right\vert .
\]
Thus we obtain the desired result by letting $m\rightarrow\infty$.
\end{proof}

Before studying the DPP, we introduce the notion of backward semigroup, which
was first introduced by Peng in \cite{Peng-lecture}. For each given
$(t,x)\in\lbrack0,T)\times\mathbb{R}^{n}$, $\delta\in(0,T-t]$ and
$u\in\mathcal{U}[t,t+\delta]$,  define $G_{t,t+\delta}%
^{t,x;u}[\cdot]:\mathcal{L}(\mathbb{R}^{n})\rightarrow L^{2}(\mathcal{F}%
_{t};\mathbb{R})$ as
\[
G_{t,t+\delta}^{t,x;u}\left[  \psi(\tilde{X}_{t+\delta}^{t,x;u})\right]
=\tilde{Y}_{t}^{t,x;u}\text{ for each }\psi\in\mathcal{L}(\mathbb{R}^{n}),
\]
where
\[
\mathcal{L}(\mathbb{R}^{n})=\{f:\mathbb{R}^{n}\rightarrow\mathbb{R}%
|\ |f(x)-f(x^{\prime})|\leq L|x-x^{\prime}| \},
\]
$L$ is the constant in Assumption \ref{assum-1}  and $(\tilde{X}%
^{t,x;u},\tilde{Y}^{t,x;u}, \tilde{Z}^{t,x;u})$ is the solution to the
following FBSDE on $[t,t+\delta]$,%
\begin{equation}
\left\{
\begin{array}
[c]{rl}%
d\tilde{X}_{s}^{t,x;u}= & b(s,\tilde{X}_{s}^{t,x;u},\tilde{Y}_{s}%
^{t,x;u},\tilde{Z}_{s}^{t,x;u},u_{s})ds+\sigma(s,\tilde{X}_{s}^{t,x;u}%
,\tilde{Y}_{s}^{t,x;u},\tilde{Z}_{s}^{t,x;u},u_{s})dB_{s},\\
d\tilde{Y}_{s}^{t,x;u}= & -g(s,\tilde{X}_{s}^{t,x;u},\tilde{Y}_{s}%
^{t,x;u},\tilde{Z}_{s}^{t,x;u},u_{s})ds+\tilde{Z}_{s}^{t,x;u}dB_{s},\text{
}s\in\lbrack t,t+\delta],\\
\tilde{X}_{s}^{t,x;u}= & x,\ \tilde{Y}_{t+\delta}^{t,x;u}=\psi(\tilde
{X}_{t+\delta}^{t,x;u}).
\end{array}
\right.  \label{eq-xy-til}%
\end{equation}
Since the coefficients of (\ref{eq-xy-til}) satisfies Assumption
\ref{assum-1}, there exists a unique solution $(\tilde{X}^{t,x;u},\tilde
{Y}^{t,x;u},\tilde{Z}^{t,x;u})$ to (\ref{eq-xy-til}). Thus $G_{t,t+\delta
}^{t,x;u}[\cdot]$ is well-defined.

Now we prove the DPP for the control system (\ref{state-eq})-(\ref{obje-eq}).

\begin{theorem}
\label{th-ddp}Suppose Assumption \ref{assum-1} holds. Then for each
$(t,x)\in\lbrack0,T)\times\mathbb{R}^{n}$ and $\delta\in(0,T-t]$, we have%
\[
W(t,x)=\underset{u\in\mathcal{U}[t,t+\delta]}{ess\inf}G_{t,t+\delta}%
^{t,x;u}\left[  W(t+\delta,\tilde{X}_{t+\delta}^{t,x;u})\right]
=\underset{v\in\mathcal{U}^{t}[t,t+\delta]}{\inf}G_{t,t+\delta}^{t,x;v}\left[
W(t+\delta,\tilde{X}_{t+\delta}^{t,x;v})\right]  .
\]

\end{theorem}

\begin{proof}
We first prove
\begin{equation}
W(t,x)\geq\underset{v\in\mathcal{U}^{t}[t,t+\delta]}{\inf}G_{t,t+\delta
}^{t,x;v}\left[  W(t+\delta,\tilde{X}_{t+\delta}^{t,x;v})\right]  .
\label{new-141}%
\end{equation}
For each given $v\in\mathcal{U}^{t}[t,T]$, noting that
\[
\left(  X_{s}^{t,x;v},Y_{s}^{t,x;v},Z_{s}^{t,x;v}\right)  _{s\in\left[
t+\delta,T\right]  }=\left(  X_{s}^{t+\delta,X_{t+\delta}^{t,x;v};v}%
,Y_{s}^{t+\delta,X_{t+\delta}^{t,x;v};v},Z_{s}^{t+\delta,X_{t+\delta}%
^{t,x;v};v}\right)  _{s\in\left[  t+\delta,T\right]  },
\]
we have%
\[
\left\{
\begin{array}
[c]{rl}%
dX_{s}^{t,x;v}= & b(s,X_{s}^{t,x;v},Y_{s}^{t,x;v},Z_{s}^{t,x;v},v_{s}%
)ds+\sigma(s,X_{s}^{t,x;v},Y_{s}^{t,x;v},Z_{s}^{t,x;v},v_{s})dB_{s},\\
dY_{s}^{t,x;v}= & -g(s,X_{s}^{t,x;v},Y_{s}^{t,x;v},Z_{s}^{t,x;v}%
,v_{s})ds+Z_{s}^{t,x;v}dB_{s},\text{ }s\in\lbrack t,t+\delta],\\
X_{t}^{t,x;v}= & x,\ Y_{t+\delta}^{t,x;v}=Y_{t+\delta}^{t+\delta,X_{t+\delta
}^{t,x;v};v}.
\end{array}
\right.
\]
By Proposition \ref{pro-w-random}, we get \ $W(t+\delta,X_{t+\delta}%
^{t,x;v})\leq Y_{t+\delta}^{t+\delta,X_{t+\delta}^{t,x;v};v}. $ Taking $u=v$
and $\psi(\cdot)=W(t+\delta,\cdot)$ in (\ref{eq-xy-til}), by comparison
theorem for FBSDE (see Theorem \ref{th-comp} ), we get \ $Y_{t}^{t,x;v}%
\geq\tilde{Y}_{t}^{t,x;v}. $ By Proposition \ref{pro-new-1}, we obtain
(\ref{new-141}).

Next, we prove%
\begin{equation}
W(t,x)\leq\underset{u\in\mathcal{U}[t,t+\delta]}{ess\inf}G_{t,t+\delta
}^{t,x;u}\left[  W(t+\delta,\tilde{X}_{t+\delta}^{t,x;u})\right]  .
\label{new-142}%
\end{equation}
It is obvious that we only need to prove%
\begin{equation}
G_{t,t+\delta}^{t,x;u}\left[  W(t+\delta,\tilde{X}_{t+\delta}^{t,x;u})\right]
\geq W(t,x),\text{ }P\text{-a.s.} \label{new-111}%
\end{equation}
for each $u\in\mathcal{U}[t,t+\delta]$. The proof for (\ref{new-111}) is
divided into four steps.

\textbf{Step 1.} Let $(\tilde{X}^{t,x;u},\tilde{Y}^{t,x;u},\tilde{Z}^{t,x;u})$
be the solution to the following FBSDE:
\begin{equation}
\left\{
\begin{array}
[c]{rl}%
d\tilde{X}_{s}^{t,x;u}= & b(s,\tilde{X}_{s}^{t,x;u},\tilde{Y}_{s}%
^{t,x;u},\tilde{Z}_{s}^{t,x;u},u_{s})ds+\sigma(s,\tilde{X}_{s}^{t,x;u}%
,\tilde{Y}_{s}^{t,x;u},\tilde{Z}_{s}^{t,x;u},u_{s})dB_{s},\\
d\tilde{Y}_{s}^{t,x;u}= & -g(s,\tilde{X}_{s}^{t,x;u},\tilde{Y}_{s}%
^{t,x;u},\tilde{Z}_{s}^{t,x;u},u_{s})ds+\tilde{Z}_{s}^{t,x;u}dB_{s},\text{
}s\in\lbrack t,t+\delta],\\
\tilde{X}_{t}^{t,x;u}= & x,\ \tilde{Y}_{t+\delta}^{t,x;u}=W(t+\delta,\tilde
{X}_{t+\delta}^{t,x;u}).
\end{array}
\right.  \label{eq-xy-til-w}%
\end{equation}
Since $\tilde{X}_{t+\delta}^{t,x;u}\in L^{2}(\mathcal{F}_{t+\delta}%
;\mathbb{R}^{n})$, for each integer $m$, we can choose a partition
$\{A_{i}^{m}:i=1,\ldots,m\}$ of $\mathcal{F}_{t+\delta}$ and $x_{i}^{m}%
\in\mathbb{R}^{n}$ such that \ $\mathbb{E}\left[  \left\vert \xi^{m}-\tilde
{X}_{t+\delta}^{t,x;u}\right\vert ^{2}\right]  \rightarrow0\text{ as
}m\rightarrow\infty, $ where $\xi^{m}=\sum_{i=1}^{m}x_{i}^{m}I_{A_{i}^{m}}$.
For each given $x_{i}^{m}$, by Proposition \ref{pro-new-1}, we can find a
$v^{i,m}\in\mathcal{U}^{t+\delta}[t+\delta,T]$ such that%
\begin{equation}
W(t+\delta,x_{i}^{m})\leq Y_{t+\delta}^{t+\delta,x_{i}^{m};v^{i,m}}\leq
W(t+\delta,x_{i}^{m})+\frac{1}{m}. \label{new-131}%
\end{equation}
Set \ $u_{s}^{m}=u_{s}I_{[t,t+\delta]}(s)+\left(  \sum_{i=1}^{m}v_{s}%
^{i,m}I_{A_{i}^{m}}\right)  I_{(t+\delta,T]}(s). $

\textbf{Step 2}. By (\ref{est-initial-1}) in Lemma \ref{est-initial}, we get%
\begin{equation}
\mathbb{E}\left[  \left\vert Y_{t+\delta}^{t+\delta,\tilde{X}_{t+\delta
}^{t,x;u};u^{m}}-Y_{t+\delta}^{t+\delta,\xi^{m};u^{m}}\right\vert ^{2}\right]
\leq C\mathbb{E}\left[  \left\vert \tilde{X}_{t+\delta}^{t,x;u}-\xi
^{m}\right\vert ^{2}\right]  \rightarrow0\text{ as }m\rightarrow\infty.
\label{new-new-2312}%
\end{equation}
Similar to the proof of Proposition \ref{pro-new-1}, one can check that
\begin{equation}%
\begin{array}
[c]{l}%
\left(  X_{s}^{t+\delta,\xi^{m};u^{m}},Y_{s}^{t+\delta,\xi^{m};u^{m}}%
,Z_{s}^{t+\delta,\xi^{m};u^{m}}\right)  _{s\in\lbrack t+\delta,T]}\\
=\left(  \displaystyle\sum_{i=1}^{m}X_{s}^{t+\delta,x_{i}^{m};v^{i,m}}%
I_{A_{i}^{m}},\displaystyle\sum_{i=1}^{m}Y_{s}^{t+\delta,x_{i}^{m};v^{i,m}%
}I_{A_{i}^{m}},\displaystyle\sum_{i=1}^{m}Z_{s}^{t+\delta,x_{i}^{m};v^{i,m}%
}I_{A_{i}^{m}}\right)  _{s\in\lbrack t+\delta,T]}.
\end{array}
\label{new-132}%
\end{equation}
Combining (\ref{new-131}) and (\ref{new-132}), we obtain%
\begin{equation}
W(t+\delta,\xi^{m})\leq Y_{t+\delta}^{t+\delta,\xi^{m};u^{m}}\leq
W(t+\delta,\xi^{m})+\frac{1}{m}. \label{new-133}%
\end{equation}
Thus, by Lemma \ref{le-w-new},
\begin{equation}
\mathbb{E}\left[  \left\vert Y_{t+\delta}^{t+\delta,\xi^{m};u^{m}}%
-W(t+\delta,\tilde{X}_{t+\delta}^{t,x;u})\right\vert ^{2}\right]
\rightarrow0\text{ as }m\rightarrow\infty. \label{y-hat-w}%
\end{equation}
By (\ref{new-new-2312}) and (\ref{y-hat-w}), we get
\begin{equation}
\mathbb{E}\left[  \left\vert Y_{t+\delta}^{t+\delta,\tilde{X}_{t+\delta
}^{t,x;u};u^{m}}-W(t+\delta,\tilde{X}_{t+\delta}^{t,x;u})\right\vert
^{2}\right]  \rightarrow0\text{ as }m\rightarrow\infty. \label{new-134}%
\end{equation}
Consider the following decoupled FBSDE:%
\begin{equation}
\left\{
\begin{array}
[c]{rl}%
d\bar{X}_{s}^{m} & =b(s,\bar{X}_{s}^{m},\tilde{Y}_{s}^{t,x;u},\tilde{Z}%
_{s}^{t,x;u},u_{s})ds+\sigma(s,\bar{X}_{s}^{m},\tilde{Y}_{s}^{t,x;u},\tilde
{Z}_{s}^{t,x;u},u_{s})dB_{s},\\
d\bar{Y}_{s}^{m} & =-g(s,\bar{X}_{s}^{m},\bar{Y}_{s}^{m},\bar{Z}_{s}^{m}%
,u_{s}^{m})ds+\bar{Z}_{s}^{m}dB_{s},\text{ }s\in\lbrack t,t+\delta],\\
\bar{X}_{t}^{m} & =x,\ \bar{Y}_{t+\delta}^{m}=Y_{t+\delta}^{t+\delta,\tilde
{X}_{t+\delta}^{t,x;u};u^{m}}.
\end{array}
\right.  \label{new-new-3422}%
\end{equation}
By (\ref{eq-xy-til-w}), we know that $(\tilde{X}_{s}^{t,x;u})_{s\in\lbrack
t,t+\delta]}$ satisfies the SDE in (\ref{new-new-3422}), which implies
$\bar{X}_{s}^{m}=\tilde{X}_{s}^{t,x;u}$ on $[t,t+\delta]$. Thus, by the
estimate of BSDE, we get%
\begin{equation}%
\begin{array}
[c]{rl}
& \mathbb{E}\left[  \underset{s\in\lbrack t,t+\delta]}{\sup}\left\vert
\tilde{Y}_{s}^{t,x;u}-\bar{Y}_{s}^{m}\right\vert ^{2}+\int_{t}^{t+\delta
}\left\vert \tilde{Z}_{s}^{t,x;u}-\bar{Z}_{s}^{m}\right\vert ^{2}ds\right] \\
& \ \ \leq C\mathbb{E}\left[  \left\vert Y_{t+\delta}^{t+\delta,\tilde
{X}_{t+\delta}^{t,x;u};u^{m}}-W(t+\delta,\tilde{X}_{t+\delta}^{t,x;u}%
)\right\vert ^{2}\right]  .
\end{array}
\label{new-124}%
\end{equation}
It follows from (\ref{new-134}) and (\ref{new-124}) that
\begin{equation}
\mathbb{E}\left[  \underset{s\in\lbrack t,t+\delta]}{\sup}\left\vert \tilde
{Y}_{s}^{t,x;u}-\bar{Y}_{s}^{m}\right\vert ^{2}+\int_{t}^{t+\delta}\left\vert
\tilde{Z}_{s}^{t,x;u}-\bar{Z}_{s}^{m}\right\vert ^{2}ds\right]  \rightarrow0
\label{new-new-3423}%
\end{equation}
as $m\rightarrow\infty.$

\textbf{Step 3.} Define $(X_{s}^{m},Y_{s}^{m},Z_{s}^{m})_{s\in\lbrack t,T]}$
as follows%
\begin{align*}
(X_{s}^{m},Y_{s}^{m},Z_{s}^{m})_{s\in\lbrack t,t+\delta]}  &  =(\bar{X}%
_{s}^{m},\bar{Y}_{s}^{m},\bar{Z}_{s}^{m})_{s\in\lbrack t,t+\delta]},\text{ }\\
(X_{s}^{m},Y_{s}^{m},Z_{s}^{m})_{s\in(t+\delta,T]}  &  =(X_{s}^{t+\delta
,\tilde{X}_{t+\delta}^{t,x;u};u^{m}},Y_{s}^{t+\delta,\tilde{X}_{t+\delta
}^{t,x;u};u^{m}},Z_{s}^{t+\delta,\tilde{X}_{t+\delta}^{t,x;u};u^{m}}%
)_{s\in(t+\delta,T]}.
\end{align*}
It is easy to verify that $(X_{s}^{m},Y_{s}^{m},Z_{s}^{m})_{s\in\lbrack t,T]}$
satisfies the following FBSDE:
\begin{equation}
\left\{
\begin{array}
[c]{rl}%
dX_{s}^{m} & =[b(s,X_{s}^{m},Y_{s}^{m},Z_{s}^{m},u_{s}^{m})+l_{s}%
^{1}]ds+[\sigma(s,X_{s}^{m},Y_{s}^{m},Z_{s}^{m},u_{s}^{m})+l_{s}^{2}]dB_{s},\\
dY_{s}^{m} & =-g(s,X_{s}^{m},Y_{s}^{m},Z_{s}^{m},u_{s}^{m})ds+Z_{s}^{m}%
dB_{s},\\
X_{t}^{m} & =x,\ Y_{T}^{m}=\phi\left(  X_{T}^{m}\right)  ,
\end{array}
\right.  \label{eq-xy-til-t}%
\end{equation}
where%
\[%
\begin{array}
[c]{l}%
l_{s}^{1}=b(s,\bar{X}_{s}^{m},\tilde{Y}_{s}^{t,x;u},\tilde{Z}_{s}%
^{t,x;u},u_{s})-b(s,\bar{X}_{s}^{m},\bar{Y}_{s}^{m},\bar{Z}_{s}^{m}%
,u_{s}),\text{ }s\in\lbrack t,t+\delta];\\
\text{ }l_{s}^{1}=0,\text{ }s\in(t+\delta,T];\\
l_{s}^{2}=\sigma(s,\bar{X}_{s}^{m},\tilde{Y}_{s}^{t,x;u},\tilde{Z}_{s}%
^{t,x;u},u_{s})-\sigma(s,\bar{X}_{s}^{m},\bar{Y}_{s}^{m},\bar{Z}_{s}^{m}%
,u_{s}),\text{ }s\in\lbrack t,t+\delta];\\
\text{ }l_{s}^{2}=0,\text{ }s\in(t+\delta,T].
\end{array}
\]
By Theorem 2.2 in \cite{Hu-JX}, we obtain
\begin{equation}%
\begin{array}
[c]{l}%
\mathbb{E}\left[  \left\vert Y_{t}^{m}-Y_{t}^{t,x;u^{m}}\right\vert
^{2}\right]  =\mathbb{E}\left[  \left\vert \bar{Y}_{t}^{m}-Y_{t}^{t,x;u^{m}%
}\right\vert ^{2}\right] \\
\leq C\mathbb{E}\left[  \displaystyle{\int}_{t}^{t+\delta}\left(  \left\vert
\tilde{Y}_{s}^{t,x;u}-\bar{Y}_{s}^{m}\right\vert ^{2}+\left\vert \tilde{Z}%
_{s}^{t,x;u}-\bar{Z}_{s}^{m}\right\vert ^{2}\right)  ds\right]  .
\end{array}
\label{new-122}%
\end{equation}

\textbf{Step 4.} By (\ref{new-new-3423}) and (\ref{new-122}), we get as
$m\rightarrow\infty$
\begin{equation}
\mathbb{E}\left[  \left\vert Y_{t}^{t,x;u^{m}}-\tilde{Y}_{t}^{t,x;u}%
\right\vert ^{2}\right]  \leq2\left\{  \mathbb{E}\left[  \left\vert
Y_{t}^{t,x;u^{m}}-\bar{Y}_{t}^{m}\right\vert ^{2}\right]  +\mathbb{E}\left[
\left\vert \bar{Y}_{t}^{m}-\tilde{Y}_{s}^{t,x;u}\right\vert ^{2}\right]
\right\}  \rightarrow0. \label{new-121}%
\end{equation}
By the definition of $W(t,x)$, we know that $Y_{t}^{t,x;u^{m}}\geq W(t,x)$
$P$-a.s. for $m\geq1$. Thus we obtain (\ref{new-111}) by (\ref{new-121}).

Finally, since
\[
\underset{v\in\mathcal{U}^{t}[t,t+\delta]}{\inf}G_{t,t+\delta}^{t,x;v}\left[
W(t+\delta,\tilde{X}_{t+\delta}^{t,x;v})\right]  \geq\underset{u\in
\mathcal{U}[t,t+\delta]}{ess\inf}G_{t,t+\delta}^{t,x;u}\left[  W(t+\delta
,\tilde{X}_{t+\delta}^{t,x;u})\right]  ,
\]
we obtain the desired result by (\ref{new-141}) and (\ref{new-142}).
\end{proof}

\begin{remark}
It is important to note that $(Y_{s}^{t,x;u^{m}},Z_{s}^{t,x;u^{m}}%
)_{s\in\lbrack t,t+\delta]}$ varies with $u^{m}$, which leads to the change of
$(X_{s}^{t,x;u^{m}})_{s\in\lbrack t,t+\delta]}$ in the fully coupled case.
This is different from the decoupled case, and the approach of establishing
the DPP for the decoupled case does not work now. To overcome this difficulty,
we introduce two auxiliary FBSDEs (\ref{new-new-3422}) and (\ref{eq-xy-til-t})
to prove the DPP.
\end{remark}

Now we prove the continuity property of $W(t,x)$ in $t$.

\begin{lemma}
\label{le-w-t}Suppose Assumption \ref{assum-1} holds. Then the value function
$W(t,x)$ is $\frac{1}{2}$ H\"{o}lder continuous in $t$.
\end{lemma}

\begin{proof}
For each $(t,x)\in\lbrack0,T)\times\mathbb{R}^{n}$ and $\delta\in(0,T-t]$, by
Theorem \ref{th-ddp}, we have
\[
W(t,x)=\inf_{v\in\mathcal{U}^{t}[t,t+\delta]}G_{t,t+\delta}^{t,x;v}\left[
W(t+\delta,\tilde{X}_{t+\delta}^{t,x;v})\right]  .
\]
Thus%
\[
\left\vert W(t,x)-W\left(  t+\delta,x\right)  \right\vert \leq\underset{v\in
\mathcal{U}^{t}[t,t+\delta]}{\sup}\left\vert G_{t,t+\delta}^{t,x;v}\left[
W(t+\delta,\tilde{X}_{t+\delta}^{t,x;v})\right]  -W\left(  t+\delta,x\right)
\right\vert .
\]
For each $v\in\mathcal{U}^{t}[t,t+\delta]$, by the definition of
$G_{t,t+\delta}^{t,x;v}\left[  \cdot\right]  $, we have
\ $G_{t,t+\delta}^{t,x;v}\left[  W(t+\delta,\tilde{X}_{t+\delta}%
^{t,x;v})\right]  \newline=\mathbb{E}\left[  W(t+\delta,\tilde{X}_{t+\delta
}^{t,x;v})+\int_{t}^{t+\delta}g\left(  s,\tilde{X}_{s}^{t,x;v},\tilde{Y}%
_{s}^{t,x;v},\tilde{Z}_{s}^{t,x;v},v_{s}\right)  ds\right]  . $
Thus, by Lemma \ref{le-w-new},
\begin{equation}%
\begin{array}
[c]{l}%
\left\vert G_{t,t+\delta}^{t,x;v}\left[  W(t+\delta,\tilde{X}_{t+\delta
}^{t,x;v})\right]  -W\left(  t+\delta,x\right)  \right\vert \\
\leq\mathbb{E}\left[  \left\vert W(t+\delta,\tilde{X}_{t+\delta}%
^{t,x;v})-W\left(  t+\delta,x\right)  \right\vert +\displaystyle{\int}%
_{t}^{t+\delta}\left\vert g\left(  s,\tilde{X}_{s}^{t,x;v},\tilde{Y}%
_{s}^{t,x;v},\tilde{Z}_{s}^{t,x;v},v_{s}\right)  \right\vert ds\right] \\
\leq C\mathbb{E}\left[  \left\vert \tilde{X}_{t+\delta}^{t,x;v}-x\right\vert
+\displaystyle{\int}_{t}^{t+\delta}\left(  1+\left\vert \tilde{X}_{s}%
^{t,x;v}\right\vert +\left\vert \tilde{Y}_{s}^{t,x;v}\right\vert +\left\vert
\tilde{Z}_{s}^{t,x;v}\right\vert \right)  ds\right]  .
\end{array}
\label{eq-new-133}%
\end{equation}
It follows from Theorem 2.2 in \cite{Hu-JX} that
\[%
\begin{array}
[c]{rl}%
\mathbb{E}\left[  \sup\limits_{t\leq s\leq t+\delta}\left(  |\tilde{X}%
_{s}^{t,x;v}|^{2}+|\tilde{Y}_{s}^{t,x;v}|^{2}\right)  +\displaystyle{\int}%
_{t}^{t+\delta}|\tilde{Z}_{s}^{t,x;v}|^{2}ds\right]  & \leq C\left(
1+\left\vert x\right\vert ^{2}\right)  ,
\end{array}
\]
which implies that%
\begin{equation}
\mathbb{E}\left[  \int_{t}^{t+\delta}\left(  1+\left\vert \tilde{X}%
_{s}^{t,x;v}\right\vert +\left\vert \tilde{Y}_{s}^{t,x;v}\right\vert
+\left\vert \tilde{Z}_{s}^{t,x;v}\right\vert \right)  ds\right]  \leq C\left(
1+\left\vert x\right\vert \right)  \delta^{\frac{1}{2}}. \label{eq-new-134}%
\end{equation}
By (\ref{eq-new-133}) and (\ref{eq-new-134}), we just need to estimate
$\mathbb{E}\left[  \left\vert \tilde{X}_{t+\delta}^{t,x;v}-x\right\vert
\right]  $. Denote $\hat{X}_{s}=\tilde{X}_{s}^{t,x;v}-x$, $\hat{Y}_{s}%
=\tilde{Y}_{s}^{t,x;v}-W\left(  t+\delta,x\right)  $, $\hat{Z}_{s}=\tilde
{Z}_{s}^{t,x;v}$, then $\left(  \hat{X},\hat{Y},\hat{Z}\right)  $ satisfies
the following FBSDE:
\begin{equation}
\left\{
\begin{array}
[c]{rl}%
d\hat{X}_{s}= & b(s,\hat{X}_{s}+x,\hat{Y}_{s}+W\left(  t+\delta,x\right)
,\hat{Z}_{s},v_{s})ds\\
& +\sigma(s,\hat{X}_{s}+x,\hat{Y}_{s}+W\left(  t+\delta,x\right)  ,\hat{Z}%
_{s},v_{s})dB_{s},\\
d\hat{Y}_{s}= & -g(s,\hat{X}_{s}+x,\hat{Y}_{s}+W\left(  t+\delta,x\right)
,\hat{Z}_{s},v_{s})ds+\hat{Z}_{s}dB_{s},\text{ }s\in\lbrack t,t+\delta],\\
\hat{X}_{t}= & 0,\ \hat{Y}_{t+\delta}=W(t+\delta,\hat{X}_{t+\delta
}+x)-W\left(  t+\delta,x\right)  .
\end{array}
\right.  \label{FBSDE-continuous}%
\end{equation}
By Theorem 2.2 in \cite{Hu-JX} and Lemma \ref{le-w-new}, we get
\[%
\begin{array}
[c]{rl}%
\mathbb{E}\left[  \sup\limits_{t\leq s\leq t+\delta}\left\vert \hat{X}%
_{s}\right\vert ^{2}\right]  & \leq C\mathbb{E}\left[  \left(
\displaystyle{\int}_{t}^{t+\delta}[|b|+|g|](s,x,W\left(  t+\delta,x\right)
,0,v_{s})ds\right)  ^{2}\right] \\
& \text{ \ }+C\mathbb{E}\left[  \displaystyle{\int}_{t}^{t+\delta}\left\vert
\sigma(s,x,W\left(  t+\delta,x\right)  ,0,v_{s})\right\vert ^{2}ds\right] \\
& \leq C\left(  1+\left\vert x\right\vert ^{2}\right)  \delta,
\end{array}
\]
which yields $\mathbb{E}[\vert\tilde{X}_{t+\delta}^{t,x;u}-x\vert] \leq
C\left(  1+\left\vert x\right\vert \right)  \delta^{\frac{1}{2}}$. Noting that
the above constant $C$ does not depend on $u$, then \ $|W(t,x)-W\left(
t+\delta,x\right)  |\leq C\left(  1+\left\vert x\right\vert \right)
\delta^{\frac{1}{2}}. $ This completes the proof.
\end{proof}

\begin{remark}
In order to prove $\mathbb{E} [\vert\tilde{X}_{t+\delta}^{t,x;v}-x\vert] \leq
C\left(  1+\left\vert x\right\vert \right)  \delta^{\frac{1}{2}}$, we
construct another FBSDE, which different from the proof in \cite{Li-W}.
Specially, we do not need additional assumption on $L_{3}$ as in \cite{Li-W}.
\end{remark}

\subsection{The value function and the HJB equation}

In this subsection, we show that the value function $W(t,x)$ defined in
(\ref{obje-eq}) is a viscosity solution to the following HJB equation
\begin{equation}
\left\{
\begin{array}
[c]{l}%
\partial_{t}W(t,x)+\inf\limits_{u\in U}H(t,x,W(t,x),DW(t,x),D^{2}W\left(
t,x\right)  ,u)=0,\\
W(T,x)=\phi(x),
\end{array}
\right.  \label{eq-hjb}%
\end{equation}
where
\begin{equation}%
\begin{array}
[c]{l}%
H(t,x,v,p,A,u)\\
=\frac{1}{2}\mathrm{tr}[\sigma\sigma^{\intercal}%
(t,x,v,V(t,x,v,p,u),u)A]+p^{\intercal}b(t,x,v,V(t,x,v,p,u),u)\\
\ \ +g(t,x,v,V(t,x,v,p,u),u),\\
V(t,x,v,p,u)=p^{\intercal}\sigma(t,x,v,V(t,x,v,p,u),u),\\
(t,x,v,p,A,u)\in\lbrack0,T]\times\mathbb{R}^{n}\times\mathbb{R}\times
\mathbb{R}^{n}\times\mathbb{S}^{n}\times U.
\end{array}
\label{def-G}%
\end{equation}

We first give the definition of viscosity solution (see
\cite{Crandall-lecture}).

\begin{definition}
(i) A real-valued continuous function $W(\cdot,\cdot)\in C\left(
[0,T]\times\mathbb{R}^{n}\right)  $ is called a viscosity subsolution (resp.
supersolution) to (\ref{eq-hjb}) if $W(T,x)\leq\phi(x)$ (resp. $W(T,x)\geq
\phi(x)$) for all $x\in\mathbb{R}^{n}$ and if for all $\varphi\in C_{b}%
^{2,3}\left(  [0,T]\times\mathbb{R}^{n}\right)  $ such that $W(t,x)=\varphi
(t,x)$ and $W-\varphi$ attains a local maximum (resp. minimum) at
$(t,x)\in\lbrack0,T)\times\mathbb{R}^{n}$, we have%
\[
\left\{
\begin{array}
[c]{l}%
\partial_{t}\varphi(t,x)+\inf\limits_{u\in U}H(t,x,\varphi(t,x),D\varphi
\left(  t,x\right)  ,D^{2}\varphi\left(  t,x\right)  ,u)\geq0\\
(\text{resp. }\varphi_{t}(t,x)+\inf\limits_{u\in U}H(t,x,\varphi
(t,x),D\varphi\left(  t,x\right)  ,D^{2}\varphi\left(  t,x\right)  ,u)\leq0).
\end{array}
\right.
\]

(ii) A real-valued continuous function $W(\cdot,\cdot)\in C\left(
[0,T]\times\mathbb{R}^{n}\right)  $ is called a viscosity solution to
(\ref{eq-hjb}), if it is both a viscosity subsolution and viscosity supersolution.
\end{definition}

In order to prove that $W(t,x)$ is a viscosity solution to the HJB equation
(\ref{eq-hjb}), we need the following assumption.

\begin{assumption}
\label{assum-l3} $L_{3}L_{W}<1$ and $8C_{4}L_{3}^{4}<1$, where $L_{W}%
=\sqrt{C_{2}}\left(  1-\sqrt{\Lambda}\right)  ^{-1}$ is the Lipschitz constant
of value function $W$ with respect to $x$, $C_{2}\,$and $C_{4}$ are defined in
Lemma \ref{sde-bsde}.
\end{assumption}

\begin{theorem}
\label{th-vis}Suppose Assumptions \ref{assum-1} and \ref{assum-l3} hold. Then
the value function $W(t,x)$ is the viscosity solution to the HJB equation
(\ref{eq-hjb}).
\end{theorem}

\textbf{Proof. }Obviously, $W\left(  T,x\right)  =\phi(x)$, $x\in
\mathbb{R}^{n}$. By Lemmas \ref{le-w-new} and \ref{le-w-t}, we know that
$W(\cdot,\cdot)\in C\left(  [0,T]\times\mathbb{R}^{n}\right)  $. We first
prove that $W$ is a viscosity subsolution. For each given $(t,x_{0})\in
\lbrack0,T)\times\mathbb{R}^{n}$, suppose $\varphi\left(  \cdot\right)  \in
C_{b}^{2,3}\left(  \left[  0,T\right]  \times\mathbb{R}^{n}\right)  $ such
that $\varphi(t,x_{0})=W(t,x_{0})$, $\varphi\geq W$ on $[0,T]\times
\mathbb{R}^{n}$ and%
\begin{equation}
L_{3}\sup_{(s,x)\in\lbrack0,T]\times\mathbb{R}^{n}}|D\varphi(s,x)|<1.
\label{new-new12345}%
\end{equation}
Consider the following FBSDE and BSDE: $\forall s\in\lbrack t,t+\delta
]\subset\lbrack0,T]$,%
\begin{equation}
\left\{
\begin{array}
[c]{rl}%
dX_{s}^{u}= & b(s,X_{s}^{u},Y_{s}^{u},Z_{s}^{u},u_{s})ds+\sigma(s,X_{s}%
^{u},Y_{s}^{u},Z_{s}^{u},u_{s})dB_{s},\\
dY_{s}^{u}= & -g(s,X_{s}^{u},Y_{s}^{u},Z_{s}^{u},u_{s})ds+Z_{s}^{u}dB_{s},\\
X_{t}^{u}= & x_{0},\ Y_{t+\delta}^{u}=\varphi(t+\delta,X_{t+\delta}^{u}),
\end{array}
\right.  \label{eq-11111}%
\end{equation}%
\begin{equation}
dY_{s}^{1,u}=-F_{1}(s,X_{s}^{u},Y_{s}^{1,u},Z_{s}^{1,u},u_{s})ds+Z_{s}%
^{1,u}dB_{s},\ Y_{t+\delta}^{1,u}=0, \label{eq-11112}%
\end{equation}
where
\[%
\begin{array}
[c]{rl}%
F_{1}\left(  s,x,y,z,u\right)  = & \partial_{t}\varphi\left(  s,x\right)
+\left(  D\varphi\left(  s,x\right)  \right)  ^{\intercal}b\left(
s,x,y+\varphi\left(  s,x\right)  ,h(s,x,y,z,u),u\right) \\
& +\frac{1}{2}\mathrm{tr}\left[  \sigma\sigma^{\intercal}\left(
s,x,y+\varphi\left(  s,x\right)  ,h(s,x,y,z,u),u\right)  D^{2}\varphi\left(
s,x\right)  \right] \\
& +g\left(  s,x,y+\varphi\left(  s,x\right)  ,h(s,x,y,z,u),u\right)  ,
\end{array}
\]%
\begin{equation}
h(s,x,y,z,u){=z+{D\varphi(s,x)}^{\intercal}\sigma}\left(  s,x,y+\varphi
(s,x\right)  ,h(s,x,y,z,u),u). \label{eq-111121}%
\end{equation}

\begin{remark}
For $\varphi\left(  \cdot\right)  \in C_{b}^{2,3}\left(  \left[  0,T\right]
\times\mathbb{R}^{n}\right)  $ such that $\varphi(t,x_{0})=W(t,x_{0})$ and
$\varphi\geq W$, we have%
\[
\varphi(t,x)-\varphi(t,x_{0})={{D\varphi(t,x_{0})}^{\intercal}(x-x}%
_{0})+o(|x-x_{0}|)
\]
and%
\[
\varphi(t,x)-\varphi(t,x_{0})\geq W(t,x)-W(t,x_{0})\geq-L_{W}|x-x_{0}|.
\]
Taking $x\rightarrow x_{0}$ such that ${{D\varphi(t,x_{0})}^{\intercal}%
(x-x}_{0})=-|{{D\varphi(t,x_{0})||}x-x}_{0}|$, we get $|{{D\varphi
(t,x_{0})|\leq}}L_{W}$. From the definition of viscosity solution and
$L_{3}L_{W}<1$, we can assume (\ref{new-new12345}) holds without loss of generality.
\end{remark}

We first prove the following lemmas.

\begin{lemma}
\label{le-h}Suppose Assumptions \ref{assum-1} and \ref{assum-l3} hold. Then
there exists a unique function $h(s,x,y,z,u)$ satisfying (\ref{eq-111121}) for
each $s\in\lbrack0,T]$, $x$, $\bar{x}\in\mathbb{R}^{n}$, $y\in\mathbb{R}$,
$z\in\mathbb{R}^{1\times d}$ and $u\in U$. Furthermore, for each given
$s\in\lbrack0,T]$, $x\in\mathbb{R}^{n}$, $y$, $\bar{y}\in\mathbb{R}$, $z$,
$\bar{z}\in\mathbb{R}^{1\times d}$ and $u\in U$,
\begin{equation}%
\begin{array}
[c]{l}%
|h(s,x,y,z,u)|\leq C(1+|x|+|y|+|z|),\\
|h(s,x,y,z,u)-h(s,\bar{x},\bar{y},\bar{z},u)|\leq C[\left(  1+\left\vert
x\right\vert +\left\vert y\right\vert +\left\vert z\right\vert \right)
\left\vert x-\bar{x}\right\vert +|y-\bar{y}|+|z-\bar{z}|],
\end{array}
\label{new-new-new-1}%
\end{equation}
and $h(\cdot)$ is continuous with respect to $s$, $x$, $y$, $z$, $u$.
\end{lemma}

\begin{proof}
For each given $s\in\lbrack0,T]$, $x\in\mathbb{R}^{n}$, $y\in\mathbb{R}$,
$z\in\mathbb{R}^{1\times d}$ and $u\in U$, we define a mapping $\Gamma
:\mathbb{R}^{1\times d}\rightarrow\mathbb{R}^{1\times d}$ as follows%
\[
\Gamma z^{\prime}={z+{D\varphi(s,x)}^{\intercal}\sigma}\left(  s,x,y+\varphi
(s,x\right)  ,z^{\prime},u)\text{ for }z^{\prime}\in\mathbb{R}^{1\times d}.
\]
For each $z_{1}$, $z_{2}\in\mathbb{R}^{1\times d}$, we have
\[%
\begin{array}
[c]{rl}
& \left\vert \Gamma z_{1}-\Gamma z_{2}\right\vert \\
& =\left\vert {\sigma}\left(  s,x,y+\varphi(s,x\right)  ,z_{1},u)^{\intercal
}{D\varphi(s,x)}-{\sigma}\left(  s,x,y+\varphi(s,x\right)  ,z_{2}%
,u)^{\intercal}{D\varphi(s,x)}\right\vert \\
& \leq L_{3}\sup_{(t,x)\in\lbrack0,T]\times\mathbb{R}^{n}}|D\varphi
(t,x)|\left\vert z_{1}-z_{2}\right\vert ,
\end{array}
\]
which implies that $\Gamma$ is a contraction mapping. Thus there exists a
unique $z^{\prime}\in\mathbb{R}^{1\times d}$ such that $\Gamma z^{\prime}%
={z}^{\prime}$. Define $h(s,x,y,z,u)=z^{\prime}$, then $h\left(  \cdot\right)
$ satisfies (\ref{eq-111121}).

For each $s\in\lbrack0,T]$, $x$, $\bar{x}\in\mathbb{R}^{n}$, $y$, $\bar{y}%
\in\mathbb{R}$, $z$, $\bar{z}\in\mathbb{R}^{1\times d}$ and $u\in U$,%
\[%
\begin{array}
[c]{l}%
|h(s,x,y,z,u)|\\
\leq|z|+L_{3}||D\varphi||_{\infty}|h(s,x,y,z,u)|+||D\varphi||_{\infty}%
|{\sigma}\left(  s,x,y+\varphi(s,x\right)  ,0,u)|\\
\leq L_{3}||D\varphi||_{\infty}|h(s,x,y,z,u)|+C(1+|x|+|y|+|z|),
\end{array}
\]%
\[%
\begin{array}
[c]{l}%
|h(s,x,y,z,u)-h(s,\bar{x},\bar{y},\bar{z},u)|\\
\leq|z-\bar{z}|+\left\vert D\varphi(s,x)-D\varphi(s,\bar{x})\right\vert
\left\vert {\sigma}\left(  s,x,y+\varphi(s,x\right)
,h(s,x,y,z,u),u)\right\vert \\
\ +||D\varphi||_{\infty}\left\vert {\sigma}\left(  s,x,y+\varphi(s,x\right)
,h(s,x,y,z,u),u)\right. \\
\text{\ \ \ \ \ \ \ \ \ \ \ \ \ \ \ }\left.  -{\sigma}\left(  s,\bar{x}%
,\bar{y}+\varphi(s,\bar{x}\right)  ,h(s,\bar{x},\bar{y},\bar{z}%
,u),u)\right\vert \\
\leq|z-\bar{z}|+||D\varphi||_{\infty}\{L_{3}|h(s,x,y,z,u)-h(s,\bar{x},\bar
{y},\bar{z},u)|+C\left(  \left\vert x-\bar{x}\right\vert +|y-\bar{y}|\right)
\}\\
\ +C\left\vert x-\bar{x}\right\vert (1+|x|+|y|+|h\left(  s,x,y,z,u\right)  |),
\end{array}
\]
which implies (\ref{new-new-new-1}). Now we prove that $h(\cdot)$ is
continuous. For $(s_{m},x_{m},y_{m},z_{m},u_{m})$ $\rightarrow(s,x,y,z,u)$,%
\[%
\begin{array}
[c]{l}%
|h(s_{m},x_{m},y_{m},z_{m},u_{m})-h(s,x,y,z,u)|\\
\leq|z_{m}-z|+L_{3}||D\varphi||_{\infty}|h(s_{m},x_{m},y_{m},z_{m}%
,u_{m})-h(s,x,y,z,u)|\\
\ \ +\left\vert {D\varphi(s}_{m}{,x}_{m}{)}^{\intercal}{\sigma}\left(
s_{m},x_{m},y_{m}+\varphi(s_{m},x_{m}\right)  ,h(s,x,y,z,u),u_{m})\right. \\
\ \ \ \ \ \left.  -{D\varphi(s,x)}^{\intercal}{\sigma}\left(  s,x,y+\varphi
(s,x\right)  ,h(s,x,y,z,u),u)\right\vert ,
\end{array}
\]
which implies that $h(\cdot)$ is continuous with respect to $s$, $x$, $y$,
$z$, $u$.
\end{proof}

\begin{lemma}
\label{le-1app} For each $s\in\lbrack t,t+\delta]$, we have%
\begin{equation}%
\begin{array}
[c]{l}%
Y_{s}^{1,u}=Y_{s}^{u}-\varphi\left(  s,X_{s}^{u}\right)  ,\ Z_{s}^{1,u}%
{=Z_{s}^{u}-{D\varphi\left(  s,X_{s}^{u}\right)  }^{\intercal}\sigma\left(
s,X_{s}^{u},Y_{s}^{u},Z_{s}^{u},u_{s}\right)  .}%
\end{array}
\label{eq-11113}%
\end{equation}

\end{lemma}

\begin{proof}
Applying It\^{o}'s formula to $Y_{s}^{u}-\varphi\left(  s,X_{s}^{u}\right)  $,
we can obtain the desired result.
\end{proof}

Under Assumption \ref{assum-l3}, we can choose a $\delta_{0}>0$ such that
\[
8C_{4}\left[  L_{2}^{4}\left(  \delta_{0}^{2}+\delta_{0}^{4}\right)
+L_{3}^{4}\right]  <1.
\]
By Theorem \ref{th-lp} in Appendix, then for each $\delta<\delta_{0}$, we have%
\begin{equation}%
\begin{array}
[c]{l}%
\mathbb{E}\left[  \sup\limits_{t\leq s\leq t+\delta}\left(  |X_{s}^{u}%
|^{4}+|Y_{s}^{u}|^{4}\right)  +\left(  \displaystyle{\int}_{t}^{t+\delta
}|Z_{s}^{u}|^{2}ds\right)  ^{2}\right] \\
\leq C\left\{  |x_{0}|^{4}+\mathbb{E}\left[  \left(  \displaystyle{\int}%
_{t}^{t+\delta}(|b(s,0,0,0,u_{s})|+|g(s,0,0,0,u_{s})|)ds\right)  ^{4}\right.
\right. \\
\ \ \ \ \left.  \left.  +\left(  \displaystyle{\int}_{t}^{t+\delta}%
|\sigma(s,0,0,0,u_{s})|^{2}ds\right)  ^{2}\right]  \right\} \\
\leq C\left(  1+\left\vert x_{0}\right\vert ^{4}\right)  ,
\end{array}
\label{new-new-23452}%
\end{equation}
where $C$ is a constant independent of $u$ and $\delta$. Consider the
following BSDE: $\forall s\in\lbrack t,t+\delta]$,%
\begin{equation}
dY_{s}^{2,u}=-F_{1}(s,x_{0},0,0,u_{s})ds+Z_{s}^{2,u}dB_{s},\ Y_{t+\delta
}^{2,u}=0. \label{eq-new-12345}%
\end{equation}

We have the following estimate.

\begin{lemma}
\label{le-129-1} For each given $v\in\mathcal{U}^{t}[t,t+\delta]$, \ we have
$\vert Y_{t}^{1,v}-Y_{t}^{2,v}\vert\leq C\delta^{\frac{3}{2}},$
where $C$ is a positive constant depend on $x$ and independent of $v$,
$\delta$.
\end{lemma}

\begin{proof}
Since%
\begin{equation}%
\begin{array}
[c]{l}%
Y_{t}^{1,v}=\mathbb{E}\left[  \int_{t}^{t+\delta}F_{1}(s,X_{s}^{v},Y_{s}%
^{1,v},Z_{s}^{1,v},v_{s})ds\right]  ,\\
Y_{t}^{2,v}=\mathbb{E}\left[  \int_{t}^{t+\delta}F_{1}(s,x_{0},0,0,v_{s}%
)ds\right]  ,
\end{array}
\label{ne-new-2}%
\end{equation}
we obtain%
\[
\left\vert Y_{t}^{1,v}-Y_{t}^{2,v}\right\vert \leq\mathbb{E}\left[  \int%
_{t}^{t+\delta}\hat{F}_{s}ds\right]  ,
\]
where
\[
\hat{F}_{s}=\left\vert F_{1}(s,X_{s}^{v},Y_{s}^{1,v},Z_{s}^{1,v},v_{s}%
)-F_{1}(s,x_{0},0,0,v_{s})\right\vert .
\]
Note that $\varphi\in C_{b}^{2,3}\left(  [0,T]\times\mathbb{R}^{n}\right)  $.
By Lemmas \ref{le-h} and \ref{le-1app}, it is easy to check that
\[
\hat{F}_{s}\leq C\left(  \left\vert X_{s}^{v}-x_{0}\right\vert +|Y_{s}%
^{1,v}|+|Z_{s}^{1,v}|+\left\vert X_{s}^{v}-x_{0}\right\vert ^{2}+|Y_{s}%
^{1,v}|^{2}+|Z_{s}^{1,v}|^{2}\right)  .
\]
Set $\tilde{X}_{s}^{v}=X_{s}^{v}-x_{0}$, $\tilde{Y}_{s}^{v}=Y_{s}^{1,v}$,
$\tilde{Z}_{s}^{v}=Z_{s}^{v}$. Then $(\tilde{X}_{s}^{v},\tilde{Y}_{s}%
^{v},\tilde{Z}_{s}^{v})$ satisfies the following FBSDE:%
\[
\left\{
\begin{array}
[c]{rl}%
d\tilde{X}_{s}^{v}= & b(s,\tilde{X}_{s}^{v}+x_{0},\tilde{Y}_{s}^{v}%
+\varphi\left(  s,X_{s}^{v}\right)  ,\tilde{Z}_{s}^{v},v_{s})ds\\
& +\sigma(s,\tilde{X}_{s}^{v}+x_{0},\tilde{Y}_{s}^{v}+\varphi\left(
s,X_{s}^{v}\right)  ,\tilde{Z}_{s}^{v},v_{s})dB_{s},\\
d\tilde{Y}_{s}^{v}= & -g(s,\tilde{X}_{s}^{v}+x_{0},\tilde{Y}_{s}^{v}%
+\varphi\left(  s,X_{s}^{v}\right)  ,\tilde{Z}_{s}^{v},v_{s})ds+\tilde{Z}%
_{s}^{v}dB_{s},\\
\tilde{X}_{t}^{v}= & 0,\ \tilde{Y}_{t+\delta}^{v}=0.
\end{array}
\right.
\]
By (\ref{new-new-23452}) and Theorem \ref{th-lp} in Appendix, we have%
\begin{equation}%
\begin{array}
[c]{rl}
& \mathbb{E}\left[  \sup\limits_{t\leq s\leq t+\delta}\left(  |\tilde{X}%
_{s}^{v}|^{p}+|\tilde{Y}_{s}^{v}|^{p}\right)  +\left(  {\int}_{t}^{t+\delta
}|Z_{s}^{v}|^{2}ds\right)  ^{\frac{p}{2}}\right] \\
& \leq C\left(  1+\mathbb{E}\left[  \sup\limits_{t\leq s\leq t+\delta}%
|X_{s}^{v}|^{p}\right]  \right)  \delta^{\frac{p}{2}}\\
& \leq C\delta^{\frac{p}{2}},
\end{array}
\label{ne-new-3}%
\end{equation}
where $p\in\lbrack2,4]$. Thus%
\begin{equation}
\mathbb{E}\left[  \int_{t}^{t+\delta}\left(  \left\vert X_{s}^{v}%
-x_{0}\right\vert +|Y_{s}^{1,v}|+\left\vert X_{s}^{v}-x_{0}\right\vert
^{2}+|Y_{s}^{1,v}|^{2}\right)  ds\right]  \leq C\delta^{\frac{3}{2}}.
\label{ne-new-1}%
\end{equation}
On the other hand, by (\ref{eq-11112}) and (\ref{ne-new-2}), we have%
\begin{align*}
\mathbb{E}\left[  {\int}_{t}^{t+\delta}|Z_{s}^{1,v}|^{2}ds\right]   &
=|Y_{t}^{1,v}|^{2}+\mathbb{E}\left[  \left(  {\int}_{t}^{t+\delta}%
F_{1}(s,X_{s}^{v},Y_{s}^{1,v},Z_{s}^{1,v},v_{s})ds\right)  ^{2}\right] \\
&  \leq2\mathbb{E}\left[  \left(  {\int}_{t}^{t+\delta}|F_{1}(s,X_{s}%
^{v},Y_{s}^{1,v},Z_{s}^{1,v},v_{s})|ds\right)  ^{2}\right]  .
\end{align*}
It is easy to check that%
\[
|F_{1}(s,X_{s}^{v},Y_{s}^{1,v},Z_{s}^{1,v},u_{s})|\leq C(1+|X_{s}^{v}%
|^{2}+|Y_{s}^{v}|^{2}+|Z_{s}^{v}|^{2}).
\]
Thus, by (\ref{new-new-23452}) and (\ref{ne-new-3}), we obtain%
\[
\mathbb{E}\left[  {\int}_{t}^{t+\delta}|Z_{s}^{1,v}|^{2}ds\right]  \leq
C\left(  \delta^{2}+\mathbb{E}\left[  \left(  {\int}_{t}^{t+\delta}|Z_{s}%
^{v}|^{2}ds\right)  ^{2}\right]  \right)  \leq C\delta^{2}.
\]
Since
\[
\mathbb{E}\left[  {\int}_{t}^{t+\delta}|Z_{s}^{1,v}|ds\right]  \leq\left(
\mathbb{E}\left[  {\int}_{t}^{t+\delta}|Z_{s}^{1,v}|^{2}ds\right]  \right)
^{\frac{1}{2}}\delta^{\frac{1}{2}}\leq C\delta^{\frac{3}{2}},
\]
we obtain the desired result.
\end{proof}

Now we compute $\inf_{v\in\mathcal{U}^{t}[t,t+\delta]}Y_{t}^{2,v}$.

\begin{lemma}
\label{le-inf}We have\ $Y_{t}^{0}=\inf_{v\in\mathcal{U}^{t}[t,t+\delta]}%
Y_{t}^{2,v},$ where $Y_{t}^{0}$ is the solution to the following ordinary
differential equation:
\[%
\begin{array}
[c]{rl}%
dY_{s}^{0}=-F_{0}\left(  s,x_{0}\right)  ds{,}\ Y_{t+\delta}^{0}=0,\text{
\ }s\in\lbrack t,t+\delta], &
\end{array}
\]
where \ $F_{0}\left(  s,x_{0}\right)  =\inf_{u\in U}F_{1}\left(
s,x_{0},0,0,u\right)  .$
\end{lemma}

\begin{proof}
For each given $v\in\mathcal{U}^{t}[t,t+\delta]$, $F_{1}\left(  s,x_{0}%
,0,0,v_{s}\right)  \geq F_{0}\left(  s,x_{0}\right)  $, by comparison theorem
of BSDE, we get $Y_{t}^{2,v}\geq Y_{t}^{0}$. On the other hand, we can choose
a deterministic control $\mu$ in $\mathcal{U}^{t}[t,t+\delta]$ such
that\ $F_{0}\left(  s,x_{0}\right)  =F_{1}\left(  s,x_{0},0,0,\mu_{s}\right)
.$ It is clear that $Y_{t}^{2,\mu}=Y_{t}^{0}$. Thus we obtain the desired result.
\end{proof}

By Theorem \ref{th-ddp}, we have%
\[
W\left(  t,x_{0}\right)  =\underset{v\in\mathcal{U}^{t}[t,t+\delta]}{\inf
}G_{t,t+\delta}^{t,x_{0};v}\left[  W(t+\delta,\tilde{X}_{t+\delta}^{t,x_{0}%
;v})\right]  .
\]
Since $\varphi\left(  t+\delta,\cdot\right)  \geq W\left(  t+\delta
,\cdot\right)  $, by Theorem \ref{th-comp} in Appendix, we get $Y_{t}^{v}\geq
W\left(  t,x_{0}\right)  $ for each $v\in\mathcal{U}^{t}[t,t+\delta]$, which
implies
\[
\underset{v\in\mathcal{U}^{t}[t,t+\delta]}{\inf}\left[  Y_{t}^{v}%
-\varphi\left(  t,x_{0}\right)  \right]  =\underset{v\in\mathcal{U}%
^{t}[t,t+\delta]}{\inf}Y_{t}^{1,v}\geq0\text{.}%
\]
By Lemma \ref{le-129-1}, we deduce\ $\inf_{v\in\mathcal{U}^{t}[t,t+\delta
]}Y_{t}^{2,v}\geq-C\delta^{\frac{3}{2}}.$ It yields that $Y_{t}^{0}%
\geq-C\delta^{\frac{3}{2}}$ by Lemma \ref{le-inf}. Thus \ $-C\delta^{\frac
{1}{2}}\leq\frac{1}{\delta}Y_{t}^{0}=\frac{1}{\delta}\int_{t}^{t+\delta}%
F_{0}\left(  s,x_{0}\right)  ds.$ Letting $\delta\rightarrow0$, we get
$F_{0}\left(  t,x_{0}\right)  \geq0$, which implies that $W\ $is a viscosity
subsolution. By the same method, we can prove that $W$ is a viscosity
supersolution. Thus $W\ $is a viscosity solution. $\square$

\begin{remark}
\label{re-mon} Note that Assumption \ref{assum-1} (ii) is only used to
guarantees the well-posedness of our fully coupled forward-backward controlled
system. In fact, following our approach, the readers may verify that all the
results in Section 3 still hold under Assumptions \ref{assum-1} (i),
\ref{assum-l3} and the following monotonicity conditions.
\end{remark}

Given a nonzero $G\in\mathbb{R}^{1\times n}$, define
\[
\lambda=(x,y,z)^{\intercal},\ A\left(  t,\lambda,u\right)  =(-G^{\intercal}g,
Gb, G\sigma)^{\intercal}(t,\lambda,u).
\]


\begin{assumption}
\label{assm-mon} (Monotonicity conditions)\newline(i) $\left\langle A\left(
t,\lambda,u\right)  -A\left(  t,\bar{\lambda},u\right)  ,\lambda-\bar{\lambda
}\right\rangle \leq-\beta_{1}\left\vert G\hat{x}\right\vert ^{2}-\beta
_{2}\left(  \left\vert G^{\intercal}\hat{y}\right\vert ^{2}+\left\vert
G^{\intercal}\hat{z}\right\vert ^{2}\right)  $, for $u\in U$;\newline(ii)
$\left\langle \phi\left(  x\right)  -\phi\left(  \bar{x}\right)  ,G\hat
{x}\right\rangle \geq\mu_{1}\left\vert G\hat{x}\right\vert ^{2}$, where
$\hat{x}=x-\bar{x}$, $\hat{y}=y-\bar{y}$, $\hat{z}=z-\bar{z}$, $\beta_{1}$,
$\beta_{2}$, $\mu_{1}$ are given nonnegative constants with $\beta_{1}%
+\beta_{2}>0$, $\beta_{2}+\mu_{1}>0$. Moreover, $\beta_{2}>0$ when $n>1$.
\end{assumption}

\section{The uniqueness of viscosity solutions}

In this section, we study the uniqueness of the viscosity solution to the HJB
equation (\ref{eq-hjb}).

\subsection{$\sigma$ independent of $y$ and $z$}

In this case, the corresponding HJB equation becomes
\begin{equation}
\left\{
\begin{array}
[c]{l}%
\partial_{t}W(t,x)+\inf\limits_{u\in U}H(t,x,W(t,x),DW(t,x),D^{2}%
W(t,x),u)=0,\\
W(T,x)=\phi(x),
\end{array}
\right.  \label{eq-hjb-yz}%
\end{equation}
where%
\[%
\begin{array}
[c]{l}%
H(t,x,v,p,A,u)\\
=\frac{1}{2}\mathrm{tr}[\sigma\sigma^{\intercal}(t,x,u)A]+p^{\intercal
}b(t,x,v,p^{\intercal}\sigma(t,x,u),u)+g(t,x,v,p^{\intercal}\sigma
(t,x,u),u),\\
(t,x,v,p,A,u)\in\lbrack0,T]\times\mathbb{R}^{n}\times\mathbb{R}\times
\mathbb{R}^{n}\times\mathbb{S}^{n}\times U.
\end{array}
\]

We adopt the approach in Barles, Buckdahn and Pardoux \cite{Baeles-BP} (see
also Wu and Yu \cite{Wu-Y}) to prove the uniqueness of the viscosity solution
to (\ref{eq-hjb-yz}) in the following theorem. Note that applying the approach
in \cite{Baeles-BP}, Buckdahn and Li \cite{Buckdahn-Li} studied a decoupled
case and Wu and Yu \cite{Wu-Y} obtained the uniqueness result for coefficients
which are independent of $u$.

\begin{theorem}
\label{th-vis-uni}Suppose that $\sigma$ is independent of $y$ and $z$ and
Assumption \ref{assum-1} (i) holds. Then there exists at most one viscosity
solution to (\ref{eq-hjb-yz}) in the class of continuous functions which are
Lipschitz continuous with respect to $x$.
\end{theorem}

See Subsection \ref{subse-pf} in Appendix for the details.

\subsection{$\sigma$ depends on $y$ and $z$}

Wu and Yu \cite{Wu-Y} studied a PDE system for which the coefficient $\sigma$
of the corresponding FBSDE satisfies $\sigma\left(  t,x,y,0\right)  =0$. Under
this assumption, the fully coupled FBSDE degenerates to a forward-backward
ordinary differential equation and the PDE system degenerates to a first order
PDE. Thus, for this case, the uniqueness result is implied by Theorem
\ref{th-vis-uni}.

In this subsection, we study the HJB equation in which $\sigma$ is dependent
on $y$ and $z$. As pointed out in Remark \ref{re-appen}, the method in
\cite{Baeles-BP} does not work. We first give the following proposition.

\begin{proposition}
\label{pr-um}Suppose $\sigma$ is independent of $y$ and $z$; and one of the
following two conditions holds true:

\begin{description}
\item[(i)] Assumption \ref{assum-1} holds;

\item[(ii)] Assumptions \ref{assum-1} (i) and \ref{assm-mon} hold.
\end{description}

\noindent Let $W$ be the value function. Then, for each $(t,x)\in
\lbrack0,T)\times\mathbb{R}^{n}$, we can find a sequence $u^{m}\in
\mathcal{U}^{t}[t,T]$ such that
\[
\mathbb{E}\left[  {\int}_{t}^{T}\left\vert Y_{s}^{t,x;u^{m}}-W\left(
s,X_{s}^{t,x;u^{m}}\right)  \right\vert ^{2}ds\right]  +\left\vert
Y_{t}^{t,x;u^{m}}-W\left(  t,x\right)  \right\vert \rightarrow0\text{ as
}m\rightarrow\infty.
\]

\end{proposition}

\begin{proof}
We only prove the first case (the condition (i) holds). The proof for the
second case is similar. The proof is divided into three Steps.

\textbf{Step 1.} For each given integer $m\geq1$, set $t_{i}^{m}%
=t+i(T-t)m^{-1}$ for $i=0,\ldots,m$. We want to choose a $u^{i,m}%
\in\mathcal{U}^{t}[t_{i}^{m},t_{i+1}^{m}]$ such that%
\begin{equation}
|Y_{t_{i}^{m}}^{i,m}-W\left(  t_{i}^{m},X_{t_{i}^{m}}^{i-1,m}\right)
|\leq\frac{1}{m^{2}(2C+1)^{m}}\text{ for }i=0,\ldots,m-1, \label{new-lx-2}%
\end{equation}
where $C$ is given in Lemma \ref{le-w-new}, $X_{t_{0}^{m}}^{-1,m}=x$ and
$(X^{i,m},Y^{i,m},Z^{i,m})$ is the solution to following FBSDE:%
\begin{equation}
\left\{
\begin{array}
[c]{rl}%
dX_{s}^{i,m}= & b(s,X_{s}^{i,m},Y_{s}^{i,m},Z_{s}^{i,m},u_{s}^{i,m}%
)ds+\sigma(s,X_{s}^{i,m},u_{s}^{i,m})dB_{s},\\
dY_{s}^{i,m}= & -g(s,X_{s}^{i,m},Y_{s}^{i,m},Z_{s}^{i,m},u_{s}^{i,m}%
)ds+Z_{s}^{i,m}dB_{s},\text{ }s\in\lbrack t_{i}^{m},t_{i+1}^{m}],\\
X_{t_{i}^{m}}^{i,m}= & X_{t_{i}^{m}}^{i-1,m},\ Y_{t_{i+1}^{m}}^{i,m}%
=W(t_{i+1}^{m},X_{t_{i+1}^{m}}^{i,m}).
\end{array}
\right.  \label{new-lx-3}%
\end{equation}
Precisely, on the interval $[t_{0}^{m},t_{1}^{m}]$, by Theorem \ref{th-ddp},
we can choose a $u^{0,m}\in\mathcal{U}^{t}[t_{0}^{m},t_{1}^{m}]$ such that%
\[
\left\vert Y_{t_{0}^{m}}^{0,m}-W\left(  t_{0}^{m},x\right)  \right\vert
\leq\frac{1}{m^{2}(2C+1)^{m+1}},
\]
where $(X^{0,m},Y^{0,m},Z^{0,m})$ is the solution to FBSDE (\ref{new-lx-3})
for $i=0$. On the interval $[t_{1}^{m},t_{2}^{m}]$, we first choose a
partition $\{A_{j}^{m}:j\geq1\}$ of $\mathbb{R}^{n}$ and $x_{j}^{m}%
\in\mathbb{R}^{n}$ such that%
\[
\left\vert \sum_{j=1}^{\infty}x_{j}^{m}I_{A_{j}^{m}}(X_{t_{1}^{m}}%
^{0,m})-X_{t_{1}^{m}}^{0,m}\right\vert \leq\frac{1}{m^{2}(2C+1)^{m+1}}.
\]
By Theorem \ref{th-ddp}, for each $x_{j}^{m}$, we can choose a $u^{1,j,m}%
\in\mathcal{U}^{t}[t_{1}^{m},t_{2}^{m}]$ such that%
\[
\left\vert Y_{t_{1}^{m}}^{1,j,m}-W\left(  t_{1}^{m},x_{j}^{m}\right)
\right\vert \leq\frac{1}{m^{2}(2C+1)^{m+1}},
\]
where $(X^{1,j,m},Y^{1,j,m},Z^{1,j,m})$ is the solution to the following
FBSDE:%
\begin{equation}
\left\{
\begin{array}
[c]{rl}%
dX_{s}^{1,j,m}= & b(s,X_{s}^{1,j,m},Y_{s}^{1,j,m},Z_{s}^{1,j,m},u_{s}%
^{1,j,m})ds+\sigma(s,X_{s}^{1,j,m},u_{s}^{1,j,m})dB_{s},\\
dY_{s}^{1,j,m}= & -g(s,X_{s}^{1,j,m},Y_{s}^{1,j,m},Z_{s}^{1,j,m},u_{s}%
^{1,j,m})ds+Z_{s}^{1,j,m}dB_{s},\text{ }s\in\lbrack t_{1}^{m},t_{2}^{m}],\\
X_{t_{1}^{m}}^{1,j,m}= & x_{j}^{m},\ Y_{t_{2}^{m}}^{1,j,m}=W(t_{2}%
^{m},X_{t_{2}^{m}}^{1,j,m}).
\end{array}
\right.  \label{new-lx-4}%
\end{equation}
Set
\begin{equation}
u_{s}^{1,m}=\sum_{j=1}^{\infty}I_{A_{j}^{m}}(X_{t_{1}^{m}}^{0,m})u_{s}%
^{1,j,m}\in\mathcal{U}^{t}[t_{1}^{m},t_{2}^{m}]\;\text{for }s\in\lbrack
t_{1}^{m},t_{2}^{m}].
\end{equation}
By Theorem 2.2 in \cite{Hu-JX} for (\ref{new-lx-3}) and (\ref{new-lx-4}),%
\[
\left\vert Y_{t_{1}^{m}}^{1,m}-Y_{t_{1}^{m}}^{1,j,m}\right\vert I_{A_{j}^{m}%
}(X_{t_{1}^{m}}^{0,m})\leq C\left\vert X_{t_{1}^{m}}^{0,m}-x_{j}%
^{m}\right\vert I_{A_{j}^{m}}(X_{t_{1}^{m}}^{0,m}).
\]
Thus, by Lemma \ref{le-w-new},%
\[%
\begin{array}
[c]{l}%
\left\vert Y_{t_{1}^{m}}^{1,m}-W\left(  t_{1}^{m},X_{t_{1}^{m}}^{0,m}\right)
\right\vert \\
\leq\left\vert Y_{t_{1}^{m}}^{1,m}-\sum_{j=1}^{\infty}Y_{t_{1}^{m}}%
^{1,j,m}I_{A_{j}^{m}}(X_{t_{1}^{m}}^{0,m})\right\vert +\left\vert \sum
_{j=1}^{\infty}Y_{t_{1}^{m}}^{1,j,m}I_{A_{j}^{m}}(X_{t_{1}^{m}}^{0,m}%
)-W\left(  t_{1}^{m},X_{t_{1}^{m}}^{0,m}\right)  \right\vert \\
\leq C\left\vert X_{t_{1}^{m}}^{0,m}-\sum_{j=1}^{\infty}x_{j}^{m}I_{A_{j}^{m}%
}(X_{t_{1}^{m}}^{0,m})\right\vert +\frac{1}{m^{2}(2C+1)^{m+1}}\\
\ \ +\left\vert W\left(  t_{1}^{m},\sum_{j=1}^{\infty}x_{j}^{m}I_{A_{j}^{m}%
}(X_{t_{1}^{m}}^{0,m})\right)  -W\left(  t_{1}^{m},X_{t_{1}^{m}}^{0,m}\right)
\right\vert \\
\leq\frac{1}{m^{2}(2C+1)^{m}}.
\end{array}
\]
Similarly, we can choose the desired $u^{i,m}\in\mathcal{U}^{t}[t_{i}%
^{m},t_{i+1}^{m}]$ for $i=2,\ldots,m-1$.

\textbf{Step 2.} Set \ $
u_{s}^{m} =\sum_{i=0}^{m-1}u_{s}^{i,m}I_{[t_{i}^{m},t_{i+1}^{m})}(s),\text{
}X_{s}^{m}=\sum_{i=0}^{m-1}X_{s}^{i,m}I_{[t_{i}^{m},t_{i+1}^{m})}(s),\newline
Y_{s}^{m} =\sum_{i=0}^{m-1}Y_{s}^{i,m}I_{[t_{i}^{m},t_{i+1}^{m})}(s),\text{
}Z_{s}^{m}=\sum_{i=0}^{m-1}Z_{s}^{i,m}I_{[t_{i}^{m},t_{i+1}^{m})}(s),$
where $u^{i,m}$, $X^{i,m}$, $Y^{i,m}$ and $Z^{i,m}$ are given in Step 1. It is
easy to check that $X^{m}$ satisfies the following SDE:%
\[
\left\{
\begin{array}
[c]{rl}%
dX_{s}^{m}= & b(s,X_{s}^{m},Y_{s}^{m},Z_{s}^{m},u_{s}^{m})ds+\sigma
(s,X_{s}^{m},u_{s}^{m})dB_{s},\\
X_{t}^{m}= & x,\text{ }s\in\lbrack t,T].
\end{array}
\right.
\]
Let $(X^{m},\tilde{Y}^{m},\tilde{Z}^{m})$ be the solution to the following
decoupled FBSDE:%
\begin{equation}
\left\{
\begin{array}
[c]{rl}%
dX_{s}^{m}= & b(s,X_{s}^{m},Y_{s}^{m},Z_{s}^{m},u_{s}^{m})ds+\sigma
(s,X_{s}^{m},u_{s}^{m})dB_{s},\\
d\tilde{Y}_{s}^{m}= & -g(s,X_{s}^{m},\tilde{Y}_{s}^{m},\tilde{Z}_{s}^{m}%
,u_{s}^{m})ds+\tilde{Z}_{s}^{m}dB_{s},\text{ }s\in\lbrack t,T],\\
X_{t}^{m}= & x,\ \tilde{Y}_{T}^{m}=\phi(X_{T}^{m}).
\end{array}
\right.  \label{new-lx-8}%
\end{equation}
In the followings, we want to prove%
\begin{equation}%
\begin{array}
[c]{rl}
& \mathbb{E}\left[  {\int}_{t}^{T}\left(  \left\vert \tilde{Y}_{s}%
^{m}-W\left(  s,X_{s}^{m}\right)  \right\vert ^{2}+|\tilde{Y}_{s}^{m}%
-Y_{s}^{m}|^{2}+|\tilde{Z}_{s}^{m}-Z_{s}^{m}|^{2}\right)  ds\right] \\
& \ \ +|\tilde{Y}_{t}^{m}-W\left(  t,x\right)  |\rightarrow0\text{ as
}m\rightarrow\infty.
\end{array}
\label{new-lx-9}%
\end{equation}
Note that $(X^{m},\tilde{Y}^{m},\tilde{Z}^{m})$ satisfies the following FBSDE
on $[t_{i}^{m},t_{i+1}^{m}]$:%
\begin{equation}
\left\{
\begin{array}
[c]{rl}%
dX_{s}^{m}= & b(s,X_{s}^{m},Y_{s}^{m},Z_{s}^{m},u_{s}^{m})ds+\sigma
(s,X_{s}^{m},u_{s}^{m})dB_{s},,\\
d\tilde{Y}_{s}^{m}= & -g(s,X_{s}^{m},\tilde{Y}_{s}^{m},\tilde{Z}_{s}^{m}%
,u_{s}^{m})ds+\tilde{Z}_{s}^{m}dB_{s},\text{ }s\in\lbrack t_{i}^{m}%
,t_{i+1}^{m}],\\
X_{t_{i}^{m}}^{m}= & X_{t_{i}^{m}}^{i-1,m},\ \tilde{Y}_{t_{i+1}^{m}}%
^{m}=\tilde{Y}_{t_{i+1}^{m}}^{m}.
\end{array}
\right.  \label{new-lx-10}%
\end{equation}
Then by Theorem 2.2 in \cite{Hu-JX} for FBSDEs (\ref{new-lx-3}) and
(\ref{new-lx-10}),%
\begin{equation}%
\begin{array}
[c]{rl}
& \mathbb{E}\left[  \left.  \sup\limits_{t_{i}^{m}\leq s\leq t_{i+1}^{m}%
}\left\vert \tilde{Y}_{s}^{m}-Y_{s}^{m}\right\vert ^{2}+\int_{t_{i}^{m}%
}^{t_{i+1}^{m}}|\tilde{Z}_{s}^{m}-Z_{s}^{m}|^{2}ds\right\vert \mathcal{F}%
_{t_{i}^{m}}\right] \\
& \leq C^{2}\mathbb{E}\left[  \left.  \left\vert W(t_{i+1}^{m},X_{t_{i+1}^{m}%
}^{m})-\tilde{Y}_{t_{i+1}^{m}}^{m}\right\vert ^{2}\right\vert \mathcal{F}%
_{t_{i}^{m}}\right]  ,
\end{array}
\label{new-lx-11}%
\end{equation}
where $C$ is the same as in Step 1. It yields that%
\begin{equation}%
\begin{array}
[c]{rl}
& \left\vert \tilde{Y}_{t_{i}^{m}}^{m}-W(t_{i}^{m},X_{t_{i}^{m}}%
^{m})\right\vert \\
& \leq\left\vert Y_{t_{i}^{m}}^{m}-W(t_{i}^{m},X_{t_{i}^{m}}^{m})\right\vert
+C\left\{  \mathbb{E}\left[  \left.  \left\vert W(t_{i+1}^{m},X_{t_{i+1}^{m}%
}^{m})-\tilde{Y}_{t_{i+1}^{m}}^{m}\right\vert ^{2}\right\vert \mathcal{F}%
_{t_{i}^{m}}\right]  \right\}  ^{\frac{1}{2}}.
\end{array}
\label{new-lx-12}%
\end{equation}
Note that (\ref{new-lx-2}), (\ref{new-lx-12}), $W(t_{m}^{m},X_{t_{m}^{m}}%
^{m})=\phi(X_{T}^{m})$ and Lemma \ref{le-w-new}. By doing estimate
recursively, we obtain%
\begin{equation}
\left\vert \tilde{Y}_{t_{i}^{m}}^{m}-W(t_{i}^{m},X_{t_{i}^{m}}^{m})\right\vert
\leq\frac{1+C+\cdots+C^{m-i-1}}{m^{2}(2C+1)^{m}}\leq\frac{1}{m}.
\label{new-lx-13}%
\end{equation}
Combining (\ref{new-lx-11}) and (\ref{new-lx-13}), we get%
\begin{equation}%
\begin{array}
[c]{rl}%
\mathbb{E}\left[  {\int}_{t}^{T}\left(  |\tilde{Y}_{s}^{m}-Y_{s}^{m}%
|^{2}+|\tilde{Z}_{s}^{m}-Z_{s}^{m}|^{2}\right)  ds\right]  & \leq C^{2}%
\sum\limits_{i=0}^{m-1}\mathbb{E}\left[  \left\vert W(t_{i+1}^{m}%
,X_{t_{i+1}^{m}}^{m})-\tilde{Y}_{t_{i+1}^{m}}^{m}\right\vert ^{2}\right] \\
& \leq\frac{C^{2}}{m}.
\end{array}
\label{new-lx-14}%
\end{equation}
By Theorem 2.2 in \cite{Hu-JX} for FBSDE (\ref{new-lx-8}) and (\ref{new-lx-14}%
), we obtain that%
\begin{equation}
\mathbb{E}\left[  \sup\limits_{t\leq s\leq T}|X_{s}^{m}|^{2}+{\int}_{t}%
^{T}\left(  |Y_{s}^{m}|^{2}+|Z_{s}^{m}|^{2}+|\tilde{Y}_{s}^{m}|^{2}+|\tilde
{Z}_{s}^{m}|^{2}\right)  ds\right]  \leq C^{\prime}(1+|x|^{2}),
\label{new-lx-15}%
\end{equation}
where $C^{\prime}$ is a constant which is independent of $m$. For $s\in\lbrack
t_{i}^{m},t_{i+1}^{m}]$, by (\ref{new-lx-15}) and Lemma \ref{le-w-new}, we
have
\begin{align*}
\mathbb{E}\left[  \left\vert \tilde{Y}_{s}^{m}-\tilde{Y}_{t_{i+1}^{m}}%
^{m}\right\vert ^{2}\right]   &  =\mathbb{E}\left[  \left\vert \mathbb{E}%
\left[  \left.  \int_{s}^{t_{i+1}^{m}}g(r,X_{r}^{m},\tilde{Y}_{r}^{m}%
,\tilde{Z}_{r}^{m},u_{r}^{m})dr\right\vert \mathcal{F}_{s}\right]  \right\vert
^{2}\right] \\
&  \leq\frac{T}{m}\mathbb{E}\left[  \int_{t_{i}^{m}}^{t_{i+1}^{m}}%
|g(r,X_{r}^{m},\tilde{Y}_{r}^{m},\tilde{Z}_{r}^{m},u_{r}^{m})|^{2}dr\right] \\
&  \leq\bar{C}(1+|x|^{2})\frac{1}{m},
\end{align*}%
\[%
\begin{array}
[c]{l}%
\mathbb{E}\left[  \left\vert X_{s}^{m}-X_{t_{i+1}^{m}}^{m}\right\vert
^{2}\right] \\
\leq2\mathbb{E}\left[  \left(  \int_{t_{i}^{m}}^{t_{i+1}^{m}}|b(r,X_{r}%
^{m},Y_{r}^{m},Z_{r}^{m},u_{r}^{m})|dr\right)  ^{2}+\int_{t_{i}^{m}}%
^{t_{i+1}^{m}}|\sigma(r,X_{r}^{m},u_{r}^{m})|^{2}dr\right] \\
\leq\frac{2T}{m}\mathbb{E}\left[  \int_{t_{i}^{m}}^{t_{i+1}^{m}}|b(r,X_{r}%
^{m},Y_{r}^{m},Z_{r}^{m},u_{r}^{m})|^{2}dr\right]  +\frac{2T}{m}%
\mathbb{E}\left[  \left(  L+L_{1}\sup\limits_{t_{i}^{m}\leq r\leq t_{i+1}^{m}%
}|X_{r}^{m}|\right)  ^{2}\right] \\
\leq\bar{C}(1+|x|^{2})\frac{1}{m},
\end{array}
\]%
\[%
\begin{array}
[c]{rl}
& \mathbb{E}\left[  \left\vert \tilde{Y}_{s}^{m}-W\left(  s,X_{s}^{m}\right)
\right\vert ^{2}\right] \\
& \leq\mathbb{E}\left[  \left\vert \tilde{Y}_{s}^{m}-\tilde{Y}_{t_{i+1}^{m}%
}^{m}+\tilde{Y}_{t_{i+1}^{m}}^{m}-W(t_{i+1}^{m},X_{t_{i+1}^{m}}^{m}%
)+W(t_{i+1}^{m},X_{t_{i+1}^{m}}^{m})-W\left(  s,X_{s}^{m}\right)  \right\vert
^{2}\right] \\
& \leq\bar{C}(1+|x|^{2})\frac{1}{m},
\end{array}
\]
where $\bar{C}$ is a constant which is independent of $m$. Thus%
\begin{equation}%
\begin{array}
[c]{rl}%
\mathbb{E}\left[  {\int}_{t}^{T}\left\vert \tilde{Y}_{s}^{m}-W\left(
s,X_{s}^{m}\right)  \right\vert ^{2}ds\right]  & \leq\sum_{i=0}^{m-1}%
\int_{t_{i}^{m}}^{t_{i+1}^{m}}\mathbb{E}\left[  \left\vert \tilde{Y}_{s}%
^{m}-W\left(  s,X_{s}^{m}\right)  \right\vert ^{2}\right]  ds\\
& \leq\bar{C}(1+|x|^{2})\frac{T}{m}.
\end{array}
\label{new-lx-16}%
\end{equation}
Then we obtain (\ref{new-lx-9}) by (\ref{new-lx-13}), (\ref{new-lx-14}) and
(\ref{new-lx-16}).

\textbf{Step 3.} Note that $(X^{t,x;u^{m}},Y^{t,x;u^{m}},Z^{t,x;u^{m}})$
satisfies the following FBSDE:%
\begin{equation}
\left\{
\begin{array}
[c]{rl}%
dX_{s}^{t,x;u^{m}}= & b(s,X_{s}^{t,x;u^{m}},Y_{s}^{t,x;u^{m}},Z_{s}%
^{t,x;u^{m}},u_{s}^{m})ds+\sigma(s,X_{s}^{t,x;u^{m}},u_{s}^{m})dB_{s},\\
dY_{s}^{t,x;u^{m}}= & -g(s,X_{s}^{t,x;u^{m}},Y_{s}^{t,x;u^{m}},Z_{s}%
^{t,x;u^{m}},u_{s}^{m})ds+Z_{s}^{t,x;u^{m}}dB_{s},\text{ }s\in\lbrack t,T],\\
X_{t}^{t,x;u^{m}}= & x,\ Y_{T}^{t,x;u^{m}}=\phi(X_{T}^{t,x;u^{m}}).
\end{array}
\right.  \label{new-lx-6}%
\end{equation}
Then, by Theorem 2.2 in \cite{Hu-JX} for FBSDEs (\ref{new-lx-8}) and
(\ref{new-lx-6}), we obtain%
\begin{equation}%
\begin{array}
[c]{l}%
\mathbb{E}\left[  \sup\limits_{t\leq s\leq T}\left(  |X_{s}^{m}-X_{s}%
^{t,x;u^{m}}|^{2}+|\tilde{Y}_{s}^{m}-Y_{s}^{t,x;u^{m}}|^{2}\right)  +{\int%
}_{t}^{T}|\tilde{Z}_{s}^{m}-Z_{s}^{t,x;u^{m}}|^{2}ds\right] \\
\leq\tilde{C}\mathbb{E}\left[  {\int}_{t}^{T}\left(  |\tilde{Y}_{s}^{m}%
-Y_{s}^{m}|^{2}+|\tilde{Z}_{s}^{m}-Z_{s}^{m}|^{2}\right)  ds\right]  ,
\end{array}
\label{new-lx-7}%
\end{equation}
where $\tilde{C}$ is a constant which is independent of $m$. Due to%
\[%
\begin{array}
[c]{rl}%
\left\vert Y_{s}^{t,x;u^{m}}-W\left(  s,X_{s}^{t,x;u^{m}}\right)  \right\vert
& \leq\left\vert Y_{s}^{t,x;u^{m}}-\tilde{Y}_{s}^{m}\right\vert +\left\vert
\tilde{Y}_{s}^{m}-W\left(  s,X_{s}^{m}\right)  \right\vert \\
& \ \ +\left\vert W\left(  s,X_{s}^{m}\right)  -W\left(  s,X_{s}^{t,x;u^{m}%
}\right)  \right\vert ,
\end{array}
\]
then, by (\ref{new-lx-9}), (\ref{new-lx-7}) and Lemma \ref{le-w-new}, we get%
\begin{equation}
\mathbb{E}\left[  {\int}_{t}^{T}\left\vert Y_{s}^{t,x;u^{m}}-W\left(
s,X_{s}^{t,x;u^{m}}\right)  \right\vert ^{2}ds\right]  +\left\vert
Y_{t}^{t,x;u^{m}}-W\left(  t,x\right)  \right\vert \rightarrow0\text{ as
}m\rightarrow\infty. \label{new-lx-5}%
\end{equation}
This completes the proof.
\end{proof}

We first give a uniqueness result when $\sigma$ is independent of $z$.

\begin{theorem}
\label{th-um1} Suppose $\sigma$ is independent of $z$, $\tilde{W}$ is a
viscosity solution to HJB equation (\ref{eq-hjb}); and one of the following
two conditions holds true:

\begin{description}
\item[(i)] Assumption \ref{assum-1} holds;

\item[(ii)] Assumptions \ref{assum-1} (i) and \ref{assm-mon} hold. Moreover,
\[
\tilde{A}\left(  t,x,y,z,u\right)  =\left(  -G^{\intercal}g\left(
t,x,y,z,u\right)  , Gb\left(  t,x,y,z,u\right)  , G\sigma( t,x,\tilde
{W}\left(  t,x\right)  ,u)\right)  ^{\intercal}%
\]
satisfies Assumption \ref{assm-mon}.
\end{description}

\noindent Let $W$ be the value function. Furthermore, we assume that
$\tilde{W}$ is Lipschitz continuous in $x$. Then $W\leq\tilde{W}$.
\end{theorem}

\begin{proof}
We only prove the first case (the condition (i) holds). The proof for the
second case is similar. Consider the following HJB equation:%
\begin{equation}
\left\{
\begin{array}
[c]{l}%
\partial_{t}F(t,x)+\inf\limits_{u\in U}H(t,x,F(t,x),DF(t,x),D^{2}%
F(t,x),u)=0,\\
F(T,x)=\phi(x),
\end{array}
\right.  \label{new-lx-17}%
\end{equation}
where
\[%
\begin{array}
[c]{l}%
H(t,x,v,p,A,u)\\
=\frac{1}{2}\mathrm{tr}[\sigma\sigma^{\intercal}(t,x,\tilde{W}%
(t,x),u)A]+p^{\intercal}b(t,x,v,p^{\intercal}\sigma(t,x,\tilde{W
}(t,x),u),u)\\
\ \ +g(t,x,v,p^{\intercal}\sigma(t,x,\tilde{W}(t,x),u),u),\\
(t,x,v,p,A,u)\in\lbrack0,T]\times\mathbb{R}^{n}\times\mathbb{R}\times
\mathbb{R}^{n}\times\mathbb{S}^{n}\times U
\end{array}
\]
and $F\left(  \cdot,\cdot\right)  :[0,T]\times\mathbb{R}^{n}\rightarrow
\mathbb{R}. $ Note that $\tilde{\sigma}(t,x,u):=\sigma(t,x,\tilde{W}(t,x),u) $
is Lipschitz continuous in $x$. By the definition of viscosity solution, it is
easy to verify that $\tilde{W}$ is also a viscosity solution to HJB equation
(\ref{new-lx-17}). Since $\tilde{\sigma}$ is independent of $(y,z)$, by
Theorems \ref{th-vis} and \ref{th-vis-uni}, $\tilde{W}$ is the value function
of the following optimization problem:
\[
\underset{u(\cdot)\in\mathcal{U}^{t}[t,T]}{\inf}\bar{Y}_{t}^{t,x;u},
\]
where the controlled system is%
\begin{equation}
\left\{
\begin{array}
[c]{rl}%
d\bar{X}_{s}^{t,x;u}= & b(s,\bar{X}_{s}^{t,x;u},\bar{Y}_{s}^{t,x;u},\bar
{Z}_{s}^{t,x;u},u_{s})ds+\sigma(s,\bar{X}_{s}^{t,x;u},\tilde{W}(s,\bar{X}%
_{s}^{t,x;u}),u_{s})dB_{s},\\
d\bar{Y}_{s}^{t,x;u}= & -g(s,\bar{X}_{s}^{t,x;u},\bar{Y}_{s}^{t,x;u},\bar
{Z}_{s}^{t,x;u},u_{s})ds+\bar{Z}_{s}^{t,x;u}dB_{s},\;s\in\lbrack t,T],\\
\bar{X}_{t}^{t,x;u}= & x,\ \bar{Y}_{T}^{t,x;u}=\phi(\bar{X}_{T}^{t,x;u}).
\end{array}
\right.  \label{new-lx-18}%
\end{equation}
For each fixed $(t,x)\in\lbrack0,T)\times\mathbb{R}^{n}$, by Proposition
\ref{pr-um}, we can find a sequence $u^{m}\in\mathcal{U}^{t}[t,T]$ such that
\begin{equation}
\mathbb{E}\left[  {\int}_{t}^{T}\left\vert \bar{Y}_{s}^{t,x;u^{m}}-\tilde{W}
\left(  s,\bar{X}_{s}^{t,x;u^{m}}\right)  \right\vert ^{2}ds\right]
+\left\vert \bar{Y}_{t}^{t,x;u^{m}}-\tilde{W}\left(  t,x\right)  \right\vert
\rightarrow0\text{ as }m\rightarrow\infty. \label{new-le-21}%
\end{equation}
Consider the following FBSDE:%
\begin{equation}
\left\{
\begin{array}
[c]{rl}%
dX_{s}^{t,x;u^{m}}= & b(s,X_{s}^{t,x;u^{m}},Y_{s}^{t,x;u^{m}},Z_{s}%
^{t,x;u^{m}},u_{s}^{m})ds\\
& +\sigma(s,X_{s}^{t,x;u^{m}},Y_{s}^{t,x;u^{m}},u_{s}^{m})dB_{s},\\
dY_{s}^{t,x;u^{m}}= & -g(s,X_{s}^{t,x;u^{m}},Y_{s}^{t,x;u^{m}},Z_{s}%
^{t,x;u^{m}},u_{s}^{m})ds+Z_{s}^{t,x;u^{m}}dB_{s},\;s\in\lbrack t,T],\\
X_{t}^{t,x;u^{m}}= & x,\ Y_{T}^{t,x;u^{m}}=\phi(X_{T}^{t,x;u^{m}}).
\end{array}
\right.  \label{new-lx-19}%
\end{equation}
By Theorem 2.2 in \cite{Hu-JX} for FBSDEs (\ref{new-lx-18}) and
(\ref{new-lx-19}),%
\begin{equation}%
\begin{array}
[c]{l}%
|Y_{t}^{t,x;u^{m}}-\bar{Y}_{t}^{t,x;u^{m}}|\\
\leq C\left\{  \mathbb{E}\left[  \int_{t}^{T}\left\vert \sigma(s,\bar{X}%
_{s}^{t,x;u^{m}},\bar{Y}_{s}^{t,x;u^{m}},u_{s}^{m})-\sigma(s,\bar{X}%
_{s}^{t,x;u^{m}},\tilde{W}(s,\bar{X}_{s}^{t,x;u^{m}}),u_{s}^{m})\right\vert
^{2}ds\right]  \right\}  ^{\frac{1}{2}}\\
\leq C\left\{  \mathbb{E}\left[  {\int}_{t}^{T}\left\vert \bar{Y}%
_{s}^{t,x;u^{m}}-\tilde{W}\left(  s,\bar{X}_{s}^{t,x;u^{m}}\right)
\right\vert ^{2}ds\right]  \right\}  ^{\frac{1}{2}}.
\end{array}
\label{new-lx-20}%
\end{equation}
Since $Y_{t}^{t,x;u^{m}}\geq W(t,x)$ for any $m\geq1$, we get $W(t,x)\leq
\tilde{W}(t,x)$ by (\ref{new-le-21}) and (\ref{new-lx-20}).
\end{proof}

Now we study the case in which $\sigma$ is dependent on $y$ and $z$.

\begin{theorem}
\label{th-um2} Suppose one of the following two conditions holds true:

\begin{description}
\item[(i)] Assumptions \ref{assum-1} and \ref{assum-l3} hold;

\item[(ii)] Assumptions \ref{assum-1} (i), \ref{assum-l3} and \ref{assm-mon} hold.
\end{description}

\noindent Let $W$ be the value function and $\tilde{W}$ be a viscosity
solution to HJB equation (\ref{eq-hjb}). Furthermore, we assume that
$\tilde{W}$ is Lipschitz continuous in $(t,x)$, $D\tilde{W}$ is Lipschitz
continuous in $x$ and $||D\tilde{W}||_{\infty}L_{3}<1$. Then $W\leq\tilde{W}$.
\end{theorem}

\begin{proof}
We only prove the first case (the condition (i) holds). The proof for the
second case is similar. Consider the following HJB equation:%
\begin{equation}
\label{new-lx-21}\partial_{t}F(t,x)+\inf\limits_{u\in U}%
H(t,x,F(t,x),DF(t,x),D^{2}F(t,x),u)=0,\ F(T,x)=\phi(x),
\end{equation}
where%
\[%
\begin{array}
[c]{l}%
H(t,x,v,p,A,u)\\
=\frac{1}{2}\mathrm{tr}[\sigma\sigma^{\intercal}(t,x,\tilde{W}(t,x),\tilde{V
}(t,x,u),u)A]+p^{\intercal}b(t,x,\tilde{W}(t,x),\tilde{V}(t,x,u),u)\\
\text{ \ \ }+g(t,x,v,p^{\intercal}\sigma(t,x,\tilde{W}(t,x),\tilde{V
}(t,x,u),u),u),\\
\tilde{V}(t,x,u)=D\tilde{W}(t,x)^{\intercal}\sigma(t,x,\tilde{W }%
(t,x),\tilde{V}(t,x,u),u),\\
(t,x,v,p,A,u)\in\lbrack0,T]\times\mathbb{R}^{n}\times\mathbb{R}\times
\mathbb{R}^{n}\times\mathbb{S}^{n}\times U
\end{array}
\]
and $F\left(  \cdot,\cdot\right)  :[0,T]\times\mathbb{R}^{n}\rightarrow
\mathbb{R}. $

Recall that for given $\tilde{W}$, there exists a unique solution $\tilde{V
}(t,x,u)$ to the above algebra equation by Lemma \ref{le-h}.

Note that
\[%
\begin{array}
[c]{rl}%
\tilde{b}(t,x,u):= b(t,x,\tilde{W}(t,x),\tilde{V}(t,x,u),u),\ \tilde{\sigma
}(t,x,u):= \sigma(t,x,\tilde{W}(t,x),\tilde{V}(t,x,u),u) &
\end{array}
\]
satisfy the following conditions:%
\begin{equation}%
\begin{array}
[c]{l}%
| \tilde{b}(t,x,u)| +\left\vert \tilde{\sigma}(t,x,u)\right\vert \leq
C(1+\left\vert x\right\vert ),\\
| \tilde{b}(t,x,u)-\tilde{b}(t,x^{\prime},u)| +\left\vert \tilde{\sigma
}(t,x,u)-\tilde{\sigma}(t,x^{\prime},u)\right\vert \leq C(1+\left\vert
x\right\vert +\left\vert x^{\prime}\right\vert )|x-x^{\prime}|.
\end{array}
\label{new-lx-30}%
\end{equation}
By the definition of viscosity solution, $\tilde{W}$ is also a viscosity
solution to HJB equation (\ref{new-lx-21}). Consider the following controlled
system:%
\begin{equation}
\left\{
\begin{array}
[c]{rl}%
d\bar{X}_{s}^{t,x;u}= & b(s,\bar{X}_{s}^{t,x;u},\tilde{W}(s,\bar{X}%
_{s}^{t,x;u}),\tilde{V}(s,\bar{X}_{s}^{t,x;u},u_{s}),u_{s})ds\\
& +\sigma(s,\bar{X}_{s}^{t,x;u},\tilde{W}(s,\bar{X}_{s}^{t,x;u}),\tilde{V
}(s,\bar{X}_{s}^{t,x;u},u_{s}),u_{s})dB_{s},\\
d\bar{Y}_{s}^{t,x;u}= & -g(s,\bar{X}_{s}^{t,x;u},\bar{Y}_{s}^{t,x;u},\bar
{Z}_{s}^{t,x;u},u_{s})ds+\bar{Z}_{s}^{t,x;u}dB_{s},\;s\in\lbrack t,T],\\
\bar{X}_{t}^{t,x;u}= & x,\ \bar{Y}_{T}^{t,x;u}=\phi(\bar{X}_{T}^{t,x;u}).
\end{array}
\right.  \label{new-lx-22}%
\end{equation}
By Proposition 3.28 in \cite{Pardoux-book}, the FBSDE (\ref{new-lx-22}) has a
unique solution $(\bar{X}^{t,x;u}, \bar{Y}^{t,x;u},\newline\bar{Z}^{t,x;u})\in
L_{\mathcal{F}}^{p}(\Omega;C([t,T];\mathbb{R}^{n}))\times L_{\mathcal{F}}%
^{p}(\Omega;C([t,T];\mathbb{R}))\times L_{\mathcal{F}}^{2,p}(t,T;\mathbb{R}%
^{1\mathbb{\times}d})$ for any $p\geq2$. Since $\tilde{b}$ and $\tilde{\sigma
}$ are independent of $(y,z)$, we can obtain that $\tilde{W}$ is the value
function of the above controlled system (\ref{new-lx-22}). For each fixed
$(t,x)\in\lbrack0,T)\times\mathbb{R}^{n}$, by Proposition \ref{pr-um}, we can
find a sequence $u^{m}\in\mathcal{U}^{t}[t,T]$ such that
\begin{equation}
\mathbb{E}\left[  {\int}_{t}^{T}\left\vert \bar{Y}_{s}^{t,x;u^{m}}-\tilde{W
}\left(  s,\bar{X}_{s}^{t,x;u^{m}}\right)  \right\vert ^{2}ds\right]
+\left\vert \bar{Y}_{t}^{t,x;u^{m}}-\tilde{W}\left(  t,x\right)  \right\vert
\rightarrow0\text{ as }m\rightarrow\infty. \label{new-lx-23}%
\end{equation}

Let $\vartheta(t,x):\mathbb{R\times R}^{n}\rightarrow\mathbb{R}$\ be a
non-negative smooth function such that its support is included in the unit
ball and $\int_{\mathbb{R\times R}^{n}}\vartheta\left(  t,x\right)  dxdt=1$.
For Lipschitz functions $\tilde{W}:[0,T]\times\mathbb{R}^{n}\rightarrow
\mathbb{R}$, we set for $(t,x)\in\mathbb{R\times R}^{n}\text{, }\epsilon>0,$
\[
\tilde{W}_{\epsilon}\left(  t,x\right)  =\epsilon^{-(n+1)}\int%
_{\mathbb{R\times R}^{n}}\tilde{W}\left(  \left[  \left(  t-t^{\prime}\right)
\vee0\right]  \wedge T,x-x^{\prime}\right)  \vartheta\left(  \epsilon
^{-1}t^{\prime},\epsilon^{-1}x^{\prime}\right)  dx^{\prime}dt^{\prime}.
\]
Then, it is easily to verify that%
\begin{align*}
\partial_{t}\tilde{W}_{\epsilon}\left(  t,x\right)   &  =\epsilon^{-(n+1)}%
\int_{\mathbb{R\times R}^{n}}\partial_{t}\tilde{W}\left(  \left[  \left(
t-t^{\prime}\right)  \vee0\right]  \wedge T,x-x^{\prime}\right)
\vartheta\left(  \epsilon^{-1}t^{\prime},\epsilon^{-1}x^{\prime}\right)
dx^{\prime}dt^{\prime},\\
D\tilde{W}_{\epsilon}\left(  t,x\right)   &  =\epsilon^{-(n+1)}\int%
_{\mathbb{R\times R}^{n}}D\tilde{W}\left(  \left[  \left(  t-t^{\prime
}\right)  \vee0\right]  \wedge T,x-x^{\prime}\right)  \vartheta\left(
\epsilon^{-1}t^{\prime},\epsilon^{-1}x^{\prime}\right)  dx^{\prime}dt^{\prime
},
\end{align*}
where $\partial_{t}\tilde{W}$ is defined almost everywhere.

Thus we obtain that $||\partial_{t}\tilde{W}_{\epsilon}||_{\infty}\leq
L_{\tilde{W}}$, $||D\tilde{W}_{\epsilon}||_{\infty}\leq L_{D\tilde{W}}$,
$\tilde{W}_{\epsilon}\rightarrow\tilde{W}$ and $D\tilde{W}_{\epsilon
}\rightarrow D\tilde{W}$ pointwisely as $\epsilon\rightarrow0$, where
$L_{\tilde{W}}$ is the Lipschitz constant of $\tilde{W}$ with respect to $t$
and $L_{D\tilde{W}}$ is the Lipschitz constant of $D\tilde{W}$ with respect to
$x$. Set \ $Y_{s}^{\epsilon}=\tilde{W}_{\epsilon}(s,\bar{X}_{s}^{t,x;u^{m}%
}),$
\[
Z_{s}^{\epsilon}=(D\tilde{W}_{\epsilon}(s,\bar{X}_{s}^{t,x;u^{m}}%
))^{\intercal}\sigma(s,\bar{X}_{s}^{t,x;u^{m}},\tilde{W}(s,\bar{X}%
_{s}^{t,x;u^{m}}),\tilde{V}(s,\bar{X}_{s}^{t,x;u^{m}},u_{s}^{m}),u_{s}^{m}).
\]
Applying It\^{o}'s formula to $\tilde{W}_{\epsilon}(s,\bar{X}_{s}^{t,x;u^{m}%
})$ on $[t,T]$, we have
\begin{equation}
\left\{
\begin{array}
[c]{l}%
dY_{s}^{\epsilon}\\
= \left[  \left(  D\tilde{W}_{\epsilon}(s,\bar{X}_{s}^{t,x;u^{m}})\right)
^{\intercal}b(s,\bar{X}_{s}^{t,x;u^{m}},\tilde{W}(s,\bar{X}_{s}^{t,x;u^{m}%
}),\tilde{V}(s,\bar{X}_{s}^{t,x;u^{m}},u_{s}^{m}),u_{s}^{m})\right. \\
\ \ +\frac{1}{2}\mathrm{tr}\left(  \sigma\sigma^{\intercal}(s,\bar{X}%
_{s}^{t,x;u^{m}},\tilde{W}(s,\bar{X}_{s}^{t,x;u^{m}}),\tilde{V}(s,\bar{X}%
_{s}^{t,x;u^{m}},u_{s}^{m}),u_{s}^{m})D^{2}\tilde{W}_{\epsilon}(s,\bar{X}%
_{s}^{t,x;u^{m}})\right) \\
\ \ \left.  + \partial_{t}\tilde{W}_{\epsilon}(s,\bar{X}_{s}^{t,x;u^{m}%
})\right]  ds+Z_{s}^{\epsilon}dB_{s},\\
Y_{T}^{\epsilon}= \phi(\bar{X}_{T}^{t,x;u^{m}}),\;s\in\lbrack t,T].
\end{array}
\right.  \label{new-lx-26}%
\end{equation}
Applying It\^{o}'s formula to $\left\vert \bar{Y}_{s}^{t,x;u^{m}}%
-Y_{s}^{\epsilon}\right\vert ^{2}$ on $[t,T]$,%
\begin{equation}%
\begin{array}
[c]{rl}%
\mathbb{E}\left[  {\int}_{t}^{T}\left\vert \bar{Z}_{s}^{t,x;u^{m}}%
-Z_{s}^{\epsilon}\right\vert ^{2}ds\right]  & \leq2\mathbb{E}\left[  {\int%
}_{t}^{T}\left\vert \bar{Y}_{s}^{t,x;u^{m}}-Y_{s}^{\epsilon}\right\vert
I_{s}^{m}ds\right] \\
& \leq2\left\{  \mathbb{E}\left[  {\int}_{t}^{T}\left\vert \bar{Y}%
_{s}^{t,x;u^{m}}-Y_{s}^{\epsilon}\right\vert ^{2}ds\right]  \right\}
^{\frac{1}{2}}\left\{  \mathbb{E}\left[  {\int}_{t}^{T}|I_{s}^{m}%
|^{2}ds\right]  \right\}  ^{\frac{1}{2}},
\end{array}
\label{new-lx-27}%
\end{equation}
where%
\begin{equation}%
\begin{array}
[c]{l}%
I_{s}^{m}\\
= \left\vert g(s,\bar{X}_{s}^{t,x;u^{m}},\bar{Y}_{s}^{t,x;u^{m}},\bar{Z}%
_{s}^{t,x;u^{m}},u_{s}^{m})+\partial_{t}\tilde{W}_{\epsilon}(s,\bar{X}%
_{s}^{t,x;u^{m}})\right. \\
\text{\ \ \ } +\left(  D\tilde{W}_{\epsilon}(s,\bar{X}_{s}^{t,x;u^{m}%
})\right)  ^{\intercal}b(s,\bar{X}_{s}^{t,x;u^{m}},\tilde{W}(s,\bar{X}%
_{s}^{t,x;u^{m}}),\tilde{V}(s,\bar{X}_{s}^{t,x;u^{m}},u_{s}^{m}),u_{s}^{m})\\
\text{\ \ \ }\left.  +\frac{1}{2}\mathrm{tr}\left(  \sigma\sigma^{\intercal
}(s,\bar{X}_{s}^{t,x;u^{m}},\tilde{W}(s,\bar{X}_{s}^{t,x;u^{m}}),\tilde
{V}(s,\bar{X}_{s}^{t,x;u^{m}},u_{s}^{m}),u_{s}^{m})D^{2}\tilde{W}_{\epsilon
}(s,\bar{X}_{s}^{t,x;u^{m}})\right)  \right\vert \\
\leq C\left(  1+\left\vert \bar{X}_{s}^{t,x;u^{m}}\right\vert ^{2}+\left\vert
\bar{Y}_{s}^{t,x;u^{m}}\right\vert +\left\vert \bar{Z}_{s}^{t,x;u^{m}%
}\right\vert \right)  .
\end{array}
\label{new-ism}%
\end{equation}
By standard estimate for decoupled FBSDE (\ref{new-lx-22}), we obtain\ $\sup
_{m\geq1}\mathbb{E}[ {\int}_{t}^{T}|I_{s}^{m}|^{2}ds] \leq C\left(
1+\left\vert x\right\vert ^{4}\right)  . $ By dominate convergence theorem, we
have
\[
\mathbb{E}\left[  {\int}_{t}^{T}\left(  \left\vert Y_{s}^{\epsilon}-\tilde
{W}\left(  s,\bar{X}_{s}^{t,x;u^{m}}\right)  \right\vert ^{2}+\left\vert
Z_{s}^{\epsilon}-\tilde{V}(s,\bar{X}_{s}^{t,x;u^{m}},u_{s}^{m})\right\vert
^{2}\right)  ds\right]  \rightarrow0
\]
as $\epsilon\rightarrow0$. Then we deduce%
\begin{equation}%
\begin{array}
[c]{rl}
& \mathbb{E}\left[  {\int}_{t}^{T}\left\vert \bar{Z}_{s}^{t,x;u^{m}}-\tilde
{V}(s,\bar{X}_{s}^{t,x;u^{m}},u_{s}^{m})\right\vert ^{2}ds\right] \\
& \leq C\left(  1+\left\vert x\right\vert ^{2}\right)  \left\{  \mathbb{E}%
\left[  {\int}_{t}^{T}\left\vert \bar{Y}_{s}^{t,x;u^{m}}-\tilde{W}\left(
s,\bar{X}_{s}^{t,x;u^{m}}\right)  \right\vert ^{2}ds\right]  \right\}
^{\frac{1}{2}}%
\end{array}
\label{new-lx-29}%
\end{equation}
by taking $\epsilon\rightarrow0$ in (\ref{new-lx-27}). Consider the following
FBSDE:%
\begin{equation}
\left\{
\begin{array}
[c]{rl}%
dX_{s}^{t,x;u^{m}}= & b(s,X_{s}^{t,x;u^{m}},Y_{s}^{t,x;u^{m}},Z_{s}%
^{t,x;u^{m}},u_{s}^{m})ds\\
& +\sigma(s,X_{s}^{t,x;u^{m}},Y_{s}^{t,x;u^{m}},Z_{s}^{t,x,u^{m}},u_{s}%
^{m})dB_{s},\\
dY_{s}^{t,x;u^{m}}= & -g(s,X_{s}^{t,x;u^{m}},Y_{s}^{t,x;u^{m}},Z_{s}%
^{t,x;u^{m}},u_{s}^{m})ds+Z_{s}^{t,x;u^{m}}dB_{s},\;s\in\lbrack t,T],\\
X_{t}^{t,x;u^{m}}= & x,\ Y_{T}^{t,x;u^{m}}=\phi(X_{T}^{t,x;u^{m}}).
\end{array}
\right.  \label{new-lx-24}%
\end{equation}
By Theorem 2.2 in \cite{Hu-JX} for FBSDEs (\ref{new-lx-22}) and
(\ref{new-lx-24}), we obtain%
\begin{equation}%
\begin{array}
[c]{l}%
|Y_{t}^{t,x;u^{m}}-\bar{Y}_{t}^{t,x;u^{m}}|\\
\leq C\left\{  \mathbb{E}\left[  \left(  \displaystyle \int_{t}^{T}\left\vert
b(s,\bar{X}_{s}^{t,x;u^{m}},\bar{Y}_{s}^{t,x;u^{m}},\bar{Z}_{s}^{t,x;u^{m}%
},u_{s}^{m})\right.  \right.  \right.  \right. \\
\text{ \ \ \ \ \ \ \ \ }\left.  \left.  -b(s,\bar{X}_{s}^{t,x;u^{m}},\tilde
{W}(s,\bar{X}_{s}^{t,x;u^{m}}),\tilde{V}(s,\bar{X}_{s}^{t,x;u^{m}},u_{s}%
^{m}),u_{s}^{m})\right\vert ds\right)  ^{2}\\
\text{ \ \ } +\displaystyle\int_{t}^{T}\left\vert \sigma(s,\bar{X}%
_{s}^{t,x;u^{m}},\bar{Y}_{s}^{t,x;u^{m}},\bar{Z}_{s}^{t,x;u^{m}},u_{s}%
^{m})\right. \\
\text{ \ \ \ \ \ \ \ \ } \left.  \left.  \left.  -\sigma(s,\bar{X}%
_{s}^{t,x;u^{m}},\tilde{W}(s,\bar{X}_{s}^{t,x;u^{m}}),\tilde{V}(s,\bar{X}%
_{s}^{t,x;u^{m}},u_{s}^{m}),u_{s}^{m})\right\vert ^{2}ds\right]  \right\}
^{\frac{1}{2}}\\
\leq C\left\{  \mathbb{E}\left[  \displaystyle {\int}_{t}^{T}\left(
\left\vert \bar{Y}_{s}^{t,x;u^{m}}-\tilde{W}\left(  s,\bar{X}_{s}^{t,x;u^{m}%
}\right)  \right\vert ^{2}\right.  \right.  \right. \\
\text{ \ \ \ \ \ \ \ \ \ \ \ \ \ \ \ \ } \left.  \left.  \left.  +\left\vert
\bar{Z}_{s}^{t,x;u^{m}}-\tilde{V}(s,\bar{X}_{s}^{t,x;u^{m}},u_{s}%
^{m})\right\vert ^{2}\right)  ds\right]  \right\}  ^{\frac{1}{2}}.
\end{array}
\label{new-lx-25}%
\end{equation}
Since $Y_{t}^{t,x;u^{m}}\geq W(t,x)$ for any $m\geq1$, we get $W(t,x)\leq
\tilde{W}(t,x)$ by (\ref{new-lx-23}), (\ref{new-lx-29}) and (\ref{new-lx-25}).
\end{proof}

\begin{remark}
Note that $\tilde{b}$ and $\tilde{\sigma}$ are only local Lipschitz continuous
in $x$. Under this condition (\ref{new-lx-30}), we can still prove that the
value function of the controlled system (\ref{new-lx-22}) is the viscosity
solution to HJB equation (\ref{new-lx-21}) by using the method in
\cite{Peng-lecture}. The uniqueness of the solution to HJB equation
(\ref{new-lx-21}) can be still obtained similarly as in
\cite{Baeles-BP,Buckdahn-Li}.
\end{remark}

\begin{remark}
If $b$ and $\sigma$ are independent of $z$, we only need to suppose that
$\tilde{W}$ is Lipschitz continuous in $x$ in Theorem \ref{th-um2}.
\end{remark}

\begin{remark}
In Theorem \ref{th-um2}, if $\partial_{t}\tilde{W}\in C([0,T]\times
\mathbb{R}^{n})$, then we do not need the assumption that $\tilde{W}$ is
Lipschitz continuous in $t$. In this case, by the definition of viscosity
solution, we can deduce that$|\partial_{t}\tilde{W}(t,x)| \leq C(1+|x|^{2}). $
Thus the inequality (\ref{new-ism}) still holds. Following the same steps as
in the above proof, we can obtain the same result.
\end{remark}

Now we study the case in which the coefficients of the controlled system $b$,
$\sigma$ and $g$ are independent of control variable $u$. It is obviously that
for this case the corresponding HJB equation (\ref{eq-hjb}) degenerates to a
semilinear parabolic equation with an algebra equation.

\begin{theorem}
\label{th-uni-pde}Suppose $b$, $\sigma$ and $g$ are independent of control
variable $u$, $\sigma$ is independent of $z$, $\tilde{W}$ is a viscosity
solution to HJB equation (\ref{eq-hjb}); and one of the following two
conditions holds true:

\begin{description}
\item[(i)] Assumption \ref{assum-1} holds;

\item[(ii)] Assumptions \ref{assum-1} (i) and \ref{assm-mon} hold. Moreover,
\[
\tilde{A}\left(  t,x,y,z\right)  =\left(  -G^{\intercal}g\left(
t,x,y,z\right)  , Gb\left(  t,x,y,z\right)  , G\sigma\left(  t,x,\tilde
{W}\left(  t,x\right)  \right)  \right)  ^{\intercal}
\]
satisfies Assumption \ref{assm-mon}.
\end{description}

\noindent Let $W$ be the value function. Furthermore, we assume that
$\tilde{W}$ is Lipschitz continuous in $x$. Then $W=\tilde{W}$.
\end{theorem}

\begin{proof}
We only prove the first case (the condition (i) holds). The proof for the
second case is similar. Following the same steps in Theorem \ref{th-um1},
$\tilde{W}$ is also a viscosity solution to PDE system
\[
\partial_{t}F(t,x)+H(t,x,F(t,x),DF(t,x),D^{2}F(t,x))=0,\ F(T,x)=\phi(x),
\]
where $H\left(  \cdot\right)  $ is the function in equation (\ref{new-lx-17})
without control variables. Since $\tilde{\sigma}$ is independent of $(y,z)$,
by Theorems \ref{th-vis} and \ref{th-vis-uni}, we have $\tilde{W}\left(
t,x\right)  =\bar{Y}_{t}^{t,x}$, where $\bar{Y}_{t}^{t,x}$ is the solution to
the following FBSDE at time $t$,%
\begin{equation}
\left\{
\begin{array}
[c]{rl}%
d\bar{X}_{s}^{t,x}= & b(s,\bar{X}_{s}^{t,x},\bar{Y}_{s}^{t,x},\bar{Z}%
_{s}^{t,x})ds+\sigma(s,\bar{X}_{s}^{t,x},\tilde{W}(s,\bar{X}_{s}^{t,x}%
))dB_{s},\\
d\bar{Y}_{s}^{t,x}= & -g(s,\bar{X}_{s}^{t,x},\bar{Y}_{s}^{t,x},\bar{Z}%
_{s}^{t,x})ds+\bar{Z}_{s}^{t,x}dB_{s},\;s\in\lbrack t,T],\\
\bar{X}_{t}^{t,x}= & x,\ \bar{Y}_{T}^{t,x}=\phi(\bar{X}_{T}^{t,x}).
\end{array}
\right.  \label{new-lx-81}%
\end{equation}
For each fixed $(t,x)\in\lbrack0,T)\times\mathbb{R}^{n}$, by Proposition
\ref{pr-um}, we have
\begin{equation}
\bar{Y}_{s}^{t,x}=\tilde{W}\left(  s,\bar{X}_{s}^{t,x}\right)  \text{ for
}s\in\lbrack t,T]. \label{new-lx-911}%
\end{equation}
Then $\left(  \bar{X}^{t,x},\bar{Y}^{t,x},\bar{Z}^{t,x}\right)  $ satisfies
the following fully coupled FBSDE:%
\begin{equation}
\left\{
\begin{array}
[c]{rl}%
dX_{s}^{t,x}= & b(s,X_{s}^{t,x},Y_{s}^{t,x},Z_{s}^{t,x})ds+\sigma
(s,X_{s}^{t,x},Y_{s}^{t,x})dB_{s},\\
dY_{s}^{t,x}= & -g(s,X_{s}^{t,x},Y_{s}^{t,x},Z_{s}^{t,x})ds+Z_{s}^{t,x}%
dB_{s},\;s\in\lbrack t,T],\\
X_{t}^{t,x}= & x,\ Y_{T}^{t,x}=\phi(X_{T}^{t,x}).
\end{array}
\right.  \label{new-lx-91}%
\end{equation}
By the uniqueness of FBSDEs (\ref{new-lx-81}) and (\ref{new-lx-91}), one has%
\begin{equation}
Y_{t}^{t,x}=\bar{Y}_{t}^{t,x}. \label{new-lx-912}%
\end{equation}
Since $Y_{t}^{t,x}=W(t,x)$, we get $W(t,x)=\tilde{W}(t,x)$ by
(\ref{new-lx-911}) and (\ref{new-lx-912}).
\end{proof}

Similarly, we have the following theorem.

\begin{theorem}
\label{uni-pde}Suppose $b$, $\sigma$ and $g$ are independent of control
variable $u$; and one of the following two conditions holds true:

\begin{description}
\item[(i)] Assumptions \ref{assum-1} and \ref{assum-l3} hold;

\item[(ii)] Assumptions \ref{assum-1} (i), \ref{assum-l3} and \ref{assm-mon} hold.
\end{description}

\noindent Let $W$ be the value function and $\tilde{W}$ be a viscosity
solution to HJB equation (\ref{eq-hjb}). Furthermore, we assume that
$\tilde{W}$ is Lipschitz continuous in $(t,x)$, $D\tilde{W}$ is Lipschitz
continuous in $x$ and $||D\tilde{W}||_{\infty}L_{3}<1$. Then $W=\tilde{W}$.
\end{theorem}

\begin{remark}
In the above theorems, we assume that the Assumption \ref{assum-1} or the
monotonicity conditions hold. It is well-known that there are other conditions
which can guarantee the existence and uniqueness of the fully coupled
controlled system (\ref{state-eq}). In fact, our approach can be generalized
to deal with any fully coupled controlled system which is well-posed and the
related $L^{2}$-estimates of the solution hold.
\end{remark}

\subsection{The smooth case}

{In this subsection, we assume that the solution of the HJB equation
$\tilde{W}\in C^{1,2}([0,T]\times\mathbb{R}^{n})$. Then, we have the following
theorem.}

\begin{theorem}
Suppose one of the following two conditions holds true:

\begin{description}
\item[(i)] Assumptions \ref{assum-1} and \ref{assum-l3} hold;

\item[(ii)] Assumptions \ref{assum-1} (i), \ref{assum-l3} and \ref{assm-mon} hold.
\end{description}

\noindent Let $W$ be the value function and $\tilde{W}\in C^{1,2}%
([0,T]\times\mathbb{R}^{n})$ be a solution to the HJB equation (\ref{eq-hjb}).
Furthermore, we assume $||\sigma||_{\infty}<\infty$, $||D\tilde{W}||_{\infty
}L_{3}<1$ and $||D^{2}\tilde{W}||_{\infty}<\infty$. Then $W=\tilde{W}$.
\end{theorem}

\begin{proof}
Without loss of generality, we only prove the case $d=1$. For each given
$u\in\mathcal{U}^{t}[t,T]$, let $(X^{t,x;u},Y^{t,x;u},Z^{t,x;u})$ be the
solution to FBSDE (\ref{state-eq}) with $\xi=x$. Applying It\^{o}'s formula to
$\tilde{W}(s,X_{s}^{t,x;u})$, we obtain%
\[
\left\{
\begin{array}
[c]{l}%
d\tilde{W}(s,X_{s}^{t,x;u})\\
= \left\{  \partial_{t}\tilde{W}(s,X_{s}^{t,x;u}))+\left(  D\tilde{W}%
(s,X_{s}^{t,x;u})\right)  ^{\intercal}b(s,X_{s}^{t,x;u},Y_{s}^{t,x;u}%
,Z_{s}^{t,x;u},u_{s})\right. \\
\ \ \left.  +\frac{1}{2}\mathrm{tr}[\sigma\sigma^{\intercal}(s,X_{s}%
^{t,x;u},Y_{s}^{t,x;u},Z_{s}^{t,x;u},u_{s})D^{2}\tilde{W}(s,X_{s}%
^{t,x;u})]\right\}  ds\\
\ \ +D\tilde{W}(s,X_{s}^{t,x;u})^{\intercal}\sigma(s,X_{s}^{t,x;u}%
,Y_{s}^{t,x;u},Z_{s}^{t,x;u},u_{s})dB_{s},\;s\in\lbrack t,T],\\
\tilde{W}(T,X_{T}^{t,x;u})= \phi(X_{T}^{t,x;u}).
\end{array}
\right.
\]
Set $
\tilde{Y}_{s}=\tilde{W}(s,X_{s}^{t,x;u}),\ \tilde{Z}_{s}=D\tilde{W}%
(s,X_{s}^{t,x;u})^{\intercal}\sigma(s,X_{s}^{t,x;u},Y_{s}^{t,x;u}%
,Z_{s}^{t,x;u},u_{s}),\newline\hat{Y}_{s}=Y_{s}^{t,x;u}-\tilde{Y}_{s},\text{
}\hat{Z}_{s}=Z_{s}^{t,x;u}-\tilde{Z}_{s}.$
Then,
\begin{equation}
\label{new-asd-11}d\hat{Y}_{s}= -\left(  \Pi_{1}(s)+\Pi_{2}(s)\right)
ds+\hat{Z}_{s}dB_{s},\ \hat{Y}_{T}= 0,
\end{equation}
where%
\[%
\begin{array}
[c]{rl}%
\Pi_{1}(s)= & H(s,X_{s}^{t,x;u},\tilde{W}(s,X_{s}^{t,x;u}),D\tilde{W }%
(s,X_{s}^{t,x;u}),D^{2}\tilde{W}(s,X_{s}^{t,x;u}),u_{s})\\
& +\partial_{t}\tilde{W}(s,X_{s}^{t,x;u})\geq0,\\
\Pi_{2}(s)= & \left(  D\tilde{W}(s,X_{s}^{t,x;u})\right)  ^{\intercal}\left[
b_{1}(s)-b_{2}(s)\right]  +g_{1}(s)-g_{2}(s)\\
& +\frac{1}{2}\mathrm{tr}[\left(  \sigma_{1}\sigma_{1}^{\intercal}%
(s)-\sigma_{2}\sigma_{2}^{\intercal}(s)\right)  D^{2}\tilde{W}(s,X_{s}%
^{t,x;u})],\\
b_{1}(s)= & b(s,X_{s}^{t,x;u},Y_{s}^{t,x;u},Z_{s}^{t,x;u},u_{s}),\\
b_{2}(s)= & b(s,X_{s}^{t,x;u},\tilde{W}(s,X_{s}^{t,x;u}),\tilde{V}%
(s,X_{s}^{t,x;u},u_{s}),u_{s}),\\
\tilde{V}(t,x,u)= & D\tilde{W}(t,x)^{\intercal}\sigma(t,x,\tilde{W
}(t,x),\tilde{V}(t,x,u),u),
\end{array}
\]
and $\sigma_{i}$, $g_{i}$ are defined similarly for $i=1,2$. Let%
\[
b_{1}(s)-b_{2}(s)=\beta_{1}(s)\hat{Y}_{s}+\gamma_{1}(s)\left(  Z_{s}%
^{t,x;u}-\tilde{V}(s,X_{s}^{t,x;u},u_{s})\right)
\]
and%
\[%
\begin{array}
[c]{l}%
Z_{s}^{t,x;u}-\tilde{V}(s,X_{s}^{t,x;u},u_{s})\\
=\hat{Z}_{s}+D\tilde{W}(s,X_{s}^{t,x;u})^{\intercal}\left(  \sigma
_{1}(s)-\sigma_{2}(s)\right) \\
=\hat{Z}_{s}+D\tilde{W}(s,X_{s}^{t,x;u})^{\intercal}\left[  \beta_{2}%
(s)\hat{Y}_{s}+\gamma_{2}(s)\left(  Z_{s}^{t,x;u}-\tilde{V}(s,X_{s}%
^{t,x;u},u_{s})\right)  \right]  ,\\
=\hat{Z}_{s}+D\tilde{W}(s,X_{s}^{t,x;u})^{\intercal}\beta_{2}(s)\hat{Y}%
_{s}+D\tilde{W}(s,X_{s}^{t,x;u})^{\intercal}\gamma_{2}(s)\left(  Z_{s}%
^{t,x;u}-\tilde{V}(s,X_{s}^{t,x;u},u_{s})\right)  ,
\end{array}
\]
where
\[%
\begin{array}
[c]{l}%
\beta_{1}(s)=\left\{
\begin{array}
[c]{cl}%
\frac{b_{1}(s)-b_{2}(s)}{Y_{s}^{t,x;u}-\tilde{Y}_{s}}, & \text{if }%
Y_{s}^{t,x;u}-\tilde{Y}_{s}\neq0,\\
0, & \text{if }Y_{s}^{t,x;u}-\tilde{Y}_{s}=0,
\end{array}
\right. \\
\gamma_{1}(s)=\left\{
\begin{array}
[c]{cl}%
\frac{b_{1}(s)-b_{2}(s)}{Z_{s}^{t,x;u}-\tilde{V}(s,X_{s}^{t,x;u},u_{s})}, &
\text{if }Z_{s}^{t,x;u}-\tilde{V}(s,X_{s}^{t,x;u},u_{s})\neq0,\\
0, & \text{if }Z_{s}^{t,x;u}-\tilde{V}(s,X_{s}^{t,x;u},u_{s})=0,
\end{array}
\right.
\end{array}
\]
$\beta_{2}\left(  \cdot\right)  $ and $\gamma_{2}\left(  \cdot\right)  $ are
defined similarly. Set
\begin{equation}
b_{1}(s)-b_{2}(s)=\tilde{\beta}_{1}(s)\hat{Y}_{s}+\tilde{\gamma}_{1}(s)\hat
{Z}_{s}, \label{new-lx-31}%
\end{equation}
where \ $
\tilde{\beta}_{1}(s)=\beta_{1}(s)+\left(  1-D\tilde{W}(s,X_{s}^{t,x;u}%
)^{\intercal}\gamma_{2}(s)\right)  ^{-1}D\tilde{W}(s,X_{s}^{t,x;u}%
)^{\intercal}\beta_{2}(s)\gamma_{1}\left(  s\right)  ,\newline\tilde{\gamma
}_{1}(s)=\left(  1-D\tilde{W}(s,X_{s}^{t,x;u})^{\intercal}\gamma
_{2}(s)\right)  ^{-1}\gamma_{1}(s).$
It is easy to check that $\beta_{1}$, $\gamma_{1}$, $\beta_{2}$, $\gamma_{2}$,
$\tilde{\beta}_{1}$ and $\tilde{\gamma}_{1}$ are bounded. Note that $\sigma$
is bounded. Similar to the proof of (\ref{new-lx-31}), we have
\[%
\begin{array}
[c]{rl}%
\mathrm{tr}[\left(  \sigma_{1}\sigma_{1}^{\intercal}(s)-\sigma_{2}\sigma
_{2}^{\intercal}(s)\right)  D^{2}\tilde{W}(s,X_{s}^{t,x;u})] & =\tilde{\beta
}_{2}(s)\hat{Y}(s)+\tilde{\gamma}_{2}(s)\hat{Z}(s),\\
g_{1}(s)-g_{2}(s) & =\tilde{\beta}_{3}(s)\hat{Y}(s)+\tilde{\gamma}_{3}%
(s)\hat{Z}(s),
\end{array}
\]
where $\tilde{\beta}_{2}$, $\tilde{\gamma}_{2}$, $\tilde{\beta}_{3}$
$\tilde{\gamma}_{3}$ are defined similarly and bounded. Then, we can rewrite
$\Pi_{2}(s)$ as \ $\Pi_{2}(s)=\beta(s)\hat{Y}(s)+\gamma(s)\hat{Z}(s), $ where
$\beta$, $\gamma\in\mathbb{R}$ are bounded. By the comparison theorem of
BSDEs, we get $\hat{Y}_{t}\geq0$ which implies $\tilde{W}(t,x)\leq
Y_{t}^{t,x;u}$. Thus $\tilde{W}\leq W$. On the other hand, by Theorem
\ref{th-um2}, we have $W\leq\tilde{W}$. Thus, $W=\tilde{W}$.
\end{proof}

\section{Appendix}

\subsection{The comparison theorem for FBSDEs}

Under Assumption \ref{assum-1}, we deduce a generalized comparison theorem for
FBSDEs. The proof is similar to that of Theorem 3.1 in \cite{Wu-compar} and
Theorem 5.11 in \cite{Li-W}. For the reader's convenience, we give a detailed
proof. Consider the following FBSDEs:
\begin{equation}
\left\{
\begin{array}
[c]{rl}%
dX_{s}^{i}= & b(s,X_{s}^{i},Y_{s}^{i},Z_{s}^{i})ds+\sigma(s,X_{s}^{i}%
,Y_{s}^{i},Z_{s}^{i})dB_{s},\\
dY_{s}^{i}= & -g(s,X_{s}^{i},Y_{s}^{i},Z_{s}^{i})ds+Z_{s}^{i}dB_{s},\\
X_{t}^{i}= & \xi,\ Y_{t+\delta}^{i}=\phi_{i}(X_{t+\delta}^{i}),\text{ }i=1,2.
\end{array}
\right.  \label{eq-appen}%
\end{equation}

\begin{theorem}
\label{th-comp}Suppose Assumption \ref{assum-1} holds. Then for $\delta
\in(0,T-t]$\ and $\xi\in L^{2}(\mathcal{F}_{t};\mathbb{R}^{n})$,
(\ref{eq-appen}) has a unique solution $(X_{s}^{i},Y_{s}^{i},Z_{s}^{i}%
)_{s\in\lbrack t,t+\delta]}$ associated with $(b,\sigma,$$g,\phi_{i})$. If
$\phi_{1}(X_{t+\delta}^{2})\geq\phi_{2}(X_{t+\delta}^{2}),P$-a.s. (resp.
$\phi_{1}(X_{t+\delta}^{1})\geq\phi_{2}(X_{t+\delta}^{1}),$ $P$-a.s.), then we
have $Y_{t}^{1}\geq Y_{t}^{2},$ $P$-$a.s.$.
\end{theorem}

\begin{proof}
Without loss of generality, we only prove the case $d=1$. Let $\hat{X}%
=X^{1}-X^{2}$, $\hat{Y}=Y^{1}-Y^{2}$, $\hat{Z}=Z^{1}-Z^{2}$. Then $\left(
\hat{X},\hat{Y},\hat{Z}\right)  $ satisfies the following FBSDE:
\[
\left\{
\begin{array}
[c]{rl}%
d\hat{X}_{s}= & \left[  b^{1}(s)\hat{X}_{s}+b^{2}(s)\hat{Y}_{s}+b^{3}%
(s)\hat{Z}_{s}\right]  ds\\
& +\left[  \sigma^{1}(s)\hat{X}_{s}+\sigma^{2}(s)\hat{Y}_{s}+\sigma^{3}%
(s)\hat{Z}_{s}\right]  dB_{s},\\
d\hat{Y}_{s}= & -\left[  g^{1}(s)\hat{X}_{s}+g^{2}(s)\hat{Y}_{s}+g^{3}%
(s)\hat{Z}_{s}\right]  ds+\hat{Z}_{s}dB_{s},\text{ }s\in\lbrack t,t+\delta],\\
\hat{X}_{t}= & 0,\ \hat{Y}_{t+\delta}=\phi^{1}(t+\delta)\hat{X}_{t+\delta
}+\phi_{1}(X_{t+\delta}^{2})-\phi_{2}\left(  X_{t+\delta}^{2}\right)  ,
\end{array}
\right.
\]
where $b^{i}$, $\sigma^{i}$, $g^{i}$, $\phi^{1}$, $i=1,2,3$ are defined in the
proof of Lemma \ref{est-initial}. Introduce the adjoint equation for the above
equation as follows%
\begin{equation}
\left\{
\begin{array}
[c]{rl}%
dh_{s}= & \left[  g^{2}(s)h_{s}+b^{2}(s)m_{s}+\sigma^{2}(s)n_{s}\right]
ds+\left[  g^{3}(s)h_{s}+b^{3}(s)m_{s}+\sigma^{3}(s)n_{s}\right]  dB_{s},\\
dm_{s}= & -\left[  g^{1}(s)h_{s}+b^{1}(s)m_{s}+\sigma^{1}(s)n_{s}\right]
ds+n_{s}dB_{s},\\
h_{t}= & 1,\ m_{t+\delta}=\phi^{1}(t+\delta)h_{t+\delta}.
\end{array}
\right.  \label{eq-dual-comp}%
\end{equation}
It is easy to check that (\ref{eq-dual-comp}) satisfies the assumptions of
Theorem 2.2 in \cite{Hu-JX}. Consequently, it has a unique solution
$(h,m,n)\in L_{\mathcal{F}}^{2}(\Omega;C([t,t+\delta];\mathbb{R}))\times
L_{\mathcal{F}}^{2}(\Omega;C([t,t+\delta];\mathbb{R}^{n}))\times
L_{\mathcal{F}}^{2,2}(t,t+\delta;\mathbb{R}^{n\times d})$. Applying It\^{o}'s
formula to $m\hat{X}-h\hat{Y}$, we get%
\[
\hat{Y}_{t}=\mathbb{E}\left[  \left.  \left(  \phi_{1}(X_{t+\delta}^{2}%
)-\phi_{2}\left(  X_{t+\delta}^{2}\right)  \right)  h_{t+\delta}\right\vert
\mathcal{F}_{t}\right]  .
\]
Since $\phi_{1}(X_{t+\delta}^{2})\geq\phi_{2}(X_{t+\delta}^{2})$, $P$-$a.s.$,
we only need to prove $h_{t+\delta}\geq0$, $P$-$a.s.$. Define $\tau
=\inf\left\{  s>t:h_{s}=0\right\}  \wedge\left(  t+\delta\right)  $ and
consider the following FBSDE on $\left[  \tau,t+\delta\right]  $,
\begin{equation}
\left\{
\begin{array}
[c]{rl}%
d\tilde{h}_{s}= & \left[  g^{2}(s)\tilde{h}_{s}+b^{2}(s)\tilde{m}_{s}%
+\sigma^{2}(s)\tilde{n}_{s}\right]  ds\\
& +\left[  g^{3}(s)\tilde{h}_{s}+b^{3}(s)\tilde{m}_{s}+\sigma^{3}(s)\tilde
{n}_{s}\right]  dB_{s},\\
d\tilde{m}_{s}= & -\left[  g^{1}(s)\tilde{h}_{s}+b^{1}(s)\tilde{m}_{s}%
+\sigma^{1}(s)\tilde{n}_{s}\right]  ds+\tilde{n}_{s}dB_{s},\\
\tilde{h}_{\tau}= & 0,\ \tilde{m}_{t+\delta}=\phi^{1}(t+\delta)\tilde
{h}_{t+\delta}.
\end{array}
\right.  \label{eq-dual-comp2}%
\end{equation}
This FBSDE has a unique solution $(\tilde{h},\tilde{m},\tilde{n})=\left(
0,0,0\right)  $. Set
\[%
\begin{array}
[c]{rl}
& \bar{h}_{s}=h_{s}I_{\left[  t,\tau\right]  }\left(  s\right)  +\tilde{h}%
_{s}I_{(\tau,t+\delta]}\left(  s\right)  ,\text{ }\bar{m}_{s}=m_{s}I_{\left[
t,\tau\right]  }\left(  s\right)  +\tilde{m}_{s}I_{(\tau,t+\delta]}\left(
s\right)  ,\\
& \bar{n}_{s}=n_{s}I_{\left[  t,\tau\right]  }\left(  s\right)  +\tilde{n}%
_{s}I_{(\tau,t+\delta]}\left(  s\right)  .
\end{array}
\]
It is clear that $(\bar{h},\bar{m},\bar{n})$ is a solution to
(\ref{eq-dual-comp}). The definition of $\tau$ yields the desired result
$\bar{h}_{t+\delta}\geq0$.
\end{proof}

\subsection{$L^{p}$ estimate of FBSDEs}

The following Lemma is a combination of Theorem 3.17 and Theorem 5.17 in
\cite{Pardoux-book}.

\begin{lemma}
\label{sde-bsde} Consider the following decoupled FBSDE
\begin{equation}
\left\{
\begin{array}
[c]{rl}%
dX_{s}= & \bar{b}(s,X_{s})ds+\bar{\sigma}(s,X_{s})dB_{s},\\
dY_{s}= & -\bar{g}(s,X_{s},Y_{s},Z_{s})ds+Z_{s}dB_{s},\\
X_{0}= & x_{0},\ Y_{T^{\prime}}=\bar{\phi}(X_{T^{\prime}}),
\end{array}
\right.  \label{fbsde-pardoux}%
\end{equation}
where $T^{\prime}\leq T$ for some fixed $T>0$,
\[
\bar{b}:[0,T]\times\Omega\times\mathbb{R}^{n}\rightarrow\mathbb{R}^{n},
\ \bar{\sigma}:[0,T]\times\Omega\times\mathbb{R}^{n}\rightarrow\mathbb{R}%
^{n\times d},
\]%
\[
\bar{g}:[0,T]\times\Omega\times\mathbb{R}^{n}\rightarrow\mathbb{R},
\ \bar{\phi}:\mathbb{R}^{n}\rightarrow\mathbb{R}.
\]
For each fixed $p>1$, if the coefficients satisfy

(i) $\bar{b}(\cdot,0)$, $\bar{\sigma}(\cdot,0)$, $\bar{g}(\cdot,0,0,0)$ are
$\mathbb{F}$-adapted processes and
\[
\mathbb{E}\left\{  |\bar{\phi}(0)|^{p}+\left(  \int_{0}^{T^{\prime}}\left[
|\bar{b}(s,0)|+|\bar{g}(s,0,0,0)|\right]  ds\right)  ^{p}+\left(  \int%
_{0}^{T^{\prime}}|\bar{\sigma}(s,0)|^{2}ds\right)  ^{\frac{p}{2}}\right\}
<\infty,
\]

(ii)%
\[%
\begin{array}
[c]{rl}%
|\bar{\psi}(s,x_{1})-\bar{\psi}(s,x_{2})| & \leq L_{1}|x_{1}-x_{2}%
|,\ \ \text{for }\ \bar{\psi}=\bar{b},\bar{\sigma},\bar{\phi};\\
|\bar{g}(s,x_{1},y_{1},z_{1})-\bar{g}(s,x_{2},y_{2},z_{2})| & \leq
L_{1}(|x_{1}-x_{2}|+|y_{1}-y_{2}|+|z_{1}-z_{2}|),
\end{array}
\]
then (\ref{fbsde-pardoux}) has a unique solution $(X,Y,Z)\in L_{\mathcal{F}%
}^{p}(\Omega;C([0,T^{\prime}],\mathbb{R}^{n}))\times L_{\mathcal{F}}%
^{p}(\Omega;C([0,T^{\prime}],\mathbb{R}^{m}))\times L_{\mathcal{F}}%
^{2,p}([0,T^{\prime}];\mathbb{R}^{m\times d})$ and there exists a constant
$C_{p}$ which only depends on $L_{1}$, $p$, $T$ such that
\[%
\begin{array}
[c]{l}%
\mathbb{E}\left\{  \sup\limits_{s\in\lbrack0,T^{\prime}]}\left[  |X_{s}%
|^{p}+|Y_{s}|^{p}\right]  +\left(  \int_{0}^{T^{\prime}}|Z_{s}|^{2}ds\right)
^{\frac{p}{2}}\right\} \\
\ \leq C_{p}\mathbb{E}\left\{  \left[  \int_{0}^{T^{\prime}}\left(  |\bar
{b}(s,0)|+|\bar{g}(s,0,0,0)|\right)  ds\right]  ^{p}+\left(  \int%
_{0}^{T^{\prime}}|\bar{\sigma}(s,0)|^{2}ds\right)  ^{\frac{p}{2}}+|\bar{\phi
}(0)|^{p}+|x_{0}|^{p}\right\}  .
\end{array}
\]

\end{lemma}

\begin{remark}
\label{re-app-2} In the above lemma, the constant $C_{p}$ depends on the
Lipschitz constant $L_{1}$. To represent this dependence, we sometimes denote
$C_{p}$ by $C_{p}(L_{1})$. Generally, we can choose $C_{p}(L_{1})$ as an
increasing function with respect to $L_{1}$.
\end{remark}

Consider the following FBSDE:%
\begin{equation}
\left\{
\begin{array}
[c]{rl}%
dX_{s}= & b(s,X_{s},Y_{s},Z_{s})ds+\sigma(s,X_{s},Y_{s},Z_{s})dB_{s},\\
dY_{s}= & -g(s,X_{s},Y_{s},Z_{s})ds+Z_{s}dB_{s},\text{ }s\in\lbrack
0,T^{\prime}],\\
X_{0}= & x,\ Y_{T^{\prime}}=\phi(X_{T^{\prime}}),
\end{array}
\right.  \label{new-new-new-new-1}%
\end{equation}
where $T^{\prime}\leq T$ for some fixed $T>0$.

\begin{theorem}
\label{th-lp}Suppose Assumption \ref{assum-1} (i) holds and $C_{p}%
2^{p-1}\left[  L_{2}^{p}\left(  T^{\frac{p}{2}}+T^{p}\right)  \right.  $
$\left.  + L_{3}^{p} \right]  <1$ for some $p\geq2$, where $C_{p}$ is defined
in Lemma 5.1 in \cite{Hu-JX}.$\ $Then FBSDE (\ref{new-new-new-new-1}) admits a
unique solution $(X,Y,Z)\in L_{\mathcal{F}}^{p}(\Omega;C([0,T^{\prime
}];\mathbb{R}^{n}))\times L_{\mathcal{F}}^{p}(\Omega;C([0,T^{\prime
}];\mathbb{R}))$$\times L_{\mathcal{F}}^{2,p}(0,T^{\prime};\mathbb{R}^{1\times
d})$ and
\[%
\begin{array}
[c]{l}%
||(X,Y,Z)||_{p}^{p}=\mathbb{E}\left[  \sup\limits_{t\in\lbrack0,T^{\prime}%
]}\left(  |X_{t}|^{p}+|Y_{t}|^{p}\right)  +\left(  \int_{0}^{T^{\prime}}%
|Z_{t}|^{2}dt\right)  ^{\frac{p}{2}}\right] \\
\ \leq C\mathbb{E}\left[  \left(  \int_{0}^{T^{\prime}}%
[|b|+|g|](t,0,0,0)dt\right)  ^{p}+\left(  \int_{0}^{T^{\prime}}|\sigma
(t,0,0,0)|^{2}dt\right)  ^{\frac{p}{2}}+|\phi(0)|^{p}+|x|^{p}\right]  ,
\end{array}
\]
where $C$ depends only on $T$, $p$, $L_{1}$, $L_{2}$, $L_{3}$.
\end{theorem}

\begin{proof}
Let $\mathcal{L}$ denote the space of all $\mathbb{F}$-adapted processes
$(Y,Z)$ such that
\[
\mathbb{E}\left[  \sup\limits_{0\leq t\leq T^{\prime}}|Y_{t}|^{p}+\left(
\int_{0}^{T^{\prime}}|Z_{t}|^{2}dt\right)  ^{\frac{p}{2}}\right]  <\infty.
\]
For each given $(y,z)\in\mathcal{L}$, consider the following decoupled FBSDE:
\begin{equation}
\left\{
\begin{array}
[c]{rl}%
dX_{t}= & b(t,X_{t},y_{t},z_{t})dt+\sigma(t,X_{t},y_{t},z_{t})dB_{t},\\
dY_{t}= & -g(t,X_{t},Y_{t},Z_{t})dt+Z_{t}dB_{t},\\
X_{0}= & x,\ Y_{T^{\prime}}=\phi(X_{T^{\prime}}).
\end{array}
\right.  \label{fbsde-y0}%
\end{equation}
Under Lipschitz conditions for $b$, $\sigma$, $g$ and $\phi$, it is easy to
deduce that the solution $(Y,Z)$ to \eqref{fbsde-y0} belongs to $\mathcal{L}$.
Denote the operator $(y,z)\rightarrow(Y,Z)$ by $\Gamma$. For two elements
$(y^{i},z^{i})\in\mathcal{L}$, $i=1,2$, let $(X^{i},Y^{i},Z^{i})$ be the
corresponding solution to \eqref{fbsde-y0}. Set
\[
\hat{y}_{t}=y_{t}^{1}-y_{t}^{2},\text{ }\hat{z}_{t}=z_{t}^{1}-z_{t}^{2},\text{
}\hat{X}_{t}=X_{t}^{1}-X_{t}^{2},\text{ }\hat{Y}_{t}=Y_{t}^{1}-Y_{t}%
^{2},\text{ }\hat{Z}_{t}=Z_{t}^{1}-Z_{t}^{2}.
\]
Due to Lemma 5.1 in \cite{Hu-JX}, we obtain
\begin{equation}%
\begin{array}
[c]{l}%
\mathbb{E}\left[  \sup\limits_{0\leq t\leq T^{\prime}}\left(  |\hat{X}%
_{t}|^{p}+|\hat{Y}_{t}|^{p}\right)  +\left(  \int_{0}^{T^{\prime}}|\hat{Z}%
_{t}|^{2}dt\right)  ^{\frac{p}{2}}\right] \\
\leq C_{p}\mathbb{E}\left[  \left(  \int_{0}^{T^{\prime}}\left\vert
b(t,X_{t}^{2},y_{t}^{1},z_{t}^{1})-b(t,X_{t}^{2},y_{t}^{2},z_{t}%
^{2})\right\vert dt\right)  ^{p}\right. \\
\ \ \left.  +\left(  \int_{0}^{T^{\prime}}\left\vert \sigma(t,X_{t}^{2}%
,y_{t}^{1},z_{t}^{1})-\sigma(t,X_{t}^{2},y_{t}^{2},z_{t}^{2})\right\vert
^{2}dt\right)  ^{\frac{p}{2}}\right] \\
\leq C_{p}2^{p-1}\left[  L_{2}^{p}\left(  \left(  T^{\prime}\right)
^{\frac{p}{2}}+\left(  T^{\prime}\right)  ^{p}\right)  +L_{3}^{p}\right]
\mathbb{E}\left[  \sup\limits_{0\leq t\leq T^{\prime}}|\hat{y}_{t}%
|^{p}+\left(  \int_{0}^{T^{\prime}}|\hat{z}_{t}|^{2}dt\right)  ^{\frac{p}{2}%
}\right] \\
\leq C_{p}2^{p-1}\left[  L_{2}^{p}\left(  T^{\frac{p}{2}}+T^{p}\right)
+L_{3}^{p}\right]  \mathbb{E}\left[  \sup\limits_{0\leq t\leq T^{\prime}}%
|\hat{y}_{t}|^{p}+\left(  \int_{0}^{T^{\prime}}|\hat{z}_{t}|^{2}dt\right)
^{\frac{p}{2}}\right]  .
\end{array}
\label{fbsde-delty}%
\end{equation}
Since $C_{p}2^{p-1}\left[  L_{2}^{p}\left(  T^{\frac{p}{2}}+T^{p}\right)
+L_{3}^{p}\right]  <1$, the operator $\Gamma$ is a contraction mapping and has
a unique fixed point $(Y,Z)$.\ Let $X$ be the solution to (\ref{fbsde-y0})
with respect to the fixed point $(Y,Z)$. Thus, $(X,Y,Z)$ is the unique
solution to (\ref{new-new-new-new-1}). Following the same steps in Theorem 2.2
in \cite{Hu-JX}, we can obtain the estimate.
\end{proof}

\subsection{The proof of Theorem \ref{th-vis-uni}}

\label{subse-pf}

In order to prove this theorem, we need the following lemmas.

\begin{lemma}
\label{le-uni2} Suppose that $\sigma$ is independent of $y$ and $z$ and
Assumption \ref{assum-1} (i) holds. Let $W_{1}$ be a viscosity subsolution and
$W_{2}$ be a viscosity supersolution to (\ref{eq-hjb-yz}). Furthermore, assume
that $W_{1}$ and $W_{2}$ are Lipschitz continuous with respect to $x$. Then
the function $w:=W_{1}-W_{2}$ is a viscosity subsolution to the following
equation
\begin{equation}
\left\{
\begin{array}
[c]{l}%
w_{t}\left(  t,x\right)  +\sup\limits_{u\in U}\left\{  \frac{1}{2}%
\mathrm{tr}[\sigma\sigma^{\intercal}(t,x,u)D^{2}w(t,x)]+C\left(  1+\left\vert
x\right\vert \right)  \left\vert Dw\left(  t,x\right)  \right\vert \right. \\
\text{\ \ \ \ \ \ \ \ \ \ \ \ \ \ \ \ \ }\left.  +C\left\vert w\left(
t,x\right)  \right\vert \right\}  =0,\\
w\left(  T,x\right)  =0,\text{ }\left(  t,x\right)  \in\lbrack0,T)\times
\mathbb{R}^{n},
\end{array}
\right.  \label{eq-uni1}%
\end{equation}
where $C$ is a constant depending only on the Lipschitz constants of $b$,
$\sigma$, $g$, $W_{1}$ and $W_{2}$.
\end{lemma}

\begin{proof}
Let $\varphi\in C_{b}^{2,3}\left(  \left[  0,T\right]  \times\mathbb{R}%
^{n}\right)  $ and let $\left(  t_{0},x_{0}\right)  \in(0,T)\times
\mathbb{R}^{n}$ be a global maximum point of $W_{1}-W_{2}-\varphi$
with$\ w(t_{0},x_{0})=\varphi(t_{0},x_{0})$. Define the function%
\[
\psi_{\epsilon,\alpha}\left(  t,x,s,y\right)  =W_{1}\left(  t,x\right)
-W_{2}\left(  s,y\right)  -\frac{\left\vert x-y\right\vert ^{2}}{\epsilon^{2}%
}-\frac{\left\vert t-s\right\vert ^{2}}{\alpha^{2}}-\varphi\left(  t,x\right)
,
\]
where $\epsilon$, $\alpha$ are positive parameters which are devoted to tend
to zero. By Lemma 3.7 in \cite{Baeles-BP}, there exists a sequence $\left(
t_{\epsilon,\alpha},s_{\epsilon,\alpha},x_{\epsilon,\alpha},y_{\epsilon
,\alpha}\right)  $ such that

(i) $\left(  t_{\epsilon,\alpha},x_{\epsilon,\alpha},s_{\epsilon,\alpha
},y_{\epsilon,\alpha}\right)  $ is a global maximum point of $\psi
_{\epsilon,\alpha}$ in $\left[  0,T\right]  \times\bar{B}_{R}\times\left[
0,T\right]  \times\bar{B}_{R}$, where $\bar{B}_{R}$ is a ball with a large
radius $R$;

(ii) $\left(  t_{\epsilon,\alpha},x_{\epsilon,\alpha}\right)  $, $\left(
s_{\epsilon,\alpha},y_{\epsilon,\alpha}\right)  \rightarrow\left(  t_{0}%
,x_{0}\right)  $ as $\left(  \epsilon,\alpha\right)  \rightarrow0$;

(iii) $\frac{\left\vert x_{\epsilon,\alpha}-y_{\epsilon,\alpha}\right\vert
^{2}}{\epsilon^{2}}$ and $\frac{\left\vert t_{\epsilon,\alpha}-s_{\epsilon
,\alpha}\right\vert ^{2}}{\alpha^{2}}$ are bounded and tend to zero when
$\left(  \epsilon,\alpha\right)  \rightarrow0$;

(iv) there exist $X$, $Y\in\mathbb{S}^{n}$ such that
\[%
\begin{array}
[c]{c}%
( \frac{2\left(  t_{\epsilon,\alpha}-s_{\epsilon,\alpha}\right)  }{\alpha^{2}%
}+\varphi_{t}\left(  t_{\epsilon,\alpha},x_{\epsilon,\alpha}\right)
,\frac{2\left(  x_{\epsilon,\alpha}-y_{\epsilon,\alpha}\right)  }{\epsilon
^{2}}+D\varphi\left(  t_{\epsilon,\alpha},x_{\epsilon,\alpha}\right)  ,X)
\in\bar{D}^{2,+}W_{1}( t_{\epsilon,\alpha},x_{\epsilon,\alpha}) ,\\
\left(  \frac{2\left(  t_{\epsilon,\alpha}-s_{\epsilon,\alpha}\right)
}{\alpha^{2}},\frac{2\left(  x_{\epsilon,\alpha}-y_{\epsilon,\alpha}\right)
}{\epsilon^{2}},Y\right)  \in\bar{D}^{2,-}W_{2}\left(  s_{\epsilon,\alpha
},y_{\epsilon,\alpha}\right)  ,\\
\left(
\begin{array}
[c]{cc}%
X & 0\\
0 & -Y
\end{array}
\right)  \leq\frac{4}{\epsilon^{2}}\left(
\begin{array}
[c]{cc}%
I & -I\\
-I & I
\end{array}
\right)  +\left(
\begin{array}
[c]{cc}%
D^{2}\varphi\left(  t_{\epsilon,\alpha},x_{\epsilon,\alpha}\right)  & 0\\
0 & 0
\end{array}
\right)  ,
\end{array}
\]
where $\bar{D}^{2,+}$ and $\bar{D}^{2,-}$ can be found in
\cite{Crandall-lecture}. Since $W_{1}$ and $W_{2}$ are sub and supersolution
to (\ref{eq-hjb-yz}) respectively, by (iv) we have
\begin{equation}%
\begin{array}
[c]{rl}
& \inf\limits_{u\in U}H(t_{\epsilon,\alpha},x_{\epsilon,\alpha},W_{1}%
(t_{\epsilon,\alpha},x_{\epsilon,\alpha}),\frac{2\left(  x_{\epsilon,\alpha
}-y_{\epsilon,\alpha}\right)  }{\epsilon^{2}}+D\varphi\left(  t_{\epsilon
,\alpha},x_{\epsilon,\alpha}\right)  ,X,u)\\
& +\frac{2\left(  t_{\epsilon,\alpha}-s_{\epsilon,\alpha}\right)  }{\alpha
^{2}}+\varphi_{t}\left(  t_{\epsilon,\alpha},x_{\epsilon,\alpha}\right)
\geq0,
\end{array}
\label{w1-ineq}%
\end{equation}%
\begin{equation}
\frac{2\left(  t_{\epsilon,\alpha}-s_{\epsilon,\alpha}\right)  }{\alpha^{2}%
}+\inf\limits_{u\in U}H(s_{\epsilon,\alpha},y_{\epsilon,\alpha},W_{2}%
(s_{\epsilon,\alpha},y_{\epsilon,\alpha}),\frac{2\left(  x_{\epsilon,\alpha
}-y_{\epsilon,\alpha}\right)  }{\epsilon^{2}},Y,u)\leq0. \label{w2-ineq}%
\end{equation}
It follows from (\ref{w1-ineq}) and (\ref{w2-ineq}) that%
\begin{equation}%
\begin{array}
[c]{l}%
\varphi_{t}\left(  t_{\epsilon,\alpha},x_{\epsilon,\alpha}\right)
+\sup\limits_{u\in U}\{H(t_{\epsilon,\alpha},x_{\epsilon,\alpha}%
,W_{1}(t_{\epsilon,\alpha},x_{\epsilon,\alpha}),\frac{2\left(  x_{\epsilon
,\alpha}-y_{\epsilon,\alpha}\right)  }{\epsilon^{2}}+D\varphi\left(
t_{\epsilon,\alpha},x_{\epsilon,\alpha}\right)  ,X,u)\\
-H(s_{\epsilon,\alpha},y_{\epsilon,\alpha},W_{2}(s_{\epsilon,\alpha
},y_{\epsilon,\alpha}),\frac{2\left(  x_{\epsilon,\alpha}-y_{\epsilon,\alpha
}\right)  }{\epsilon^{2}},Y,u)\}\geq0.
\end{array}
\label{new-1234569}%
\end{equation}
By (iv) and Lipschitz continuity of $\sigma$, $b$ and $W_{i}$, then%
\begin{equation}%
\begin{array}
[c]{l}%
\mathrm{tr}[\sigma\sigma^{\intercal}(t_{\epsilon,\alpha},x_{\epsilon,\alpha
},u)X]-\mathrm{tr}[\sigma\sigma^{\intercal}(s_{\epsilon,\alpha},y_{\epsilon
,\alpha},u)Y]\\
\leq\frac{4}{\epsilon^{2}}\mathrm{tr}[(\sigma(t_{\epsilon,\alpha}%
,x_{\epsilon,\alpha},u)-\sigma(s_{\epsilon,\alpha},y_{\epsilon,\alpha
},u))(\sigma(t_{\epsilon,\alpha},x_{\epsilon,\alpha},u)-\sigma(s_{\epsilon
,\alpha},y_{\epsilon,\alpha},u))^{^{\intercal}}]\\
\text{ \ }+\mathrm{tr}[\sigma\sigma^{\intercal}(t_{\epsilon,\alpha
},x_{\epsilon,\alpha},u)D^{2}\varphi\left(  t_{\epsilon,\alpha},x_{\epsilon
,\alpha}\right)  ]\\
\leq\frac{4}{\epsilon^{2}}\rho_{\epsilon}(|t_{\epsilon,\alpha}-s_{\epsilon
,\alpha}|)+C\frac{\left\vert x_{\epsilon,\alpha}-y_{\epsilon,\alpha
}\right\vert ^{2}}{\epsilon^{2}}+\mathrm{tr}[\sigma\sigma^{\intercal
}(t_{\epsilon,\alpha},x_{\epsilon,\alpha},u)D^{2}\varphi\left(  t_{\epsilon
,\alpha},x_{\epsilon,\alpha}\right)  ],
\end{array}
\label{new-1234567}%
\end{equation}%
\[
\left\vert \frac{2\left(  x_{\epsilon,\alpha}-y_{\epsilon,\alpha}\right)
}{\epsilon^{2}}+D\varphi\left(  t_{\epsilon,\alpha},x_{\epsilon,\alpha
}\right)  \right\vert \leq L_{W_{1}}\text{, }\left\vert \frac{2\left(
x_{\epsilon,\alpha}-y_{\epsilon,\alpha}\right)  }{\epsilon^{2}}\right\vert
\leq L_{W_{2}}%
\]
and%
\[%
\begin{array}
[c]{l}%
\left(  \frac{2\left(  x_{\epsilon,\alpha}-y_{\epsilon,\alpha}\right)
}{\epsilon^{2}}+D\varphi\left(  t_{\epsilon,\alpha},x_{\epsilon,\alpha
}\right)  \right)  ^{\intercal}b\left(  t_{\epsilon,\alpha},x_{\epsilon
,\alpha},W_{1}(t_{\epsilon,\alpha},x_{\epsilon,\alpha}),\right. \\
\text{\ \ \ \ \ \ \ \ \ \ \ \ \ }\left(  \frac{2\left(  x_{\epsilon,\alpha
}-y_{\epsilon,\alpha}\right)  }{\epsilon^{2}}+D\varphi\left(  t_{\epsilon
,\alpha},x_{\epsilon,\alpha}\right)  \right)  ^{\intercal}\sigma
(t_{\epsilon,\alpha},x_{\epsilon,\alpha},u),u)\\
\text{ \ }-\left(  \frac{2\left(  x_{\epsilon,\alpha}-y_{\epsilon,\alpha
}\right)  }{\epsilon^{2}}\right)  ^{\intercal}b(s_{\epsilon,\alpha
},y_{\epsilon,\alpha},W_{2}(s_{\epsilon,\alpha},y_{\epsilon,\alpha}),\left(
\frac{2\left(  x_{\epsilon,\alpha}-y_{\epsilon,\alpha}\right)  }{\epsilon^{2}%
}\right)  ^{\intercal}\sigma(s_{\epsilon,\alpha},y_{\epsilon,\alpha},u),u)\\
\leq C\left(  \left\vert \frac{2\left(  x_{\epsilon,\alpha}-y_{\epsilon
,\alpha}\right)  }{\epsilon^{2}}\right\vert \rho_{\epsilon}(|t_{\epsilon
,\alpha}-s_{\epsilon,\alpha}|)+\left\vert w\left(  t_{\epsilon,\alpha
},x_{\epsilon,\alpha}\right)  \right\vert \right. \\
\ \ \ \ \ \ \left.  +\left(  1+\left\vert x_{\epsilon,\alpha}\right\vert
\right)  \left\vert D\varphi\left(  t_{\epsilon,\alpha},x_{\epsilon,\alpha
}\right)  \right\vert +\left\vert x_{\epsilon,\alpha}-y_{\epsilon,\alpha
}\right\vert +\frac{\left\vert x_{\epsilon,\alpha}-y_{\epsilon,\alpha
}\right\vert ^{2}}{\epsilon^{2}}\right)  ,
\end{array}
\]
where $L_{W_{i}}$ is the Lipschitz constant for $W_{i}$ and $\rho_{\epsilon
}(s)\rightarrow0$ as $s\rightarrow0^{+}$ for fixed $\epsilon$. We can do the
same analysis for $g$. For (\ref{new-1234569}) we first let $\alpha
\rightarrow0$. And then let $\epsilon\rightarrow0$, we have
\[%
\begin{array}
[c]{rl}
& \sup\limits_{u\in U}\left\{  \frac{1}{2}\mathrm{tr}[\sigma\sigma^{\intercal
}(t_{0},x_{0},u)D^{2}\varphi(t_{0},x_{0})]+C\left(  1+\left\vert
x_{0}\right\vert \right)  \left\vert D\varphi\left(  t_{0},x_{0}\right)
\right\vert +C\left\vert w\left(  t_{0},x_{0}\right)  \right\vert \right\} \\
& +\varphi_{t}\left(  t_{0},x_{0}\right)  \geq0.
\end{array}
\]
Therefore, $w$ is a subsolution to (\ref{eq-uni1}).
\end{proof}

\begin{remark}
\label{re-appen}When $\sigma$ depends on $y$, the right hand side of
(\ref{new-1234567}) will include the term%
\[
\frac{C}{\epsilon^{2}}|W_{1}(t_{\epsilon,\alpha},x_{\epsilon,\alpha}%
)-W_{2}(t_{\epsilon,\alpha},x_{\epsilon,\alpha})|^{2},
\]
which tends to $\infty$ as $\epsilon\rightarrow0$. Thus the above method does
not work.
\end{remark}

Set $\psi\left(  x\right)  =\left[  \log\left(  \left(  \left\vert
x\right\vert ^{2}+1\right)  ^{\frac{1}{2}}\right)  \right]  ^{2}$,
$x\in\mathbb{R}^{n}$.

\begin{lemma}
\label{le-uni3} Suppose Assumption \ref{assum-1} (i) holds. Then, for any
$A>0$, there exists a constant $C_{1}>0$ such that the function
\[
\chi\left(  t,x\right)  =\exp\left\{  \left(  C_{1}\left(  T-t\right)
+A\right)  \psi(x)\right\}
\]
satisfies
\[
\chi_{t}\left(  t,x\right)  +\sup\limits_{u\in U}\left\{  \frac{1}%
{2}\mathrm{tr}[\sigma\sigma^{\intercal}(t,x,u)D^{2}\chi(t,x)]+C\left(
1+\left\vert x\right\vert \right)  \left\vert D\chi\left(  t,x\right)
\right\vert +C\chi\left(  t,x\right)  \right\}  <0
\]
in $\left[  t_{1},T\right]  \times\mathbb{R}^{n}$, where $t_{1}=T-\frac
{A}{C_{1}}$.
\end{lemma}

It is easy to verify this lemma directly. So we omit the proof.

\textbf{Proof of Theorem \ref{th-vis-uni} }We only need to prove that for any
$\alpha>0$, $w$ satisfies%
\[
\left\vert w\left(  t,x\right)  \right\vert \leq\alpha\chi\left(  t,x\right)
\text{, in }\left[  0,T\right]  \times\mathbb{R}^{n}\text{.}%
\]
It is clear that for some $A>0$,
\[
\lim_{\left\vert x\right\vert \rightarrow\infty}\left\vert w\left(
t,x\right)  \right\vert \exp\left(  -A\left[  \log\left(  \left(  \left\vert
x\right\vert ^{2}+1\right)  ^{\frac{1}{2}}\right)  \right]  ^{2}\right)  =0
\]
uniformly for $t\in\left[  0,T\right]  $. This implies that $\left\vert
w\right\vert -\alpha\chi$ is bounded from above in $\left[  t_{1},T\right]
\times\mathbb{R}^{n}$ and that
\[
M:=\max_{\left[  t_{1},T\right]  \times\mathbb{R}^{n}}\left(  \left\vert
w\left(  t,x\right)  \right\vert -\alpha\chi\left(  t,x\right)  \right)
\exp\left(  -C\left(  T-t\right)  \right)
\]
is achieved at some point $\left(  t_{0},x_{0}\right)  $. Without loss of
generality, we assume that\newline$|w( t_{0},x_{0})| >0$ and $w\left(
t_{0},x_{0}\right)  >0$.

Note that
\[
w\left(  t,x\right)  -\alpha\chi\left(  t,x\right)  \leq\left(  w\left(
t_{0},x_{0}\right)  -\alpha\chi\left(  t_{0},x_{0}\right)  \right)
\exp\left(  -C\left(  t-t_{0}\right)  \right)  .
\]
Then, $\left(  t_{0},x_{0}\right)  $ can be seen as a global maximum point for
$w(t,x)-h\left(  t,x\right)  $ where
\[
h\left(  t,x\right)  =\alpha\chi\left(  t,x\right)  +\left(  w\left(
t_{0},x_{0}\right)  -\alpha\chi\left(  t_{0},x_{0}\right)  \right)
\exp\left(  -C\left(  t-t_{0}\right)  \right)  .
\]

Since $w$ is a viscosity subsolution to (\ref{eq-hjb-yz}), if $t_{0}\in\lbrack
t_{1},T)$, then we have
\[%
\begin{array}
[c]{rl}
& \sup\limits_{u\in U}\left\{  \frac{1}{2}\mathrm{tr}[\sigma\sigma^{\intercal
}(t_{0},x_{0},u)D^{2}h(t_{0},x_{0})]+C\left(  1+\left\vert x_{0}\right\vert
\right)  \left\vert Dh\left(  t_{0},x_{0}\right)  \right\vert +Cw\left(
t_{0},x_{0}\right)  \right\} \\
& +h_{t}\left(  t_{0},x_{0}\right)  \geq0.
\end{array}
\]
That is
\[%
\begin{array}
[c]{rl}
& \sup\limits_{u\in U}\left\{  \frac{1}{2}\mathrm{tr}[\sigma\sigma^{\intercal
}(t_{0},x_{0},u)D^{2}\chi(t_{0},x_{0})]+C\left(  1+\left\vert x_{0}\right\vert
\right)  \left\vert D\chi\left(  t_{0},x_{0}\right)  \right\vert +C\chi\left(
t_{0},x_{0}\right)  \right\} \\
& +\chi_{t}\left(  t_{0},x_{0}\right)  \geq0.
\end{array}
\]
It is a contradiction to Lemma \ref{le-uni3}. Therefore $t_{0}=T$. Since
$\left\vert w\left(  T,x\right)  \right\vert =0$, we obtain%
\[
\left\vert w\left(  t,x\right)  \right\vert \leq\alpha\chi\left(  t,x\right)
\text{, in }\left[  t_{1},T\right]  \times\mathbb{R}^{n}\text{.}%
\]
Thus, let $\alpha\rightarrow0$., we can obtain $\left\vert w\right\vert =0$ in
$\left[  t_{1},T\right]  \times\mathbb{R}^{n}$. Applying successively the same
argument on the interval $\left[  t_{2},t_{1}\right]  $ where $t_{2}=\left(
t_{1}-A/C_{1}\right)  ^{+}$ and then, if $t_{2}>0$ on $\left[  t_{3}%
,t_{2}\right]  $ where $t_{3}=\left(  t_{2}-A/C_{1}\right)  ^{+}$...etc..We
finally obtain that $\left\vert w\right\vert =0$ in $\left[  0,T\right]
\times\mathbb{R}^{n}$. This completes the proof. $\blacksquare$

\end{document}